\documentclass[reqno,11pt]{amsart}
\usepackage{amsthm,amsfonts,amssymb,euscript,   
mathrsfs,graphics,color,amsmath,latexsym,marginnote} 
\usepackage{cite}      
\usepackage{dsfont}       
\usepackage[english]{babel}      
\usepackage{mathtools}
\mathtoolsset{showonlyrefs} 
\usepackage{lmodern} % math, rm, ss, tt 
\usepackage{fix-cm}

\numberwithin{equation}{section}

\usepackage{datetime}
\usepackage{lipsum}
\usepackage{enumerate}
\usepackage{tikz}
\usetikzlibrary{matrix}
\usepackage[pdftex]{hyperref}
\RequirePackage{amscd}
\RequirePackage{epic}
\RequirePackage[all]{xy}
\RequirePackage{url}
\RequirePackage[shortlabels]{enumitem}
\usepackage{empheq}
\usepackage{epsfig}
\usepackage{lineno} % per numerare le righe 
 \usepackage{perpage}
 \usepackage[english]{babel}

\usepackage[utf8]{inputenc}
\usepackage{cite}
\RequirePackage{amscd}
\RequirePackage{epic}
\RequirePackage{eepic}
\RequirePackage[all]{xy}
\RequirePackage{url}
\RequirePackage[shortlabels]{enumitem}
\usepackage{empheq}
\usepackage{nomencl}
\usepackage{epsfig}
\usepackage{graphicx}
\theoremstyle{plain}

\newtheorem{theorem}{Theorem}[section]
\newtheorem{proposition}[theorem]{Proposition}
\newtheorem{lemma}[theorem]{Lemma}

\newtheorem{remark}[theorem]{Remark}

\newtheorem{definition}[theorem]{Definition}

\renewcommand{\Re}{\mathrm{Re}\,}
\renewcommand{\Im}{\mathrm{Im}\,}

\providecommand{\vect}[2]{{\bigl[\begin{smallmatrix}#1\\#2\end{smallmatrix}\bigr]}}   
\providecommand{\sm}[4]{{\bigl[\begin{smallmatrix}#1&#2\\#3&#4\end{smallmatrix}\bigr]}}
    \newcommand{\set}[1]{{\left\{#1\right\}}}
\newcommand{\pa}[1]{{\left(#1\right)}}
\newcommand{\norm}[1]{{\left |#1\right |}}
\RequirePackage{epic}
\RequirePackage{eepic}
\RequirePackage[all]{xy}
\RequirePackage{url}
\RequirePackage[shortlabels]{enumitem}
\usepackage{empheq}
\usepackage{nomencl}
\usepackage{epsfig}
\usepackage{graphicx}

\newcommand{\T}{\mathbb{T}}
\newcommand{\Z}{\mathbb{Z}}
\newcommand{\R}{\mathbb{R}}
\newcommand{\C}{\mathbb{C}}

\newcommand{\eps}{\varepsilon}

\renewcommand{\Re}{\operatorname{Re}}
\renewcommand{\Im}{\operatorname{Im}}

\newcommand{\na}{\widehat{n}}

\newcommand{\co}[1]{\textit{#1}}%corsivo
\newcommand{\gr}[1]{\textbf{#1}}%grassetto

\newcommand{\id}{\operatorname{id}}

\usepackage{amsthm}

\newcommand{\g}{\gamma}

\newcommand{\s}{{\sigma}}

\newcommand{\ii}{{\rm i}}
\def\wc{ {}}

\def\norma#1{\left \|#1\right \|}

\newcommand{\x}{\xi}

\newcommand{\N}{{\mathbb N}}

\newcommand{\cC}{{\mathcal C}}
\newcommand{\cD}{{\mathcal D}}

\newcommand{\cF}{{\mathcal F}}

\newcommand{\cH}{{\mathcal H}}

\newcommand{\cK}{{\mathcal K}}

\newcommand{\cM}{{\mathcal M}}

\newcommand{\cR}{{\mathcal R}}

\newcommand{\tb}{{\mathtt{b}}}

\newcommand{\td}{{\mathtt{d}}}

\newcommand{\tf}{{\mathtt{f}}}

\newcommand{\tm}{{\mathtt{m}}}

\newcommand{\tr}{{\mathtt{r}}}

\newcommand{\tC}{{\mathtt{C}}}

\newcommand{\tH}{{\mathtt{H}}}
\newcommand{\tK}{{\mathtt{K}}}

\newcommand{\tN}{{\mathtt{N}}}

\usepackage{bm}

\newcommand{\al}{{\alpha}}
\newcommand{\bt}{{\beta}}

\newcommand{\abs}[1]{\left| #1 \right|}

\newcommand{\jap}[1]{\langle #1 \rangle}
\newcommand{\und}[1]{\underline{#1}}

\newcommand{\e}{{\varepsilon}}

\newcommand{\jml}[1]{\lfloor #1 \rfloor}

\newcommand{\uno}{{\mathbb I}}

\newcommand{\meas}{{\mathtt{meas}}}

\newcommand{\tw}{{\mathtt{w}}}
\renewcommand{\th}{{\mathtt{h}}}
\newcommand{\nnorm}[1]{{\left\vert\kern-0.25ex\left\vert\kern-0.25ex\left\vert #1 
    \right\vert\kern-0.25ex\right\vert\kern-0.25ex\right\vert}}
\newcommand{\es}{e^{L_{S} }}

\usepackage{framed,enumitem} 

\newcommand{\suca}{\mathtt N}
\newcommand{\ri}{r}

\newcommand{\rf}{{r'}}

\newcommand{\twi}{\tw}

\newcommand{\imm}{{\rm{i}}}

\newcommand{\CalgM}{{C_{\mathtt{alg},\mathtt M}(p)}}

\newcommand{\jjap}[1]{\lfloor #1 \rfloor}

\definecolor{aqua}{RGB}{20,100,200}

\newcommand{\ov}{\overline}

\makeatletter
\renewcommand{\@setsubjclass}{%
  \ifx\@empty\@subjclass\else
    \begingroup
      \textit{2020 Mathematics Subject Classification. }
      \@subjclass
    \endgroup
  \fi}
\makeatother

\makeatletter
\def\l@subsection{\@tocline{2}{0pt}{2.5pc}{5pc}{}}
\def\l@subsubsection{\@tocline{3}{0pt}{4.5pc}{5pc}{}}
\begin{document}

\title[On the control of high Sobolev norms for the Wave equation on $\T^d$]{On the control of high Sobolev norms for the Wave equation on $\T^d$ over exponentially long times}

%Growth of Sobolev norms and exponential-type stability for the wave equation}}

\date{}

\author{Roberto Feola}
\address{\scriptsize{Dipartimento di Matematica e Fisica, Universit\`a degli Studi Roma Tre, Largo San Leonardo Murialdo 1, 00146}}
\email{roberto.feola@uniroma3.it}

\author{Jessica Elisa Massetti}
\address{\scriptsize{Dipartimento di Matematica, Universit\`a degli Studi di Roma ``Tor Vergata'', Via della Ricerca Scientifica 1, 00133}}
\email{massetti@mat.uniroma2.it}

%\thanks{thanks}  

\begin{abstract} 
We consider a one-parameter family of nonlinear wave equations on the 
$d$-dimensional torus, with polynomial nonlinearities of arbitrary degree $q+1$, 
where $q\geq 1$. 
We investigate the long-time behavior of high Sobolev $H^s$-norms 
of solutions in different settings.     
In the one-dimensional case, and for almost any value of the mass parameter $\mathtt{m}>0$, 
we prove exponentially long stability times   for small initial data. 
The proof relies on normal form techniques together 
with suitable \emph{weak} Diophantine conditions.
In higher space dimensions, for initial data $u_0\in H^{s}$, $s \geq s_1 + 1$, satisfying 
suitable smallness conditions on the \emph{low} Sobolev norm $H^{s_1}$ and 
on the $L^2$-norm, we prove a polynomial upper bound on the possible growth 
of the high Sobolev $H^{s}$-norm, 
over finite but exponentially long time scales in the regularity parameter $s_1$.
 The key ingredient consists 
 in establishing suitable \emph{a priori} tame estimates for the solution. 
The result applies in \emph{any} space dimension $d\geq 1$ 
and for \emph{all} values of the mass parameter $\mathtt{m}\geq 0$.
\end{abstract}

 \keywords{Wave equations, exponential-type stability, lifespan for semi-linear PDEs, small divisors}  
\subjclass{35L05, 35B45, 37K45, 37J40}
  
\maketitle
\tableofcontents

\section{Introduction}

Let $\T^d :=(\mathbb{R}/2\pi \mathbb{Z})^{d}$, $d\ge 1$, $q\geq1$. 
In this paper we consider  the wave type equations
\begin{equation}\label{eq:wave}
 \partial_{tt}\psi-\Delta \psi+\mathtt{m}\psi+\psi^{q+1}=0\,,
 \qquad 
 x\in \mathbb{T}^{d}:=(\R/2\pi\Z)^{d}
\end{equation}
where 
 $t\in\mathbb{R}$, $\mathtt{m}\geq0$,
 and $\Delta$ denotes the Laplacian operator acting 
 on the Fourier basis of periodic functions by linearity as
 \[
 \Delta e^{{\rm i} j \cdot x}=-|j|^{2}e^{{\rm i} j\cdot x}\,, 
 \qquad 
 |j| := \|j\|_{\ell_2}\,,\quad \forall\, j\in \mathbb{Z}^{d}\,.
 \]

Because of the periodicity in the space variable, we shall identify functions to their Fourier expansion and work on the Sobolev phase space with the usual norm on the sequences of Fourier's coefficients of involved functions. More specifically, for any $s\in\R$,
 we consider  the standard Sobolev space $H^{s}(\mathbb{T}^{d};\C)$
 defined as
 \begin{equation}\label{spaziodiSobolev}
 H^{s}:=H^{s}(\T^{d};\C):=
 \Big\{
 u(x)=\sum_{j\in\Z^{d}}u_{j}e^{\ii jx}\;:\;
 \|u\|_{H^{s}}^2:=\sum_{j\in\Z^{d}}\langle j\rangle^{2s}|u_{j}|^2<+\infty
 \Big\}\,,
 \end{equation}
 where $\langle j\rangle:=\max\{1,|j|\}$ for any $j\in\Z^{d}$.
 The general purpose of the present paper is to investigate the \emph{long time} evolution of 
 \emph{small amplitude} initial data, with particular emphasis on the 
 lifespan and the possible growth of their Sobolev norms. More precisely, 
 given $1\ll s_1\leq s$ and  a solution $u(t)$ with initial datum 
 $u_0$ with 
 $\|u_0\|_{H^{s_1}}\sim \delta$,
we prove that its \emph{low} Sobolev norm $\|u(t)\|_{H^{s_1}}$ remains of size $\delta$
 over a very long time scale $[0,T_{\delta}]$. Moreover, we establish upper bounds on the possible growth
of the
\emph{high} Sobolev norms $\|u(t)\|_{H^{s}}$ over the same time scale.
A central point is to give a lower 
 bound on $T_{\delta}$ strictly greater than the natural time $O(\delta^{-q})$ 
 of existence and stability 
 given by classical local existence theory. In this paper we actually prove that 
 $T_{\delta}$ is exponentially long with respect to $1/\delta$.
We address this problem for equation \eqref{eq:wave} in several different settings.
 
 (i) We first consider the dispersive case $\mathtt{m}>0$, where the spatial variable $x$ belongs to the one-dimensional torus $\T$. In this case, we consider  \eqref{eq:wave} 
 as a \emph{one parameter}  
family of PDEs, with the mass $\mathtt{m}$ regarded  as a free parameter 
to be chosen in a 
suitable way 
in order to obtain the exponentially long time stability of the solutions. 

(ii) We next consider the case in which $\mathtt{m}>0$ 
is treated as a \emph{fixed} parameter, 
and the spatial variable $x$ belongs to the $d$-dimensional torus $\T^d$, 
for arbitrary dimension $d>1$.

(iii) We finally analyze the \emph{completely resonant case}, 
corresponding to the wave equation, where the parameter 
$\mathtt{m}$ vanishes, namely $\mathtt{m}=0$.

 \medskip
 Of course the choice of the manifold where  the space variable lives, its dimension and 
  the nonlinearity at stake, play a central role in describing the dynamics for long time.
  
In fact the Cauchy problem associated to \eqref{eq:wave}
has been widely studied on the \emph{Euclidean space}, where dispersive properties of the linear flow
may lead even to global existence results. 
We mention for instance \cite{shatah1982, Sideris1983, Georgiev2010} 
%\cite{} \red{Shatah, gieorgiev, sideris , forse ikeda}
 and reference therein.
 We stress the fact that the choice of the nonlinearity plays a key role also in the dispersive case
 leading to possible ill-posedness of the problem, 
 as it is  shown in the counterexample by Klainerman-Majda in \cite{KlainermanMadja}.

\smallskip
On compact manifolds instead, there is no dispersive effect that could help in the control of Sobolev norms on large time scales (for instance the decay in time of solutions).   On the other hand, the dynamics is governed by the interaction among linear and nonlinear frequencies of oscillations (i.e. the spectrum of the linearized operator) and possible resonant effects.

One of the main feature of the present article is that 
we intend to propose a method that enables one to tackle the problem 
of stability/possible growth of norms of solutions of equation \eqref{eq:wave} 
in a uniform way with respect to the dimension and to the mass parameter. 
This calls for a comment. 
Of course, the one dimensional Klein Gordon (KG) equation (i.e. $\mathtt{m}\neq 0$) 
is very special because it enjoys the possibility of being attacked 
through dynamical systems techniques coming from perturbations theory, 
by means of a Birkhoff Normal Form  procedure (BNF in the following). 
Inspired by that, in this peculiar case, we then considered to be suitable to combine 
a BNF together with energy-like methods, 
in order to get both a longtime stability result \emph{plus} 
the information about the possible growth of
a whole scale of norms beyond the "stable one". 
In order to give a systematic and clear presentation 
of the results and related comments, 
we shall in this introduction distinguish between 
the one dimensional case with mass, and 
the higher dimensional case no matter what the mass is.

\subsection{Main results}
We now present the main results of the paper, 
starting with the one-dimensional case and then turning to the higher-dimensional case.

\vspace{0.5em}
\noindent
{\bf One dimensional case with $\mathtt{m}\neq0$: Birkhoff normal form approach.}
A fruitful approach to this issue comes from normal forms \`a la Birkhoff, 
where the general idea is to take advantage of the Hamiltonian 
nature of the equation and iteratively perform suitable 
change of coordinates that aims at removing the non-resonant terms from the 
nonlinearity up to the highest possible order.  
A crucial point in the implementation of this procedure 
is the derivation of sharp quantitative lower bounds for the 
so-called \emph{small divisors} that naturally arise in the construction
and that come in this form
\begin{equation}\label{ondina}
\omega_{j_1} \pm \omega_{j_2} \pm \cdots \pm \omega_{j_N}\,,
\qquad j_1,\ldots, j_N\in\Z^d\,, N\in\N\,,
\end{equation}
where the $\omega_{j_i}$ are the linear frequencies of oscillations mentioned above. 
The number $N$ of interactions grows with the steps of iteration.
To provide quantitative bounds from below for the above quantities for any $N$ 
is a hard issue and requires the presence of parameters 
in order to modulate the frequencies and 
avoid exact zeros of equation \eqref{ondina}. 
For the Klein-Gordon equation the frequencies are given 
by $\omega_j = \sqrt{|j|^2 + \tm}$ and the ''mass'' parameter allows to 
escape from exact zeros of the small divisors. Nevertheless, 
such quantities could still accumulate to zero very fast and 
create dangerous loss of regularity along the Birkhoff procedure.
Up to eventually reordering \eqref{ondina}, in order 
to overcome small divisors one typically imposes \emph{Diophantine conditions}
of the following  form:
\begin{equation}\label{seconda}
\left|\sum_{i=1}^{\ell}\omega_{j_i}-\sum_{k=\ell+1}^{N}\omega_{j_k}\right|
\geq (\max_{3}\left\{|j_1|,...,|j_N|\right\})^{-\tau} \gamma\,,
\end{equation}
\emph{except} in the case
\begin{equation}\label{nonresForte}
N\;\;{\rm even},\;\; \ell=\frac{N}{2}\;\;\;{\rm and\,(up\,to\,permutations)}\;\;\;
j_i= j_{i+\tfrac{N}{2}}\,,
\end{equation}
for some fixed $\gamma,\tau>0$ and 
where $ \max_{3}\left\{|j_1|,...,|j_N|\right\} $ denotes the third
largest number among $|j_1|,...,|j_N|$.

The strategy described above has been firstly successfully exploited in  one space dimension.
We refer for instance to  the seminal works \cite{Bambusi:2003} for the Klein-Gordon equation, 
and \cite{Bambusi-Grebert:2006}, where
polynomial bounds for a wide class of {\it tame-modulus} PDEs are proved. 
%More precisely, it is shown that for any $\suca\gg 1$ there exists $p(\suca) $ (tending to infinity as $\suca\to \infty$) such that for all $p\ge p(\suca)$ and all $\delta-$small initial data in $H^p$ one has 
%$T\ge C(\suca,p)\delta^{-\suca}$, provided $\delta<\delta_0(\suca,p)$.
We also refer to
 \cite{DS0},\cite{DS},\cite{BDGS}, and to the more recent results
(where nonlinearities may contain derivatives)
  \cite{Yuan-Zhang},\cite{Delort-2009},\cite{Delort-2015},\cite{Berti-Delort},\cite{FI}. 
%\red{aggiornare la lista o ridurla}.
All the papers mentioned above have in common that the spectrum of
the linearized problem at zero has “good separation” properties and estimates like 
\eqref{seconda} are achievable.
In higher dimension this is no more the case and one is able to prove only 
weaker lower bounds that, a priori, cause losses of derivatives that obstruct the 
Birkhoff normal form procedure.
However, some ``partial result'' (few steps of normal forms and polynomial time of stability 
$O(\delta^{-p})$ with  some fixed $p$) 
on higher dimensional manifolds are still available.
In this line of thoughts, we mention for instance
 \cite{Delort:2009vn}, \cite{FGI20} on the Klein-Gordon equation, \cite{Imekraz2016}, \cite{BFGI2021}
 on the Beam equation in high space dimension,
 \cite{FM2022} on the Schr\"odinger on generic tori, 
 \cite{HIT2016}, \cite{IT2017},
 \cite{ID2019}, \cite{BFP2022}, \cite{BFF2021} on the water waves equation
 (see also \cite{FIM2022} on a different fluid model). 
As a matter of fact, recently in \cite{BambuFeoMont, BambuFeoLanMonz} the authors prove ``almost global" type of results on higher dimensional compact manifolds, i.e. time of stability of order $\delta^{-N}$ for any $N$, under suitable assumptions, 
 for various models of PDEs involving free parameters. 
Note however that the latter results strongly rely on the super-linear asymptotic of the frequencies of oscillations, that is $\omega_{j}\sim |j|^{\alpha}$ with $\alpha>1$, situation that does not occur in the Klein-Gordon/wave equation, where $\omega_{j}\sim|j|$.
For results regarding the stability of \emph{low} norms for the 
Klein-Gordon equation in high dimensions we quote
\cite{BFGrelongwave2020, Xiang2025}.

 Note that all the results mentioned above regard polynomial stability times in Sobolev spaces. 
Passing from polynomial estimates to exponential-type ones 
is not trivial and it is related to the regularity of initial data. Concerning 
exponential type stability times for analytic or Gevrey initial data see
 \cite{Faou-Grebert:2013, Cong, BMP:CMP, FeoMass:Beam,CongLiuWang}  and reference therein.
 While the first result on exponential long time stability in Sobolev regularity has been achieved in 
 \cite{BMP:CMP}, see also \cite{FeoMass:Beam, CongMiShi22}.
Inspired by the optimization result of Biasco-Massetti-Procesi and Feola-Massetti \cite{BMP:CMP,FeoMass:Beam} for Sobolev initial data, in the present article we prove
exponential long time stability for equation \eqref{eq:wave} 
with $\mathtt{m}\neq 0$ on $\mathbb{T}$. In addition to that, we are able to provide a priori 
bounds on the possible growth of higher Sobolev norms.
We shall clarify and discuss this latter point in the next paragraphs. Before going there, we state the following result.

\begin{theorem}{\bf (Exponential stability for Klein-Gordon and growth of high norms on $\T$).}\label{teonuovo}
Let $d=1$ and fix any $0<\gamma<1$. 
There is a positive measure set $\mathfrak{M}_{\gamma}\subset[1,2]$
with $\meas([1,2]\setminus\mathfrak{M}_{\gamma})=O(\gamma)$, 
 an absolute constant $\mathtt{c}>0$ 
such that for any $\mathtt{m}\in \mathfrak{M}_{\gamma}$
the following holds.
Fix 
\[
0< \delta\leq \gamma^{\tb}\,,\qquad \mathtt{b}:=24\mathtt{c}^{2}\big[2^6 (36)^{2}\big]^{5/3}\,,
\]
and let
\begin{equation}\label{caramellina1}
p=p(\delta):=1+\left(\frac{1}{24\mathtt{c}^2\ln(1/\gamma)} \ln\big(\frac{1}{\delta}\big)\right)^{5/9}\,.
\end{equation}
Then, 
for any $s\geq p+1$ and any initial conditions
$(\psi_0,\psi_1)\in H^{s+\frac{1}{2}}(\T^{d};\R)\times H^{s-\frac{1}{2}}(\T^{d};\R)$ such that
\begin{equation}\label{main:smallcondSob}
\|\psi_0\|_{H^{p+\frac{1}{2}}}+\|\psi_1\|_{H^{p-\frac{1}{2}}}
+2^{p}(\|\psi_0\|_{H^{\frac{1}{2}}}+\|\psi_1\|_{H^{-\frac{1}{2}}})\leq \frac{\delta}{4}\,,
\end{equation}
there exists  a time $T=T(\delta)>0$ and a unique solution 
$\psi(t)$ of \eqref{eq:wave} with initial condition 
$\psi(0)=\psi_0$, $\psi_{t}(0)=\psi_1$ such that 
\begin{equation}\label{patata1}
\begin{aligned}
&\psi\in C^{0}\big([0,T]; H^{s+\frac{1}{2}}(\T;\R)\big)\cap 
C^{1}\big([0,T]; H^{s-\frac{1}{2}}(\T;\R)\big)\,,
\\& T\geq T_{good}:=
\frac{1}{\delta}
\exp\Big\{
\frac{\mathtt{c}(\ln(1/\gamma))^{-1/9})}{(24\mathtt{c}^2)^{10/9}}(\ln(1/\delta) )^{1+\frac{1}{9}}
\Big\}\,.
\end{aligned}
\end{equation} 
Moreover one has
\begin{equation}\label{stimaIncrediblelowNORM}
\|\psi(t)\|_{H^{p+\frac{1}{2}}}+\|\partial_t\psi(t)\|_{H^{p-\frac{1}{2}}}+2^{p}(\|\psi(t)\|_{H^{\frac{1}{2}}}
+\|\partial_t\psi(t)\|_{H^{-\frac{1}{2}}})\lesssim \delta\,,
\end{equation}
and
\begin{equation}\label{stimaIncredibleBis}
 \|\psi(t)\|_{H^{s+\frac{1}{2}}} + 2^{s}\|\psi(t)\|_{L^{\frac{1}{2}}}
\lesssim_{s}\delta\,
 \Big[
 1+\big(
C_s t
 \big)^{s-p}
 \Big]\,,
\end{equation}
for any $t\in[0,T]$ and for some $C_s>0$.
\end{theorem}
The key point in getting the result above is to understand precisely the mutual dependence between
the regularity $p$ of the initial data and their size $\delta$, 
which is encoded in \eqref{caramellina1}. The crucial midstep 
to get \eqref{caramellina1}-\eqref{patata1} is the following theorem.

\begin{theorem}[\bf Sobolev stability of low norms]\label{teoBNFwave}
Let
$p>2^{6}(36)^{2}+1$ and fix any $0<\gamma<1$. 
There is a positive measure set $\mathfrak{M}_{\gamma}\subset[1,2]$
with $\meas([1,2]\setminus\mathfrak{M}_{\gamma})=O(\gamma)$, 
 an absolute constant $\mathtt{c}>0$ 
such that for any $\mathtt{m}\in \mathfrak{M}_\gamma$
the following holds.
For any 
\begin{equation}\label{smalldelta1Sob}
0<\delta\leq\gamma^{\mathtt{c} p}\,,
\end{equation}
and any initial datum {$(\psi_0,\psi_1)\in H^{p+1/2}\times H^{p-1/2}$}
satisfying the smallness condition \eqref{main:smallcondSob}
the solution $(\psi(t), \partial_t\psi(t))$ of \eqref{eq:wave} with 
$(\psi(0), \partial_t\psi(0))=(\psi_0,\psi_1)$ exists and satisfies
\begin{equation}\label{boundsol1Sob}
\|\psi(t)\|_{H^{p+\frac{1}{2}}}+\|\partial_t\psi(t)\|_{H^{p-\frac{1}{2}}}
+2^{p}(\|\psi(t)\|_{H^{\frac{1}{2}}}
+\|\partial_t\psi(t)\|_{H^{-\frac{1}{2}}})\leq 8\delta\,, 
\qquad \forall\;  |t|\leq T_0\,,
\end{equation}
with 
\begin{equation}\label{longtime1Sob}
T_0\geq
\frac{\gamma^{cp^2}}{\delta}\left(\frac{1}{\delta}\right)^{\frac{1}{c}(p-1)^{1/5}}\,.
\end{equation}
\end{theorem}

\vspace{0.5em}
\noindent
{\bf Higher dimensional case.} 
As briefly mentioned above, a major point of the present paper is to propose a flexible approach
that allow to deal with  \eqref{eq:wave} 
posed on $\T^{d}$ for \emph{any} dimension $d>1$
and \emph{any} fixed mass parameter 
$\mathtt{m}\geq0$, and that, as a consequence, does  not rely at all on 
normal forms techniques and small divisors bounds.
%
%Therefore, we study the long time behavior for \eqref{eq:wave} 
%for any $d>1$ and any fixed $\mathtt{m}\geq0$ by exploiting 
%the following paradigm.
 In this resonant context (i.e. when no free parameters allow to prove bounds like 
 \eqref{seconda}), we shall introduce a different paradigm.
Inspired by the observation in \cite[Remark $1.4$]{BMP:CMP} 
 (see also \cite{CongMiShi22}),  we shall deeply exploit 
the \emph{tameness} property of the Sobolev 
norm $\|\cdot\|_{H^{s}}$ in \eqref{spaziodiSobolev}.  
%\red{Però la tameness ce l'ha quella col japjap...forse parlare di equivalent norm?}
In fact, this tameness property will be encoded through 
an equivalent Sobolev norm introduced in Section \ref{sec:equivnorm}. 
This choice of norm makes it possible 
to exploit more effectively the rescaling from higher to lower Sobolev indices,
which in turn provides sharper control of the size of the nonlinearity.
As a consequence, for initial data carrying only a 
small amount of energy at low frequencies (for instance, with sufficiently small $L^2$-norm),
 we prove the stability of the $\|\cdot\|_{H^{s_1}}$-norm, $s_1>d/2$, over 
 a time scale which is exponentially long in the regularity parameter $s_1$.
 Furthermore, by exploiting the \emph{smoothing} properties of the nonlinearity in \eqref{eq:wave}
 (in appropriate complex variables)
 and combining sharp tame estimates with an improved Gr\"onwall 
 lemma, we obtain a control from above of the \emph{high} Sobolev norm $\|\cdot\|_{H^{s}}$, $s>s_1$,
proving that it could grow at most polynomially in $t$.
This is the content of the following result.
\begin{theorem}[\bf{Long time estimates for Klein-Gordon on $\T^d$}]\label{thm:mainNOBNF}
Consider equation \eqref{eq:wave} for $d\geq1$, $q\geq1$ and $\mathtt{m}>0$. 
Fix $0<\delta\ll1$, $s_0>d/2$,  consider $s_1\geq s_0$.
There exist   constants $ \mathtt{M}>0$ and $\mathtt{c}=\mathtt{c}_{\tm}$ such that
for any $s\geq s_1+1 $ and any 
$(\psi_0,\phi_0)\in H^{s+\frac{1}{2}}(\T^{d})\times H^{s-\frac{1}{2}}(\T^{d})$ such that
\begin{equation}\label{piccolezza dati trisNOBNF}
\|\psi_0\|_{H^{s_1+1/2}}+\|\phi_0\|_{H^{s_1-1/2}}
+
(2\mathtt{M})^{s_1}\big(\|\psi_0\|_{H^{1/2}}+\|\phi_0\|_{H^{-1/2}}\big)\leq \delta\,,
\end{equation}
%\begin{equation}\label{piccolezza dati trisNOBNF}
%\|\psi_0\|_{H^{s_1+1/2}}+\|\phi_0\|_{H^{s_1-1/2}}\leq \delta\,,
%\qquad 
%(2\mathtt{M})^{s_1}\big(\|\psi_0\|_{H^{1/2}}+\|\phi_0\|_{H^{-1/2}}\big)\leq \delta\,,
%\end{equation}
the following holds.
There exist a time $T=T(\delta,\mathtt{M},s_1)>0$
and a unique solution $\psi(t)$ of \eqref{eq:wave} 
with initial condition $\psi(0)=\psi_0$, $\psi_{t}(0)=\phi_0$ 
such that
\begin{equation}\label{patata1dimd}
\begin{aligned}
&\psi\in C^{0}\big([0,T]; H^{s+1/2}(\T^{d};\C)\big)\cap 
C^{1}\big([0,T]; H^{s-1/2}(\T^{d};\C)\big)\,,
\\& T\geq T_{good}:=\frac{1}{\delta^{q}}\frac{2^{q(s_1-s_0)}}{\mathtt{c}_{\tm}\mathtt{M}^{qs_0}}\,.
\end{aligned}
\end{equation} 
Moreover one has
\begin{equation}\label{stimaIncredibleBisNOBNF}
\|\psi(t)\|_{H^{s+\frac{1}{2}}}+\|\partial_{t}\psi(t)\|_{H^{s-\frac{1}{2}}}
\lesssim_{s}
\big(\|\psi_0\|_{H^{s+\frac{1}{2}}}+\|\phi_0\|_{H^{s-\frac{1}{2}}}\big)
 \Big[
 1+\Big(
 \mathtt{M}^{q(s-s_1)}
 \frac{t}{T_{good}}
% \frac{\mathtt{c}_{\tm}\mathtt{M}^{q s_0}\delta^{q}}{2^{q(s_1-s_0)}}\, t
 \Big)^{s-s_1}
 \Big]\,,
\end{equation}
for any $t\in[0,T]$ and,
as a consequence, the bound
\begin{equation}\label{stimaIncredible}
\sup_{t\in[0,\widehat{T}]}
\big(\|\psi(t)\|_{H^{s+\frac{1}{2}}}+\|\partial_{t}\psi(t)\|_{H^{s-\frac{1}{2}}}\big)
\lesssim_{s}
\big(\|\psi_0\|_{H^{s+\frac{1}{2}}}+\|\phi_0\|_{H^{s-\frac{1}{2}}}\big)
(1+\mathtt{M}^{q(s-s_1)^2}) \,,
\end{equation}
for any $ \widehat{T}\leq T_{good}$.
\end{theorem}

%\vspace{3em}
%
%\noindent
%The general purpose of this work is to to provide results of stability of the solutions of 
%\eqref{eq:wave} over \emph{exponentially long} time scales, 
%both in the case of $\mathtt{m} > 0$ and $\mathtt{m}=0$, and in 
%any space dimension. 
%With that in mind, of course, the $1$-dimensional case with $\mathtt{m}>0$ 
%may benefit of normal form techniques tailored for this specific model. 
%Therefore, in this specific case, we will take full advantage 
%of the dynamical system approach of Birkhoff Normal Form (BNF in the sequel) 
%combined with energy estimates, in order to give \emph{both} 
%a stability result within an exponentially long time $T_{good}$ \emph{and} 
%the possible growth of a scale of norms of the solutions. 

The last result is about the wave equation in the \emph{massless} case 
(i.e. equation  \eqref{eq:wave} with $\mathtt{m}=0$).

\begin{theorem}{\bf (Long time estimates for the wave equation on $\T^d$).}\label{thm:mainNOBNF2}
Consider equation \eqref{eq:wave} for $d\geq1$, $q\geq1$ and $\mathtt{m}=0$. 
Fix $0<\delta\ll1$, $s_0>d/2$,  consider $s_1\geq s_0$.
There exist   constants $ \mathtt{M}>0$ and $\mathtt{c}\geq 1$ such that
for any $s\geq s_1+1 $ and any 
$(\breve{\psi},\breve{v})\in H^{s}(\T^{d})\times H^{s-1}(\T^{d})$ such that
\begin{equation}\label{piccolezza dati trisNOBNF1}
\|\breve{\psi}\|_{H^{s_1}}+\|\breve{v}\|_{H^{s_1-1}}
%\leq \delta\,,
%\qquad 
+(2\mathtt{M})^{s_1}\big(\|\breve{\psi}\|_{L^2}+\|\breve{v}\|_{H^{-1}}\big)\leq \delta\,,
\end{equation}
the following holds.
There exist a time $T=T(\delta,\mathtt{M},s_1)>0$
and a unique solution $\psi(t)$ of \eqref{eq:wave} with $\mathtt{m}=0$ and 
with initial condition $\psi(0)=\breve{\psi}$, $\psi_{t}(0)=\breve{v}$ 
such that
\begin{equation}\label{patata1dimdMASS}
\begin{aligned}
&\psi\in C^{0}\big([0,T]; H^{s}(\T^{d};\C)\big)\cap 
C^{1}\big([0,T]; H^{s-1}(\T^{d};\C)\big)\,,
\\& T\geq T_{good}:=\frac{2^{\frac{1}{2}(s_1-s_0)}}{\mathtt{M}^{qs_0} 2^{q+3}\mathtt{c}^{q}}\,.
%\frac{1}{\delta^{q}}\frac{2^{q(s_1-s_0)}}{\mathtt{c}_{\tm}\mathtt{M}^{qs_0}}\,.
\end{aligned}
\end{equation} 
Moreover one has
\begin{equation}\label{stimaIncredibleBisNOBNFMASS}
\|\psi(t)\|_{H^{s}}+\|\partial_{t}\psi(t)\|_{H^{s-1}}
\lesssim_{s}
\big(\|\psi_0\|_{H^{s}}+\|\phi_0\|_{H^{s-1}}\big)
 \Big[
 1+\mathtt{M}^{3qs_0(s-s_1-s_0)^2} \left(\frac{T}{T_{good}}\right)^{s-s_1+3}
 \Big]\,,
\end{equation}
for any $t\in[0,T]$ and,
as a consequence, the bound
\begin{equation}\label{stimaIncredibleMASS}
\sup_{t\in[0,\widehat{T}]}\big(
\|\psi(t)\|_{H^{s}}+\|\partial_{t}\psi(t)\|_{H^{s-1}}\big)
\lesssim_{s}
\big(\|\psi_0\|_{H^{s}}+\|\phi_0\|_{H^{s-1}}\big)
(1+
\mathtt{M}^{3qs_0(s-s_1-s_0)^2} ) \,,
\end{equation}
for any $ \widehat{T}\leq T_{good}$.
\end{theorem}

\medskip
Some comments and comparisons among 
the results above are in order.\\ 
$\bullet$ The strong relationship among the regularity of initial data, the size of the $L^2$-norm and the 
 time of  stability of solutions is explicitly evident.
This is shown in particularly in conditions 
\eqref{caramellina1}, \eqref{main:smallcondSob} and \eqref{patata1} of Theorem \ref{teonuovo}
where all the upper/lower bounds depend on $\delta$.
To resume, we obtained the following:
\begin{align}
&{\rm regularity:}\quad p\sim  \ln\big(\frac{1}{\delta}\big)^{\frac{5}{9}}\,,
\label{penna1}
\\
&{\rm size:}\quad \|\psi_0\|_{H^{p+\frac{1}{2}}}+\|\psi_1\|_{H^{p-\frac{1}{2}}}\lesssim\delta\,,
\qquad
\|\psi_0\|_{H^{\frac{1}{2}}}+\|\psi_1\|_{H^{-\frac{1}{2}}}\lesssim\delta 
e^{-\big[\ln\big(\frac{1}{\delta}\big)\big]^{\frac{5}{9}}}\,,\label{penna2}
\\
&{\rm time:}\;\;\; T\sim e^{\big[\ln\big(\frac{1}{\delta}\big)\big]^{1+\frac{1}{9}}}\,.
\label{penna3}
\end{align}
We stress the fact that the optimal choice in \eqref{penna1} is substantially unique, as a byproduct 
of our procedure. In contrast the time of stability in \eqref{penna3} is genuinely exponential 
in the parameter $(1/\delta)$ which is a drastic improvement for the usual times of stability
achieved for Sobolev initial data.

 On the other hand, in conditions \eqref{piccolezza dati trisNOBNF} and \eqref{patata1dimd} 
 of Theorem \ref{thm:mainNOBNF}, the smallness parameter $\delta$ 
 and the regularity index $s_1$ remain free parameters. 
 To resume, we obtained the following:
\begin{align}
&{\rm size:}\quad \|\psi_0\|_{H^{s_1+\frac{1}{2}}}+\|\psi_1\|_{H^{s_1-\frac{1}{2}}}\lesssim\delta\,,
\qquad
\|\psi_0\|_{H^{\frac{1}{2}}}+\|\psi_1\|_{H^{-\frac{1}{2}}}\lesssim\delta e^{-s_1}\,,\label{penna2bis}
\\
&{\rm time:}\;\;\; T\sim \frac{1}{\delta^q}e^{s_1}\,.
\label{penna3bis}
\end{align}

 Since the lifespan in \eqref{penna3bis} depends exponentially on $s_1$, 
 one can appropriately set $s_1$ as a function of $\delta$ in order to obtain a time of stability suitably long 
 depending only on the smallness of the initial conditions, consistently with the previous case.
% 
% it is natural to \emph{choose} $s_1$
%  as a function of $\delta$. This allows us to recover a lifespan 
%  depending exclusively on $1/\delta$, consistently with the previous case.
  We stress that the freedom on the choice of $s_1=s_1(\delta)$ is due to  the method used 
  for proving Theorem \ref{thm:mainNOBNF} (see Section \ref{sec:hamstructurewave}). 
  For clearness, by choosing 
  $s_1\equiv 1/\delta$ one gets a very long time of stability $T\sim e^{1/\delta}$
  which is extremely longer with respect to one obtained (by BNF procedure) in 
  \eqref{penna3}. 
  The drawback is that the smallness \eqref{penna2bis}
  becomes $$\|\psi_0\|_{H^{\frac{1}{2}}}+\|\psi_1\|_{H^{-\frac{1}{2}}}\lesssim\delta e^{-1/\delta}$$
which is much more restrictive than the one in \eqref{penna2}.
 
 To make this reasoning clearer, let us take $s_1$ in such a way that the time of stability 
 \eqref{penna3} and \eqref{penna3bis} match, more specifically, we set
 $s_1=\big[\ln\big(\frac{1}{\delta}\big)\big]^{1+\frac{1}{9}}$. 
 In this case,  \eqref{penna3} and  \eqref{penna3bis}  become  essentially identical, 
 while the smallness condition \eqref{penna2bis} reduces to
 \[
\|\psi_0\|_{H^{\frac{1}{2}}}+\|\psi_1\|_{H^{-\frac{1}{2}}}\lesssim\delta 
e^{-\big[\ln\big(\frac{1}{\delta}\big)\big]^{1+\frac{1}{9}}}\ll 
\delta e^{-\big[\ln\big(\frac{1}{\delta}\big)\big]^{\frac{5}{9}}}\,,
 \]
 which shows that  condition \eqref{penna2} is less restrictive.
 This is the price to pay for having a method that works in any space dimension and 
is free from
diophantine requirements on the linear frequencies.

It is worth mentioning that a combination 
of the two methods may lead to an improvement in the stability lifespan,
 without requiring the low norms to be accordingly small with respect to $\delta$.
 A first step in this direction is shown by
Cong-Mi-Shi in \cite{CongMiShi22}
for the NLS with convolution potential and purely cubic nonlinearity, where the method allows to reach up to super-exponential time scales, without asking for super-exponentially smallness of the low norm.\\
$\bullet$ Apart from the additional information on the possible growth of the high scales of norm, Theorem \ref{teonuovo} states a stability result for the KG comparable in terms of times/size to the one proved in \cite{FeoMass:Beam}. However, we emphasize that handling a BNF procedure for the KG equation involves additional technical difficulties related to the linear asymptotic of the frequencies of oscillation rather than the quadratic one proper to the Beam case, and more delicate Diophantine estimates come into play. See \cite{FeoMass:misurewave}.\\
$\bullet$ The flavor of Theorem \ref{thm:mainNOBNF} and \ref{thm:mainNOBNF2} evokes the results obtained in \cite{feoMassJFA} for the NLS on $\T^d$. However Theorem \ref{thm:mainNOBNF2} concerning the pure Wave equation (i.e. $\tm = 0$), require substantial new ideas and careful analysis with respect to the NLS case studied in \cite{feoMassJFA}. On the one hand the zero average terms that are involved must be kept in truck accurately in the Energy-method procedure, by defining an appropriate norm that fits well this resonant framework. On the other one, these terms call for an improved Gr\"onwall inequality that allow to handle both these extra non-autonomous terms and the superlinear power of the unbalanced norm that will be involved in the integral inequality, see estimates \eqref{suits111INTRO} in the following together with Lemma \ref{drago2}.

\subsection{Structure of the paper and strategy of the proofs.}
\label{sec:hamstructurewave}
In this section, we briefly discuss the main ideas underlying the proofs of our main results. 
We proceed in order, starting with the result on $\T$, 
which is based on the Birkhoff Normal Form approach.

\medskip
\noindent
{\bf Hamiltonian structure.} In general, we remark that equation \eqref{eq:wave} can be written as a first order system
which admits a natural Hamiltonian structure. Indeed, by setting 
$v=\partial_{t}\psi$,  and defining 
\begin{equation}\label{nonlionearitawave}
f(\psi):=\psi^{q+1}\,,\qquad F(\psi)=\frac{1}{q+2}\psi^{q+2}\,,
\end{equation}
we rewrite  \eqref{eq:wave}  as
\begin{equation}\label{eq:wave2}
\left\{\begin{aligned}
\partial_{t}\psi&=v 
\\ 
\partial_{t}v&=-\omega^2\psi-f(\psi)\,,
\end{aligned}\right.
\end{equation}
where $\omega$ 
is the Fourier multiplier defined by linearity as
\begin{equation}\label{omegoneWave}
\omega e^{{\rm i} j\cdot x}=\omega_{j} e^{{\rm i} j\cdot x}\,,
\qquad 
\omega_{j}:=\sqrt{|j|^{2}+\mathtt{m}}\,,
\qquad
\forall \,j\in \mathbb{Z}^{d}\,,\quad \mathtt{m}>0\,.
\end{equation}
Note that \eqref{eq:wave2} is the Hamiltonian vector field associated to the Hamiltonian 
$H_\R: H^{1}(\mathbb{T}^{d};\mathbb{R})\times L^{2}(\mathbb{T}^{d};\mathbb{R})\to \R$ defined as
\begin{equation}\label{WaveRealHam}
 H_{\mathbb{R}}(\psi,v)=
\int_{\mathbb{T}^{d}}
\big(
\frac{1}{2}v^{2}+\frac{1}{2}(\omega^{2}\psi) \psi+F(\psi)
\big){\rm d}x\,,
\end{equation} 
i.e. system \eqref{eq:wave2} reads
\[
\left(\begin{matrix}
\partial_{t}\psi \\ \partial_{t}v
\end{matrix}\right)
=J\nabla {H_{\mathbb{R}}}(\psi,v)=
J\left(\begin{matrix}
\partial_{\psi}H_{\mathbb{R}}(\psi,v)\\
\partial_{v}H_{\mathbb{R}}(\psi,v)
\end{matrix}\right)\,,
\quad 
J=\sm{0}{1}{-1}{0}
\]
where
$\nabla := (\partial_\psi,\partial_v)$ denotes the $L^{2}$-gradient.
Indeed we have 
\begin{equation}\label{eq:1.14bis}
\mathrm{d}H_{\mathbb{R}}(\psi,v)\cdot W
=
\Omega_{\mathbb{R}}(X_{H_{\mathbb{R}}}(\psi,v),W)
\end{equation}
for any 
$W \in H^{1}(\mathbb{T}^{d};\mathbb{R})\times L^{2}(\mathbb{T}^{d};\mathbb{R})$, 
where $\Omega_{\mathbb{R}}$ 
is the non-degenerate symplectic form
\[
\Omega_{\mathbb{R}}(W_1,W_2) 
= \int_{\mathbb{T}}(\psi_1v_2-v_1\psi_2)dx\,,
\quad \forall \;W_1=\vect{\psi_1}{v_1}\,,\;W_2=
\vect{\psi_2}{v_2}\in H^{1}(\mathbb{T}^{d};\mathbb{R})\times L^{2}(\mathbb{T}^{d};\mathbb{R})\,.
\]
The Poisson brackets between two Hamiltonian 
$H_{\mathbb{R}}, G_{\mathbb{R}}: 
H^{1}(\mathbb{T}^{d};\mathbb{R})\times L^{2}(\mathbb{T}^{d};\mathbb{R})\to \mathbb{R}$
are defined in the classical manner as
\begin{equation}\label{realpoisson}
\{H_{\mathbb{R}},G_{\mathbb{R}}\}
:=\Omega_{\mathbb{R}}(X_{H_{\mathbb{R}}},X_{G_{\mathbb{R}}})\,.
\end{equation}

\noindent
%We now discuss how to study system \eqref{eq:wave2}
%distinguishing the cases $\tm=0$ or $\tm>0$, and $d=1$ or $d>1$.

\vspace{0.5em}
\noindent
{\bf One dimensional case with $\mathtt{m}\neq0$: Birkhoff normal form approach.}
In this case it is convenient to introduce suitable complex variables as follows.

Let  us  consider 
\begin{equation}\label{sottoreal}
\cR = 
\set{(u^+, u^-)\in 
H^{\frac{1}{2}}(\mathbb{T}^{d};\mathbb{\C})\times  H^{\frac{1}{2}}(\mathbb{T}^{d},\C) \, : 
u^- = \bar{u}^+ }\,,
\end{equation}
and let us define  the linear symplectomorphism
\[
\cC:H^{1}(\mathbb{T}^{d};\mathbb{R})\times  L^{2}(\mathbb{T}^{d};\mathbb{R})
\to 
H^{\frac{1}{2}}(\mathbb{T}^{d};\mathbb{\C})\times  H^{\frac{1}{2}}(\mathbb{T}^{d};\mathbb{\C}) \cap \cR \,,
\]

\begin{equation}\label{beam5}
\vect{\psi}{v} \mapsto \mathcal{C}\vect{\psi}{v} = \vect{u}{\bar{u}}\,,\quad \mathcal{C}:=
\frac{1}{\sqrt{2}}\left(
\begin{matrix}
\omega^{\frac{1}{2}} & {\rm i} \omega^{-\frac{1}{2}}\\
\omega^{\frac{1}{2}} & -{\rm i} \omega^{-\frac{1}{2}}
\end{matrix}
\right)\,,
\end{equation}
where $\omega$ is the Fourier multiplier defined in \eqref{omegoneWave}\footnote{Note that, since 
$\tm \neq0$, one has $\omega_{j}\neq0$ for all $j\in \Z^{d}$. Therefore the operator $\omega$
is invertible.}.
The vector field $X_{H_{\mathbb{R}}}$ is then pushed forward to the new Hamiltonian one 
\begin{equation}\label{eq:waveComp}
\cC_* X_{\mathbb{R}} = X_H = (\dot u, \dot{\bar u}) \,,
\quad \quad 
\dot{u} 
= - {\rm i}\omega u- \frac{{\rm i} }{\sqrt 2}
\omega^{-1/2}f\left(\omega^{-1/2}\left(\frac{u+\bar u}{\sqrt 2}\right)\right) 
= - {\rm i}\partial_{\bar{u}}H(u,\bar{u})
\end{equation}
where 
$\partial_{\bar{u}}=(\partial_{\Re u}+{\rm i} \partial_{\Im u})/2$,
$\partial_{u}=(\partial_{\Re u}-{\rm i} \partial_{\Im u})/2$
and 
\begin{equation}\label{waveHam}
H(u,\bar{u})=H_{\mathbb{R}}(\mathcal{C}^{-1}\vect{u}{\bar{u}})
=
\int_{\mathbb{T}^{d}}\bar{u}\, \omega u\ \mathrm{d}x 
+\int_{\mathbb{T}^{d}} F\Big( \frac{\omega^{-1/2}(u+\bar{u})}{\sqrt{2}}\Big)\ {dx}\,.
\end{equation}
The (complex) induced $2$-form is 
\begin{equation}\label{complex form}
\Omega:= (\cC^{-1})^*\Omega_\R = \int_{\mathbb{T}^{d}}{\rm{i}} du\wedge d\bar u\,dx\,,
\end{equation}
which yields, for any 
$w_1 = (\xi , \bar\xi), w_2 
= (\eta, \bar\eta) \in H^{\frac{1}{2}}(\mathbb{T}^{d};\mathbb{\C})
\times  H^{\frac{1}{2}}(\mathbb{T}^{d};\mathbb{\C}) \cap \cR$,
\begin{equation}\label{symcompform}
\Omega (w_1,w_2) = \int_{\mathbb{T}^{d}} \xi \bar{\eta} - \bar\xi\eta \, dx\, \, \in\R
\end{equation}
and intrinsically defines the Hamiltonian through
\begin{equation}\label{diffHam}
\Omega(X_{H}(u), w) = \mathrm{d}H(u)\cdot w                                                                                                                                                                                                                                                                                                                                                                                                                                                                                                                                                                                                                                                                                                                                                                                                                                                                                                                                                                                                                                                                                                                                                                                                                                                                                                                                                                                                                                                                                                                                                                                                                                                                                                                                                                                                                                                                                                                                                                                                                                                                                                                                                                                     
\end{equation}
for any $w$.
Accordingly, we set the (complex) Poisson brackets as
\[
\set{H, G} := 
\Omega(X_H, X_G)
\]  
where  $H=H_{\mathbb{R}}\circ\mathcal{C}^{-1}$  
(resp $G=G_{\mathbb{R}}\circ\mathcal{C}^{-1}$ ) 
and $X_H = \cC_*X_{H_\R}$ (resp $X_G = \cC_*X_{G_\R}$)
which yields\footnote{
Note that the naturality of the Poisson brackets 
holds $\set{H_\R, G_\R}\circ \cC^{-1} = \set{H,G}$.
}
\begin{equation}\label{Poissonbrackets}
\begin{aligned}
\{H,G\} 
%\\&
={\rm i} \int_{\mathbb{T}^{d}}
\big(\partial_{u}G\partial_{\bar{u}}H-
\partial_{\bar{u}}G\partial_{u}H\big) \mathrm{d}x\,.
\end{aligned}
\end{equation}

The main ideas in order to prove the stability Theorem \ref{teoBNFwave} are the following.\\ 
First of all, we consider the Hamiltonian in \eqref{waveHam}, which, 
passing to the Fourier side, has the form 
\begin{equation}\label{marconi1}
H = \sum_{j\in \Z} \omega_j |u_{j}|^{2}+O(u^3)\,,\qquad 
\omega_{j}:=\omega_{j}(\mathtt{m})=\sqrt{j^2+\mathtt{m}}\,,
\end{equation}
where $O(u^3)$ denotes a non linearity with a zero at the origin of order at least $3$.
From this Hamiltonian viewpoint,  through suitable symplectic change of coordinates, we shall pull 
 the Hamiltonian $H$ back to 
 \emph{Birkhoff Normal Form} (BNF)
\[
\widetilde{H}=\sum_{j\in \Z} \omega_j |u_{j}|^{2}+\mathfrak{Z}+\mathfrak{R}\,,
\] 
where $\mathfrak{Z}$ depends only on the ``actions''  $|u_j|^2$ (and does not affect the dynamics)
while $\mathfrak{R}$ has an \emph{high} degree of homogeneity $\sim O(|u|^{\suca+2})$ 
for some natural $\suca\gg1$. 
{Then, the natural time of stability of the flow of 
$\widetilde{H}$ becomes $T(\delta)\sim O(\delta^{-\suca})$.}
The crucial difficulties in this approach
 regard the regularity 
of the phase space
of initial data and interactions among linear frequencies of oscillations $\omega_j$.

More precisely, along the iteration, we must deal with two main problems:
(i) we must deal with the small divisors problem (see \eqref{ondina}-\eqref{seconda});
(ii) we must keep track of the dependence of constants and regularity on the 
number of steps $\mathtt{N}$ of normal form. Comments come in order.

\smallskip
\noindent
$(i)$ As already mentioned, certain \emph{Diophantine} conditions on the frequency vector
$\omega=(\omega_j)_{j\in\Z}\in \R^{\Z}$ are required.
In our setting, this issue is particularly delicate because 
of the relative ``lack'' of parameters in the  frequencies  of oscillation \eqref{omegoneWave}.
Indeed, in our case there is only \emph{one} parameter 
available to modulate arbitrary subsets of the frequencies.
For instance this feature does not allow to prove strong lower bounds \`a la Bourgain 
as in \cite{BMP:CMP}. On the contrary we rely on the lower bounds for the Klein-Gordon equation 
proved in \cite{FeoMass:misurewave}. For completeness we state such result in Theorem 
\ref{thm:mainVERO} of section \ref{sec:small}.

\smallskip
\noindent
$(ii)$ In order to get sharp estimates in the BNF procedure we adopt the Hamiltonian 
formalism introduced in \cite{BMP:CMP, FeoMass:Beam}.
The key point is to introduce an equivalent Sobolev norm (see Section \ref{sec:equivnorm})
defined as follows:
\begin{align}\label{nuovanormaIntro}
\|u\|_{s,R}^{2}&:=
(\jjap{D}^{s}u,\jjap{D}^{s}u)_{L^{2}}
=\sum_{j\in\Z^{d}} \jjap{j}^{2s} |\widehat{u}(j)|^{2}\,,
\qquad \jjap{j}:=\jjap{j}_{R}:=\max{\{R,|j|\}}\,.
\end{align}
In Lemma \ref{lemmaequivalenza} we prove the equivalence of the norm above 
with the standard Sobolev norm in \eqref{spaziodiSobolev}.
The advantage is that 
the new norm encodes precisely the fact that we 
deal with initial conditions with very small $L^{2}$-norm 
(see the smallness assumption \eqref{main:smallcondSob}).

Relying on these technical ingredients, in Section \ref{sec:normaldim1}
we prove an abstract Birkhoff normal form result (see Theorem \ref{rescalingHaminiziale}).
The key step in the proof is the derivation of sharp estimates 
for the solutions of the homological equation, 
which are established in Section \ref{sec:homo}.
The stability result stated in Theorem \ref{teoBNFwave} follows 
as a consequence of the normal form theorem 
established in Theorem \ref{rescalingHaminiziale}.
An optimization argument carried out in Subsection \ref{sec:ottimizzo} 
(which justifies the choice made in \eqref{caramellina1}) 
then yields the stability estimate \eqref{stimaIncrediblelowNORM} 
over the time scale specified in \eqref{patata1}.
The estimate \eqref{stimaIncredibleBis} 
on high Sobolev norms in Theorem \ref{teonuovo} follows 
by an improved Gr\"onwall type inequality. 
We provide more details on this point in the remaining part of this introduction.

\vspace{0.5em}
\noindent
{\bf Higher dimensional case with $\mathtt{m}\neq0$: a priori estimates.}
Let us discuss the ideas of the proof of Theorem \ref{thm:mainNOBNF}.
Here again we rewrite equation \eqref{eq:wave} as \eqref{eq:waveComp}, which, by Duhamel 
formulation reads as
\begin{equation}\label{eq:duhamelintro}
u(t)=e^{-\ii\omega t}u(0)+\int_{0}^{t}e^{-\ii\omega (t-\s)}
\omega^{-1/2}f\left(\omega^{-1/2}\left(\frac{u(\s)+\bar{u}(\s)}{\sqrt 2}\right)\right) d\s\,.
\end{equation}
Basic a priori estimates, together with \emph{tame} estimates for 
the norm in \eqref{nuovanormaIntro} (see Lemma \ref{stimaNonlin}),  
guarantee that  
%(see Proposition \ref{thm:energyBasic})
\[
\|u(t)\|_{{s_1},R}\leq \|u(0)\|_{{s_1},R}
+C\int_{0}^{t}\|u(\s)\|_{s_0,R}^{q}\|u(\s)\|_{s_1,R}d\s\,,
\]
for some constant $C>0$, 
providing an \emph{a priori} control on the
Sobolev norm $\|\cdot\|_{{s_1,R}}$, $s_1\geq s_0>d/2$ of the solution.
Thanks to the tame/scaling properties of the norm (see Lemma \ref{interopolo}), by taking 
$R\equiv2\mathtt{M}$ with $\mathtt{M}\gg1$ in \eqref{piccolezza dati trisNOBNF} 
large enough,
we can improve the above estimate as follows
\[
\|u(t)\|_{{s_1},R}\leq \|u(0)\|_{{s_1},R}
+\widetilde{C}2^{-q(s_1-s_0)}\int_{0}^{t}\|u(\s)\|_{s_1,R}^{q+1}d\s\,.
\]
This is the content of Theorem \ref{thm:energyBasic}.
The extra factor $2^{-q(s_1-s_0)} $
allows us to prove the stability of the solutions over a time scale as in \eqref{patata1dimd}.
More precisely, combining the energy estimates above with a Gr\"onwall-like lemma,
one extends the solution of \eqref{eq:waveComp} over the claimed time scale.

Now,proving the estimates on the high Sobolev norms 
\eqref{stimaIncredibleBisNOBNF} is more subtle, and require an additional argument. \\
As a consequence of the change of variables \eqref{beam5}, we start by noting that
the nonlinearity in \eqref{eq:waveComp} 
acquires some smoothing properties (see Lemma \ref{lem:stimasmooth} for details).
Therefore, on the solutions of \eqref{eq:duhamelintro}, one obtains estimates like
\[
\|u(t)\|_{s,R}\lesssim_{s}
\|u(0)\|_{s,R}
+C_{s}\int_{0}^{t}\|u(\s)\|_{s_0,R}^{q}\|u(\s)\|_{s-1,R}d\s\,,
\]
for any $s\geq s_1+1\geq s_0$ (see inequalities \eqref{energyaprioriBIS}).
Since the norm $\|\cdot\|_{s,R}$ enjoys the classical 
\emph{interpolation estimates}
in Sobolev spaces (see Lemma \ref{interopolo}-(ii)) the estimate above becomes
\[
\|u(t)\|_{s,R}\leq \|u_0\|_{s,R}
+\widetilde{C}_s
\int_{0}^{t}
\|u(\tau)\|^{q+1-\lambda}_{s_1,R}
\|u(\tau)\|_{s,R}^{\lambda}d\tau\,\,
\]

for an appropriate $0<\lambda < 1$, which is crucial for what follows. This is the content of Proposition \ref{thm:energy}.
Combining the improved a priori estimates with an improved 
Gr\"onwall lemma (Lemma \ref{drago}-(ii) in the appendix)
we are able to get a sharper control on the high norms $\|u(t)\|_{s}$. 
In particular we show that these norms \emph{may} 
grow to infinity at most at a polynomial rate in $t$, 
this proving Theorem \ref{thm:mainNOBNF}.
We remark that the estimates \eqref{stimaIncredibleBisNOBNF} holds over $[0,T]$, 
with $T\gtrsim T_{good}$ in \eqref{patata1dimd} which depends 
only on the \emph{low} index regularity $s_1$.

With a similar argument we prove the upper bound \eqref{stimaIncredibleBis} 
in Theorem \ref{teonuovo}. The key difference is that, in this latter case,
the stability of the low Sobolev norm and (as a consequence) the lower bound on the lifespan, 
is obtained through a normal form procedure.

\vspace{0.5em}
\noindent
{\bf High space dimension with $\mathtt{m}=0$: the completely resonant case.}
In the case of no mass, the operator $\omega$ in \eqref{omegoneWave}
reduces to the derivative operator $|D| := \sqrt{-\Delta}$.
It is evident that $|D|$ is invertible only when it acts on the subspace of zero average functions.
Therefore, in general, the 
symplectomorphism $\cC$ is not well
defined. 

In order to handle this parameter free case, we shall 
distinguish and treat separately the averages of the involved 
functions from the non constant terms. 
In order to do that, given $\psi\in H^{s}(\T^{d};\C)$
we shall write
\begin{equation}\label{splittoperp}
\psi=\psi_{0}+\psi^{\perp}\,,\qquad \psi_{0}:=\frac{1}{(2\pi)^d}\int_{\T^{d}}\psi(x)dx\,,
\qquad \psi^{\perp}:=\psi-\psi_0\,.
\end{equation}
Accordingly we denote $\Pi_0$ and $\Pi_0^{\perp}$ the corresponding $L^{2}$ projectors
on the average and zero average part respectively.
With this formalism equation \eqref{eq:wave}-\eqref{nonlionearitawave} reads 
\[
\left\{
\begin{aligned}
&\partial_{tt}\psi_0+\Pi_{0}(f(\psi))=0
\\
&\partial_{tt}\psi^{\perp}+|D|^{2}\psi^{\perp}+\Pi_0^{\perp}(f(\psi))=0\,.
\end{aligned}
\right.
\]
Setting $v_0=\partial_{t}\psi_0$ and $v^{\perp}=\partial_{t}\psi^{\perp}$
we shall write
\begin{equation}\label{eq:massazero}
\left\{
\begin{aligned}
&\partial_{t}\psi_0=v_0
\\
&\partial_{t}v_0=-\Pi_0(f(\psi))
\\&
\partial_t\psi^{\perp}=v^{\perp}
\\&
\partial_{t}v^{\perp}=-|D|^{2}\psi^{\perp}-\Pi_0^{\perp}(f(\psi))\,.
\end{aligned}
\right.
\end{equation}
On the subspace $H^{s}_{\perp}:=H^{s}(\T^{d};\C)\cap \{u : u=\Pi_{0}^{\perp}u\}$
we introduce the variables\footnote{In principle, one could choose  the map $\cC$ in
\eqref{beam5} with $\tm=0$ restricted to $H^{s}_{\perp}$. However, in this case we do not care
about the  Hamiltonian structure of the whole system and so  we choose to use  the simpler map 
$\cD$ whose inverse is given  by
\[
\psi^{\perp}=u^{\perp}+\bar{u}^{\perp}\,,\qquad v^{\perp}=-\ii |D|(u^{\perp}-\bar{u}^{\perp})\,.
\]}
\begin{equation}\label{complexNew}
\left(
\begin{matrix}
\psi^\perp
\\
v^{\perp}
\end{matrix}
\right) \mapsto \cD \left(
\begin{matrix}
\psi^\perp
\\
v^{\perp}
\end{matrix}
\right) = \left(
\begin{matrix}
u^\perp
\\
\bar{u}^{\perp}
\end{matrix}
\right)\,,\quad \cD:=
\frac{1}{2}\left(
\begin{matrix}
1 & {\rm i} |D|^{-1}\\
1 & -{\rm i} |D|^{-1}
\end{matrix}
\right)\,,
\end{equation}
so that \eqref{eq:massazero} becomes
\begin{equation}\label{sistemasplittato}
\left\{
\begin{aligned}
&\partial_{t}\psi_0=v_0
\\
&\partial_{t}v_0=-\Pi_0\big(f(\psi_0+u^{\perp}+\bar{u}^{\perp})\big)
\\&
\partial_{t}u^{\perp}=-\ii |D|u^{\perp}
-\frac{\ii}{2}|D|^{-1}\Pi_0^{\perp}\Big(f\big(
\psi_0+u^{\perp}+\bar{u}^{\perp}
\big)\Big)\,.
\end{aligned}
\right.
\end{equation}
It is worth noticing that, in view of the change of coordinates 
\eqref{complexNew} and by introducing the \emph{modified} energy 
\[
E_{s}(t):=E_{s,R}(t):=|\psi_0(t)|+|v_0(t)|+\|u^{\perp}(t)\|_{s,R}\,,
\]
we have the following equivalence 
(see Lemma \ref{lemmaequivalenza} and Remark \ref{rmk:equivalenzaTotale})
\[
\|\psi(t)\|_{H^{s_1}}+\|\partial_{t}\psi(t)\|_{H^{s_1-1}}
+R^{s_1}\big(
\|\psi(t)\|_{L^{2}}+\|\partial_{t}\psi(t)\|_{H^{-1}}
\big)\sim_{s}E_{s}(t)\,.
\]
We also remark that, by the equivalence above
and the smallness assumption \eqref{piccolezza dati trisNOBNF1}, one deduces that 
there exists a constant $c\geq 1$ such that
\begin{equation*}
\begin{aligned}
\\
E_{s_1}(0)&\leq c\delta\,,
\qquad 
|\breve{\psi}_0|+|\breve{v}_0|
\leq E_{s_1}(0) R^{-s_1}\,,
\qquad R:=2\mathtt{M}\,.
\end{aligned}
\end{equation*}
We stress that the averages in $x$, that is $\breve{\psi}_0$ and $\breve{v}_0$, are \emph{exsponentially} small
in the regularity parameter $s_1$.

Now, in order to prove 
Theorem \ref{thm:mainNOBNF2} we analyze the evolution in time 
of the quantity $E_{s}(t)$ following  the overall strategy described for 
Theorem \ref{thm:mainNOBNF}.

As a first step, we prove that the \emph{low} energy $E_{s_1}(t)$ satisfies 
$E_{s_1}(t)\leq 2E_{s_1}(0) $ for $t\in[0,T]$ with $T$ as in \eqref{patata1dimdMASS}.
This is a consequence of the a priori bounds in 
Proposition \ref{thm:energyBasicBIS} and a Gr\"onwall type inequality.

As a second step, the most delicate one,  we analyze $E_{s}(t)$ for $s\geq s_1+1$.
Relying again on the smoothing properties of the nonlinearity,  
in Proposition \ref{thm:energyBIS} (see estimate \eqref{suits111}) 
we prove that
\begin{equation}\label{suits111INTRO}
E_{s}(t)\leq E_{s}(0)+|\breve{v}_0| t+
K
\int_{0}^{t}
\int_{0}^{\s}(E_{s_1}(\tau))^{q+1}d\tau d\s
%\\&
+K\int_{0}^{t}  (E_{s_1}(\s))^{q+1-\lambda}
(E_{s}(\s))^{\lambda}d\s\,,
\end{equation}
for some constant $K=K(s)$ large enough, and where $\lambda=1-1/(s-s_1)$.
Here again, the crucial point in the estimate above is that the \emph{high} norm $E_s(t)$ in the second integral summand
appears with an exponent \emph{strictly smaller than one}. 
All the other terms depend only on the \emph{low} norm $E_{s_1}(t)$
on which we have already proved an a priori control over the long time scale $[0,T]$, $T\gtrsim 2^{s_1/2}$.
Now, by setting for $t\in[0,T]$
\begin{equation*}
\begin{aligned}
f(t)&:= 
E_{s}(0)+|\breve{v}_0| t+
K
\int_{0}^{t}
\int_{0}^{\s}(E_{s_1}(\tau))^{q+1}d\tau d\s
\,,
\\
g(\s)&:= K
(E_{s_1}(\s))^{q+1-\lambda}\,,
\end{aligned}
\end{equation*}
we rewrite \eqref{suits111INTRO} as
\[
E_{s}(t)\leq f(t)+\int_{0}^{t}g(\s) (E_{s}(\s))^{\lambda}d\s\,.
\]
It is worth noticing that, since $f(t)$ depends on time, 
the Gr\"onwall type inequality of Lemma \ref{drago}-(ii) cannot be applied.
Nevertheless, in  Lemma \ref{drago2} we provide a more subtle version of Gr\"onwall inequality, that guarantees a control form above of $E_{s}(t)$
only in terms of the functions $f(t), g(t)$, namely on the low norm $E_{s_1}(t)$.
Combining this with the a priori estimates on low norm
one gets the upper bound \eqref{stimaIncredibleBisNOBNFMASS}. 
This is the content of Section \ref{sec:finaleWave}.

\medskip
The paper is organized as follows. In Section \ref{sec:equivnorm} we introduce the general functional  
setting. Then the analysis is divided into two parts:
(1) we study the Klein-Gordon equation 
(i.e. when $\tm\neq0$) first considering the one dimensional case,  and
then studying the higher space dimension; (2) we consider the completely resonant case of the Wave 
equation (i.e. $\tm=0$).

\section{Equivalent norms and functional properties}\label{sec:equivnorm}

\noindent
{\bf Notations}.
Through all the paper we shall adopt the following conventions. \\We shall use $A\lesssim B$ in order to denote 
$A\le C B$ where $C$ is a positive constant
depending on  parameters fixed once for all, 
for instance $d$ and $s$.
 We will emphasize by writing $\lesssim_{q}$
 when the constant $C$ depends on some other parameter $q$.
 We shall write $\lessdot_{s}$
 if there exist a constant $C>0$ (independent of $s$)
 such that $A\leq C^{s}B$.  
  We shall write $\lessdot_{s,s_0}$
 if there exist a constant $C>0$ (independent of $s$ and $s_0$)
 such that $A\leq C^{\max\{s,s_0\}}B$. 
 
 \noindent
Moreover, for $r\in\R^{+}$, we shall
denote by $B_{r}(H^{s}(\mathbb{T}^{d};\mathbb{C}))$
the open ball of $H^{s}(\mathbb{T}^{d};\mathbb{C})$ 
with radius $r$ centred at the origin.

\medskip

As usual, we identify the Sobolev space of $2\pi$-periodic functions 
$\mathbb{\R}^{d}\ni x \mapsto u(x)\in \mathbb{C}$ with 
$H^{s}(\mathbb{T}^d;\mathbb{C})$, and develop any function $u: \T^d\to \C$ 
in its Fourier series as\footnote{We also use the notation
$u_j^+ := u_j$ and
$u_j^- := \ov{u_j}  $.}
\begin{equation}\label{espFourier}
u(x) =
\sum_{j \in \Z^{d} } u_{j} e^{\ii j\cdot x } \, , \qquad 
u_{j} := \frac{1}{(2\pi)^{{d}}} \int_{\mathbb{T}^{d}} u(x) e^{-\ii j \cdot x } \, dx \, .
\end{equation}
We endow $H^{s}(\mathbb{T}^{d};\mathbb{C})$ with the norm 
\begin{equation}\label{normaNormalissima}
\|u\|_{H^{s}}:=\sum_{j\in\Z^{d}}\langle j\rangle^{2s}|u_{j}|^2\,,\qquad \langle j\rangle:=\max\{1,|j|\}\,.
\end{equation}

%\subsection{An equivalent norm}
\smallskip
\noindent
{\bf An equivalent norm.}
For our purposes, on the space $H^{s}(\T^{d};\C)$ 
we shall define and make use of the following equivalent norm, 
instead of the usual one.

\begin{definition}{\bf (Equivalent norm).}\label{def:equinorm}
Let $s, R\in \R$ with $R\geq1$. 
Let us define the operator  $\jjap{D}$ defined by 
linearity as 
\begin{align}
&\jjap{D}e^{\ii j\cdot x}=\jjap{j}^{2}e^{\ii j\cdot x}\,,
\qquad \jjap{j}:=\jjap{j}_{R}:=\max{\{R,|j|\}}\,,\quad j\in\Z^{d}\,,
\label{def:japjapModificato}
\end{align}
For any $u\in H^{s}(\T^{d};\C)$ we define
\begin{align}
\|u\|_{s,R}^{2}&:=
%|u|^{2}_{s,R}:=
(\jjap{D}^{s}u,\jjap{D}^{s}u)_{L^{2}}
=\sum_{j\in\Z^{d}} \jjap{j}^{2s} |\widehat{u}(j)|^{2}\,.
\label{normaJapJap}
\end{align}
\end{definition}
\begin{remark}\label{normaR2}
The norm $\|\cdot \|_{s,R}$  is monotone w.r.t. the Sobolev index $s$: 
for any $s_1\geq s_2\geq s_0>d/2$ 
and any $u\in H^{s_1}(\T^{d};\C)$ one has
$\|u\|_{s_2,R}\leq \|u\|_{s_1,R}$.

\noindent
For the special case where $R=2$ in $\|\cdot\|_{s,R}$, along the paper we shall write
\begin{equation}\label{def:japjapModificatoRR2}
\|u\|_{p} := \|u\|_{p,2}\,.
\end{equation}
Notice also that,
 one has
$\|u\|_{s,1}\equiv\|u\|_{H^{s}}$.
\end{remark}

\begin{lemma}{\bf (Equivalence of norms).}\label{lemmaequivalenza}
 %For any $u\in H^{s}(\T^{d};\C)$, 

\smallskip
\noindent
$(i)$
 Let $s\geq 0$, consider 
 $(\psi,v)\in H^{s+\frac{1}{2}}(\T^d,\R)\times H^{s-\frac{1}{2}}(\T^d,\R)$ and the Fourier (multiplier) operator $\omega$ defined in \eqref{omegoneWave}. 
Then 
\begin{equation}\label{def:cambio1}
u=(1/\sqrt{2})(\omega^{1/2}\psi+\ii \omega^{-1/2}v)\, 
\end{equation}
belongs to $H^{s}(\T^{d};\C)$ and 
there exists a constant $C=C(\mathtt{m})>0$ such that
\begin{equation}\label{primaEquiv}
\frac{1}{C}	(\|\psi\|_{s+\frac{1}{2},1} + \|v\|_{s -\frac{1}{2},1} ) 
\leq	
\| u \|_{s,1} \leq C (\|\psi\|_{s+\frac{1}{2},1} + \|v\|_{s - \frac{1}{2},1} )\,.
\end{equation}

\smallskip
\noindent
$(ii)$
 Let $s\geq 0$ and consider $(\eta,w)\in H_{\perp}^{s}(\T^d,\R)\times H_{\perp}^{s-1}(\T^d,\R)$. 
Then 
\begin{equation}\label{def:cambio2}
z=(1/2)(\eta+\ii |D|^{-1}w)\,
\end{equation}
belongs to $H_{\perp}^{s}(\T^{d};\C)$ and 
there exists a pure constant $C>0$ such that
\begin{equation}\label{secondaEquiv}
C^{-1}\big(\|\eta\|_{s,1} + \|w\|_{s-1,1} \big) 
\leq	
\| z \|_{s,1} 
\leq 
C \big(\|\eta\|_{s,1} + \|w\|_{s-1,1} \big)
\end{equation}

\smallskip
\noindent
$(iii)$ Let $R\geq1$. For any $s\geq0$ and any $u\in H^{s}(\T^d;\C)$ one has
\begin{align}
\|u\|_{H^s}+R^{s}\|u\|_{L^{2}}
&\leq 
2 \|u\|_{s,R}
\leq  
2\big(
\|u\|_{H^s}+R^{s}\|u\|_{L^{2}}
\big)\,.\label{equiIncredibile1}
\end{align}

\end{lemma}

\begin{proof}
Item $(i)$. By a straightforward
computation one can check that (recall \eqref{omegoneWave})
\begin{equation}\label{equiomeghino}
\max\{\sqrt{\tm},|j|\}
\leq \omega_{j}\leq \sqrt{(\tm +1)}\max\{1,|j|\}\,,\qquad \forall\, j\in\Z^{d}\,.
\end{equation}
Hence, using \eqref{def:cambio1}, we deduce that
\[
\begin{aligned}
\|u\|_{s,1}^{2}&\lesssim\|\omega^{\frac{1}{2}}\psi\|_{s,1}^{2}+\|\omega^{-\frac{1}{2}}v\|_{s,1}^2
\\&
\stackrel{\eqref{normaNormalissima}}{\lesssim}
\sum_{j\in\Z^{d}}\langle j\rangle^{2s}\omega_{j}^{\frac{1}{2}}|\psi_{j}|^{2}
+
\sum_{j\in\Z^{d}}\langle j\rangle^{2s}\omega_{j}^{-\frac{1}{2}}|v_{j}|^{2}
\stackrel{\eqref{equiomeghino}}{\lesssim_{\mathtt{m}}}
\sum_{j\in\Z^{d}}\langle j\rangle^{2(s+\frac{1}{2})}|\psi_{j}|^{2}
+
\sum_{j\in\Z^{d}}\langle j\rangle^{2(s-\frac{1}{2})}|v_{j}|^{2}\,,
\end{aligned}
\]
which implies the second inequality in \eqref{primaEquiv}.
Now, by \eqref{def:cambio1}, we shall write
\[
\psi=(1/\sqrt{2})\omega^{-\frac{1}{2}}(u+\bar{u})\,,
\qquad 
v=(1/(\ii\sqrt{2}))\omega^{\frac{1}{2}}(u-\bar{u})\,.
\]
Therefore,	 the first inequality in \eqref{primaEquiv} follows reasoning as above.
\smallskip 

\noindent
Item $(ii)$. 
It follows by reasoning exactly as for Item $(i)$, using
\eqref{def:cambio2} and the fact that 
\[
\eta= z+\bar{z}\,,\qquad w=(-\ii/2)|D|(z-\bar{z})\,.
\]

\noindent
Item $(iii)$. 
The equivalence \eqref{equiIncredibile1} follows 
by an explicit computation using 
\[
|j|\leq \max\{1,|j|\}\,,\quad R\leq \max\{R,|j|\}\,,\quad \max\{1,|j|\}\leq \max\{R,|j|\}\,,
\]
and recalling \eqref{normaNormalissima}-\eqref{normaJapJap}.
%The second inequality in \eqref{equiIncredibile1} is trivial.
%It is easy to check (using Remark \ref{rmk:facile} and recalling 
%\eqref{normaAncoraequi}) that
%$\|u\|_{H^{s}}\leq |u|_{H^s}$ and 
%$\|u\|_{L^{2}}^{2}\leq R^{-s}|u|_{H^s}^2$. 
%Moreover one has that
%\begin{equation*}
%\begin{aligned}
%\|u\|_{H^s}+R^{s}\|u\|_{L^{2}}&\stackrel{\eqref{normaNormalissima}}{=}
%\Big(\sum_{j\in\Z^{d}}R^{2s}|\widehat{u}(j)|^{2}\Big)^{1/2}
%+\Big(\sum_{j\in\Z^{d}}|\widehat{u}(j)|^{2}\Big)^{1/2}
%+\Big(\sum_{j\in\Z^{d}}|j|^{2s}|\widehat{u}(j)|^{2}\Big)^{1/2}
%\\&\leq 3\Big(\sum_{j\in\Z^{d}}(\max{R,|j|})^{2s}|\widehat{u}(j)|^{2}\Big)^{1/2} 
%\stackrel{\eqref{normaJapJap}}{=}3\|u\|_{s,R}\,.
%\end{aligned}
%\end{equation*}
%Then we deduce the first inequality in \eqref{equiIncredibile1}. This concludes the proof.
\end{proof}

\begin{remark}\label{rmk:equivalenzaTotale}
By \eqref{def:cambio1} and combining 
\eqref{primaEquiv} and \eqref{equiIncredibile1} one has, for $s\geq0$ the equivalence
\begin{equation}\label{rmk:equivalenzaTotale1}
\|\psi\|_{s+\frac{1}{2},1} + \|v\|_{s -\frac{1}{2},1}+R^{s}(\|\psi\|_{\frac{1}{2},1} + \|v\|_{-\frac{1}{2},1})
\sim_{\tm}\|u\|_{s,R}\,.
\end{equation}
Similarly, \eqref{def:cambio2} and combining 
\eqref{secondaEquiv} and \eqref{equiIncredibile1} one has, for $s\geq0$ the equivalence
\begin{equation}\label{rmk:equivalenzaTotale2}
\|\psi\|_{s,1} + \|v\|_{s -1,1}+R^{s}(\|\psi\|_{0,1} + \|v\|_{-1,1})
\sim\|u\|_{s,R}\,.
\end{equation}
\end{remark}

\begin{definition}{\bf (Projectors).}\label{def:projectors}
Given $N>1$, we define 
the \emph{projector} $\Pi_{N}$ as
\[
\Pi_{N}u=\frac{1}{(2\pi)^{d/2}}
\sum_{\substack{|j|\leq N}}
\widehat{u}(j)e^{\ii j\cdot x}\,,
\qquad 
u\in H^{s}(\T^{d};\C)\,.
\]
We set $\Pi_{N}^{\perp}:=\uno-\Pi_{N}$ where $\uno$ is the identity.
\end{definition}

We have the following classical result.
\begin{lemma}\label{interopolo}
Let $s>0$. $(i)$ One has
\begin{align*}
\|\Pi_{N}u\|_{s+\beta,R}&\leq \max\{R,N\}^{\beta}\|u\|_{s,R}\,,
\qquad 
\forall\, u\in H^{s}(\T^{d};\C) \,,\; \beta\geq0
\\
\|\Pi_{N}^{\perp}u\|_{s,R}&\leq \max\{R,N\}^{-\beta}\|u\|_{s+\beta,R}\,,
\qquad 
\forall\, u\in H^{s+\beta}(\T^{d};\C) \,,\; \beta\geq0
\end{align*}
$(ii)$ For $0\leq s_1\leq s\leq s_2$, $s:=(1-\lambda) s_1+\lambda s_2$, one has
\begin{equation*}%\label{albero3}
\|u\|_{s,R}\leq 18 R^s \|u\|_{s_1,R}^{1-\lambda}\|u\|_{s_2,R}^{\lambda}\,,
\qquad 
\forall\, u\in H^{s_2}(\T^{d};\C)\,.
\end{equation*}
\end{lemma}

\begin{proof}
The proof is based on standard interpolation properties 
of the $H^s$-norm and the equivalence relations  
in Lemma \ref{lemmaequivalenza}. 
For the proof, see the corresponding \cite[Lemma 2.6]{feoMassJFA}.
\end{proof}

\begin{remark}[Scaling]\label{scaling property}
In the same spirit as Lemma \ref{interopolo}, 
using that $\jjap{j}\geq {R}$ for any $j\in\Z^{d}$,
%Remark \ref{rmk:facile}, 
we get,  for $s\geq s_0>d/2$,
\begin{align}
\|u\|_{s_0,R}&\leq R^{-(s-s_0)}\|u\|_{s,R}\,,\qquad u\in H^{s}(\T^d;\C)\,.
\label{confort3}
\end{align}
\end{remark}
We also need the following Lemma.
\begin{lemma}\label{lem:stimasmooth}
Let $R\geq 1$ and $s\geq 0$. The following holds.

\noindent
$(i)$ For $z\in H^{s}(\T^{d};\C)$, one has 
\begin{equation}\label{collina1}
\|\omega^{-\frac{1}{2}}z\|_{s+\frac{1}{2}, R}\lesssim_{\tm} \sqrt{R}\|z\|_{s,R}\,.
\end{equation}
\noindent
$(ii)$ For $u\in H_{\perp}^{s}(\T^{d};\C)$, one has 
\begin{equation}\label{collina2}
\||D|^{-1}u\|_{s+1, R}\lesssim R\|u\|_{s,R}\,.
\end{equation}
\end{lemma}
\begin{proof}
We only prove item $(i)$. Item $(ii)$ follows reasoning similarly.
First of all we recall that $\omega_{j}\sim_{\tm}\max\{1,|j|\}$. Therefore, using 
\eqref{def:japjapModificato}-\eqref{normaJapJap},
we have
\[
\begin{aligned}
\|\omega^{-\frac{1}{2}}z\|_{s+\frac{1}{2}, R}^{2}&\lesssim 
\sum_{j\in\Z^{d}}\big(\max\{R,|j|\}\big)^{2(s+\frac{1}{2})}\omega_{j}^{-1}|z_{j}|^2
\\&
\lesssim_{\tm}R\sum_{|j|\leq }R^{2s}|z_{j}|^{2}+\sum_{|j|>R}|j|^{2s}|j|\omega_{j}^{-1}|z_{j}|^{2}
\lesssim_{\tm}R\|z\|_{s,R}^{2}\,.
\end{aligned}
\]
This implies the claim.
\end{proof}
\begin{remark}\label{rmk:collina}
Under the assumptions of Lemma \ref{lem:stimasmooth}, it is trivial to check the estimates
\begin{equation}\label{collina3}
\|\omega^{-\frac{1}{2}}z\|_{s, R}\lesssim_{\tm}\|z\|_{s,R}\,,
\qquad 
\||D|^{-1}u\|_{s, R}\lesssim \|u\|_{s,R}\,.
\end{equation}
\end{remark}

We now prove  that the norm $\|\cdot\|_{s,R}$
introduced in Definition  \ref{def:equinorm} for any fixed $R>0$ satisfies \emph{tame} estimates.
\begin{lemma}{\bf (Basic nonlinear estimates).}\label{stimaNonlin}
For any $s\geq s_0>d/2$ there exists an absolute constant $\mathtt{M}>0$ such that,
for all integers $q_1,q_2\geq0$, $q_1+q_2\geq1$,
\begin{equation}\label{confort}
\| u^{q_1}\bar{u}^{q_2} \|_{s,R}\leq \mathtt{M}^{(q_1+q_2-1) s}\|u\|_{s_0,R}^{q_1+q_2-1}\|u\|_{s,R}\,,
\qquad 
\forall\; u\in H^{s}(\T^d;\C)\,.
\end{equation}
Moreover one has
\begin{equation}\label{confort2}
\| u^{q_1}\bar{u}^{q_2} \|_{s,R}
\leq \mathtt{M}^{(q_1+q_2-1) s_0}
(\mathtt{M}/R)^{(q_1+q_2-1)(s-s_0)}\|u\|_{s,R}^{q_1+q_2}\,,
\end{equation}
\end{lemma}
\begin{proof}
Estimate \eqref{confort} follows by Lemma $3.4$ in \cite{feoMassJFA}. Then 
the bound in \eqref{confort2} follows by combining also \eqref{confort3}.
\end{proof}

%
%We shall use the following.
%\begin{corollary}\label{coro:nonline}
%Consider 
%\[
%g(u):=\omega^{-1/2}f\left(\omega^{-1/2}\left(\frac{u+\bar u}{\sqrt 2}\right)\right) \,,
%\]
%where $\omega$ is in \eqref{omegoneWave} and $f$ is given by \eqref{nonlionearitawave}.
%Then, for any $s\geq s_0>d/2$,  there exists an absolute constant  $M>0$ such that
%
%\begin{align}
%|g(u)|_{s+1}&\leq \mathtt{M}^{2q s}|u|_{s_0}^{2q}|u|_{s}\,,\label{stimaff1}
%\\
%|g(u)|_{s+1}&\leq \mathtt{M}^{2qs_0}(\mathtt{M}/R)^{2q(s-s_0)}|u|_{s}^{2q+1}\,,\label{stimaff2}
%\end{align}
%for all $u\in H^{s}(\T^{d};\C)$.
%\end{corollary}
%\begin{proof}
%The bounds \eqref{stimaff1}-\eqref{stimaff2} follows using that $|\omega^{-1/2}z|_{s+1/2}\lesssim|z|_{s}$
%for any  $z\in H^{s}(\T^{d};\C)$ and Lemma \ref{stimaNonlin} with $q_1+q_2=2q+1$.
%\end{proof}

%\part{$\mathtt{m}\neq 0$: BNF and Energy estimates on $\T^d$}
\part{\texorpdfstring{$\mathtt{m}\neq0$}{m≠0}: 
BNF and Energy estimates on 
\texorpdfstring{$\T^{d}$}{Td}}
Here we study the equation \eqref{eq:wave} both in the case $d=1$ and $d>1$
when the mass parameter $\mathtt{m}\neq0$.

\noindent
In section \ref{sec:normaldim1} we provide the proof of Theorems \ref{teonuovo}-\ref{teoBNFwave}
regarding the one dimensional case. The proof is based on a normal form approach and heavily relies on Hamiltonian formalism.
In section \ref{sec:preliminaryBNF} we introduce the spaces of Hamiltonians and the small divisors estimates
needed for the normal form procedure. Some technical results
are postponed in Appendix \ref{sec:birkoff}.

\noindent
In section \ref{sec:stimetutte} we give the proof of Theorem \ref{thm:mainNOBNF} 
regarding the high dimensional case.
The result is based on the a priori energy estimates, which for semantic reasons are given in subsection \ref{sec:kleintd}.

\section{Preliminaries}\label{sec:preliminaryBNF}
\subsection{Spaces of Hamiltonians}
\label{sec:2}
In this section we shall restrict the study to the case of $d=1$.
The stability result will be a consequence of a normalization procedure 
that amounts in transforming the Hamiltonian 
into a suitable normal form (the Birkhoff normal form). 
We will follow the setting and the notation of \cite{FeoMass:Beam, BMP:CMP}.
In order to set the convenient functional setting, 
we shall rather work in the space of sequences 
that correspond to the above functional spaces, 
by systematically identifying  $ L^2(\T,\C)$ with  the Banach space 
$\cF(\ell^2(\C))$ of $2\pi$-periodic functions (expanded as in \eqref{espFourier})
such that their Fourier's coefficients $(u_j)_{j\in\Z}\in \ell^2(\C)$. 
Then,  the  Hamiltonian  in \eqref{waveHam} (see also \eqref{nonlionearitawave})
reads 
\begin{equation}\label{waveHamFourier}
H(u,\bar{u})= 
\sum_{j\in\Z} \omega_{j}|u_{j}|^{2} + 
\sum_{\substack{j_i\in\Z,\s_i\in\{\pm\} 
\\
\sum_{i=1}^{q+2}\s_i j_i=0}} 
F_{j_1,\ldots,j_{q+2}}^{\s_1\ldots\s_{q+2}}
u_{j_1}^{\s_1}\cdots u_{j_{q+2}}^{\s_{q+2}}\,=: \sum_{j\in\Z} \omega_{j}|u_{j}|^{2} + \tH_{q+2}\,,
\end{equation}
%\begin{equation}\label{waveHamFourier}
%H(u,\bar{u})= 
%\sum_{j\in\Z} \omega_{j}|u_{j}|^{2} + \sum_{p=3}^{\infty}
%\sum_{\substack{j_i\in\Z,\s_i\in\{\pm\} \\
%\sum_{i=1}^{p}\s_i j_i=0}} 
%F_{j_1,\ldots,j_p}^{\s_1\ldots\s_p}
%u_{j_1}^{\s_1}\cdots u_{j_p}^{\s_p}\,=: \sum_{j\in\Z} \omega_{j}|u_{j}|^{2} + \tH_{\ge 3}\,,
%\end{equation}
%where we used the analyticity in the neighborhood 
%of the origin of $F$ to expand 
%the second integral in \eqref{waveHam}
for some coefficients 
$|F_{j_1,\ldots,j_{q+2}}^{\s_1\ldots\s_{q+2}}|
\lesssim C^{q+2}$ for some $C>0$, and,  
rearranging the sum in multi-index notation we set
\[
\tH_{q+2}(u,\bar u) = 
\sum_{\substack{\al,\bt\in\N^\Z\,, |\al| + |\bt|=q+2 \\ \sum_j j(\al_j - \bt_j)=0}} 
H_{\al,\bt} u^{\al}{\bar{u}}^{\bt}\,,
\]

\noindent
Accordingly, by defining 
\begin{align*}
d u_j = \frac{1}{\sqrt 2}\pa{d x_j+ \imm d y_j}\,,
\quad 
d \bar u_j = \frac{1}{\sqrt 2}\pa{d x_j- \imm d y_j}\,,
\\  
\frac{\partial}{\partial  u_j} =  
\frac{1}{\sqrt 2}\pa{\frac{\partial}{\partial x_j} 
- \imm \frac{\partial}{\partial y_j}}\,,
\quad  
\frac{\partial}{\partial  \bar u_j} =  
\frac{1}{\sqrt 2}\pa{\frac{\partial}{\partial x_j} 
+ \imm \frac{\partial}{\partial y_j}}\,,
\end{align*}
the corresponding $2$-form and Hamiltonian vector field read
\begin{equation}\label{symp form seq}
\imm \sum_{j\in\Z} du_j\wedge d\bar{u}_j\,, 
\quad \quad 
X_H^{(j)} = -\imm \frac{\partial }{\partial{\bar{u}_j}}H(u) \,.
\end{equation}
Note that when some real analytic $H$ admits 
a holomorphic extension $\widehat H$ 
on some ball of radius $r>0$ in $\ell^2(\C)$ that is  
\[ 
(u_+, u_{-})\in B_r(\ell^2(\C))\times B_r(\ell^2(\C)) \to \widehat H(u_+,u_-)\, 
:\quad \quad 
H(u)=\widehat H (u,\bar u)\,,
\]  
then it admits a Taylor expansion
\[
\widehat H(u_+,u_-)  = \sum^\ast_{\al,\bt\in\N^\Z} H_{\al,\bt}u_+^\al u_-^\bt\,,
\]
where we denote by $\sum^\ast$ the sum 
restricted to those $\al,\bt: |\al|+ |\bt|<\infty$.

\noindent
One can see that
\[
\frac{\partial}{\partial \bar u_j} H(u) 
= 
\frac{\partial \widehat H(u_+,u_-)}{\partial u_{-,j}} \Big\vert_{u_+=\bar u_-=u}\,.
\]

%\noindent
%{\bf Spaces of sequences.}
From now on we shall pass to the Fourier side and work on spaces of weighted sequences.
%Let  $\mathtt w=(\mathtt w_j)_{j\in\Z}$ be  the  real sequence  
%\begin{align}
%&\tw=\tw(p):= \pa{\lfloor j\rfloor^{ p}}_{j\in\Z}\,,
%\label{peso sub}
%\end{align}
Recalling \eqref{normaJapJap}-\eqref{def:japjapModificatoRR2}
%\eqref{normnorma}-\eqref{es:fgrowth} 
we consider the Hilbert space
\begin{equation}\label{pistacchio}
\th_{p}:=
\Big\{u:= \pa{u_j}_{j\in\Z}\in\ell^2(\C)\; : \; \|u\|_{p}^2:= 
\sum_{j\in\Z} \mathtt \lfloor j\rfloor^{2p} \abs{u_j}^2 < \infty
\Big\}\,,
\end{equation} 
endowed with the scalar product
\begin{equation}\label{scalarW}
(u,v)_{\mathtt{h}_{p}}:=\sum_{j\in\Z} \mathtt \lfloor j\rfloor^{2p}u_j \bar v_j\,,\qquad u,v\in\th_{p}\,,
\end{equation}
and where $\|\cdot\|_{p}$ is in \eqref{normaJapJap}.
Moreover, given $r>0$, we denote by $B_r(\th_{p})$ 
the closed
ball of radius $r$ centred at the origin of $\th_{p}$.

\vspace{0.5em}
\noindent
In the following we shall systematically identify $2\pi$-periodic 
functions with their Fourier coefficients, writing $\mathtt{h}_{p}$ instead of $H^{p}$.

The space $\mathtt{h}_{p}$ satisfies an \emph{algebra} property in the following sense. 
Let
$\star:\th_{p} \times \th_{p} \to \th_{p}$ be the convolution operation defined as
\[
\pa{f,g} \mapsto f\star g :=\Big(
\sum_{j_1+j_2=j
} 
f_{j_1}g_{j_2}\Big)_{j\in \Z}\,.
\]
The map $\star: \pa{f,g} \mapsto f\star g$ is continuous in the following sense:
\begin{lemma}\label{algebra}
For $p>1/2$ we have 
\begin{align}
\|f\star g\|_{p}&\leq \CalgM \|f\|_{p} \|g\|_{p}\,,
\qquad\qquad \forall f,g\in \th_{p}\,,\label{stimacalgcalgMM}
\end{align}
where 
\begin{equation}\label{costanti algebra}
\CalgM:= \sqrt{2}\sqrt{2 + \frac{2p+1}{2p-1}}\,.
\end{equation}
\end{lemma}

\begin{proof}
It follows by Lemma $5.5$ in \cite{BMP:CMP}.
\end{proof}

\vspace{0.5em}
By endowing the space $\th_{p}$ in \eqref{pistacchio}
with the symplectic structure induced by the symplectic form 
$\Omega$ in \eqref{symp form seq}, 
we introduce the following class of Hamiltonians.

\begin{definition}\label{def:admiHam}
Let $r>0$ and consider a Hamiltonian
$ H : B_r(\th_{p}) \to \R$
such that there exists a pointwise  
absolutely convergent power series expansion\footnote{As usual given 
a vector $k\in \Z^\Z$, $|k|:=\sum_{j\in\Z}|k_j|$.}
\begin{equation}\label{HamPower}
H(u)  = 
\sum_{\substack{
\al,\bt\in\N^\Z\,, \\2\leq |\al|+|\bt|<\infty} }
\!\!\!
H_{\al,\bt}u^\al \bar u^\bt\,,
\qquad
u^\al:=\prod_{j\in\Z}u_j^{\al_j}\,.
\end{equation}

\noindent
$\bullet$ {\bf (Admissible Hamiltonians).}
We say that $H$ as in \eqref{HamPower}
is \emph{admissible} if  the following properties hold: 
\begin{enumerate}
\item \emph{Reality condition}:
\begin{equation}\label{real}
H_{\al,\bt}= \overline{ H}_{\bt,\al}\,,\qquad \forall\, \al,\bt\in\mathbb{N}^{\Z}\,;
\end{equation}
\item \emph{Momentum conservation}:	
\begin{equation}\label{momento}
H_{\al,\bt}\neq0\quad \Rightarrow\quad
\pi(\al - \bt) := \sum_{j\in\Z}j\pa{\al_j - \bt_j}= 0\,.
\end{equation}
\end{enumerate}	

\noindent
We denote by 
$\mathcal{A}_{r}(\th_{p})$ the the space of admissible
 Hamiltonians such that the \co{majorant}
\begin{equation}\label{etamag}
\und { H} (u):= 
\sum_{(\alpha,\beta)\in\cM} \abs{{H}_{\alpha,\beta}}u^{\alpha}\bar{u}^{\beta}
\end{equation} 
is point-wise  absolutely convergent on $B_r(\th_{p})$,
where we set
\begin{equation}\label{mass-momindici}
\mathcal{M}:=
\left\{
(\alpha,\beta)\in \mathbb{N}^{\Z}\times\N^{\Z} : \pi(\al - \bt)= 0, 
|\alpha|+|\beta|<\infty
\right\}\,,
\end{equation}

\smallskip
\noindent
$\bullet$ {\bf (Regular Hamiltonians).} 
We say that $H$ as in \eqref{HamPower}
is \emph{regular} if  the following properties hold: 
\begin{enumerate}
\item $H$ belongs to $\mathcal{A}_{r}(\th_{p})$;

\item one has
\begin{equation}\label{normaHamilto}
|H|_{r,p}
:=
r^{-1} \Big(
\sup_{\norma{u}_{p}\leq r} 
\norma{{X}_{{\underline H}}}_{p} 
\Big) < \infty\,.
\end{equation}
\end{enumerate}	
 We denote by $\cH_{r}(\th_{p})$ the space of regular Hamiltonians.
% 
% \smallskip
% $\bullet$ {\bf (Scaling degree).} Given $\td \in\N$, let  $\cH^{(\td)}$  be the vector space 
%of homogeneous polynomials of degree $\td + 2$, 
%that is admissible  Hamiltonians of the form
%\[
%\sum_{\substack{(\al,\bt)\in\mathcal{M} \\ |\al| + |\bt| = \td + 2}} H_{\al,\bt} u^{\al} \bar{u}^{\bt}\,.
%\]
%We shall say that a regular Hamiltonian $H$ has 
%\emph{scaling degree} $\td=\td(H)$ if  $H\in \cH_{r}(\th_{\mathtt w})\cap\cH^{(\td)}$.
\end{definition}

\begin{remark}
Note that, given two admissible Hamiltonians $H,G$,
the Poisson brackets in \eqref{Poissonbrackets} read 
\begin{equation}\label{poipoisson}
\{H,G\}={\rm i}
\sum_{j\in\Z}
\big(\partial_{u_j}G\partial_{\bar{u}_j}H-
\partial_{\bar{u}_j}G\partial_{u_j}H\big)\,.
\end{equation}
\end{remark}

\begin{remark}
\label{embeddignsspaces}
We remark the following facts:

\noindent
$\bullet$
%Given two positive sequences $\tw = \pa{\tw_j}_{j\in\Z},\tw' = (\tw'_j)_{j\in\Z}$
%we write that $\tw\leq \tw'$ if the inequality holds
%point wise, namely
%\[
%\tw\leq \tw' \quad
%\iff\quad
%\tw_j\leq \tw'_j\,,\ \ \ \forall\, j\in\Z\,.
%\]
If $r'\le r$ and $p\leq p'$ 
then $B_{r'}(\th_{p'}) \subseteq B_r(\th_{p})$.

\noindent
$\bullet$ If a Hamiltonian $H$ satisfies \eqref{real},  it means that it 
is real analytic in the real and imaginary part of $u$.

\noindent
$\bullet$ If a Hamiltonian $H$ satisfies \eqref{momento} then it 
Poisson commutes with $\sum_{j\in\Z} j\abs{u_j}^2$.

\noindent
$\bullet$ The Hamiltonian functions being defined modulo a constant term, 
we shall assume without loss of generality that $H(0)=0$. 
\end{remark}
Finally,  let us consider  a regular Hamiltonian 
$S\in\cH_r(\th_p)$ and  its flow $\Phi_{S,t}$
which is well-defined (see Lemma \ref{ham flow} for details), 
and let 
\begin{equation}\label{quadraticWave}
D_{\omega}:=\sum_{j\in\Z}\omega_{j}|u_j|^{2}\,,
\end{equation}
and its flow $\phi_{\omega, t}$,  
where $\omega_j = \sqrt{j^2 + \mathtt{m}}$.

\begin{definition} \label{def:adjoperator}
$(i)$ The \emph{Lie derivative} of $H$ along the flow of $S$ is given by 
\begin{equation}\label{liederiv}
L_S H =  \frac{d}{dt}_{|t=0} \phi_{S,t}^* H(u) = \frac{d}{dt}_{|t=0} H(\phi_{S,t}(u))\,.
\end{equation}

\noindent
$(ii)$ Given $H\in \cH_{r}(\th_{p})$  we define the \emph{adjoint action} 
of the Hamiltonian $D_{\omega}$ 
as the Lie derivative operator 
\begin{equation}\label{def:adjaction}
L_{\omega} H:= \frac{d}{dt}_{|t=0} \phi_{\omega,t}^* H 
= 
\sum_{(\al,\bt)\in \mathcal{M}} 
-{\rm i}\big(\omega \cdot(\al-\bt)\big)H_{\al,\bt}u^{\al}\bar{u}^{\bt}\,,
\end{equation}
where $\mathcal{M}$ is the set of indexes defined in \eqref{mass-momindici}.
\end{definition}

\begin{remark}{\bf (Change of variables).}\label{rmk:change}
Along the paper we shall study how a Hamiltonian 
$H$ behaves  along the flow of 
a given regular Hamiltonian $S$.
In fact one has
\[
\begin{aligned}
\frac{d}{dt} H(\phi_{S,t}(u))=dH(\phi_{S,t}(u))\cdot X_{S}(\phi_{S,t}(u))
&=
\Omega(X_{H}(\phi_{S,t}(u)), X_{S}(\phi_{S,t}(u)))
\\&
=\{H,S\}\circ \phi_{S,t}(u)
\stackrel{\eqref{liederiv}}{=}(L_{S}H)\circ \phi_{S,t}(u)\,.
\end{aligned}
\]
Then \eqref{liederiv} corresponds to $L_{S}H=\{H,S\}$.
\end{remark}

\noindent
In our work the crucial point is 
that all the dependence on the parameters $r,p$ 
of the  norm in \eqref{normaHamilto}  can be {\it encoded}
in the coefficients
\begin{equation}\label{coeffSobo}
c^{(j)}_{r,p}(\al,\bt)=
r^{|\al|+|\bt|-2}\frac{\lfloor j\rfloor^{ 2p}}{\prod_{i\in\Z}\lfloor i\rfloor^{p(\al_i+\bt_i)} }\,.
\end{equation}
defined for any $\al,\bt\in\N^\Z$ and $j\in\Z$.

To formalize such \emph{encoding} we reason as follows.
For any $H\in \cH_{r}(\th_{p})$ we define  a map
\[
B_1(\ell^2)\to \ell^2 \,,\quad y=\pa{y_j}_{j\in \Z }\mapsto 
\pa{Y^{(j)}_{H}(y;r,p)}_{j\in \Z}
\]
by setting
\begin{equation}\label{giggina}
Y^{(j)}_{H}(y;r,p) := 
\sum_{(\al,\bt)\in\mathcal{M}} |H_{\al,\bt}| \frac{(\al_j+\bt_j)}{2}c^{(j)}_{r,p}(\al,\bt) y^{\al+\bt-e_j}
\end{equation}
where  $e_j$ is the $j$-th basis vector in $\N^\Z$, while the coefficient
$c^{(j)}_{r,p}(\al,\bt)$ is defined right above in \eqref{coeffSobo}.
The following properties give a systematic way for 
computing the norm of a given Hamiltonian 
and its relation w.r.t. another one.
We have the following.
\begin{lemma} \label{norme proprieta} 
Let $\ri, r'>0$, $p,p'>1/2$. % $\twi,\twf\in \R_+^\Z$. 
The following properties hold.

\vspace{0.3em}
\noindent
$(1)$ The norm of $H$ can be expressed as
\begin{equation}\label{ypsilon}
\abs{H}_{r,p}= 
\sup_{|y|_{\ell^2}\le 1}\abs{Y_H(y;r,p)}_{\ell^2}\,.
\end{equation}

\vspace{0.3em}
\noindent
$(2)$ Given $H^{(1)}\in \cH_{r',p'}$ and $H^{(2)}\in \cH_{\ri,p}$\,,
such that $\forall\, \al,\bt\in \N^\Z$ and  $j\in \Z$ with $\al_j+\bt_j\neq 0$ 
one has
\begin{equation}\label{alberellobello}
|H^{(1)}_{\al,\bt}| c^{(j)}_{\rf,p'}(\al,\bt)  \le 
c|H^{(2)}_{\al,\bt}| c^{(j)}_{\ri,p}(\al,\bt),
\end{equation}
for some $c>0$, then
\[
|H^{(1)}|_{\rf,p'}
\le c |H^{(2)}|_{\ri,p}\,.
\]

\vspace{0.3em}
\noindent
$(3)$ {\bf (Immersion).} For any $p>1$ the norm $\norm{\cdot}_{r,p}$ 
is monotone increasing in $r$. Moreover, letting $r>0$,
one has, 
for any $p' >0$, $p>1$,
\begin{equation}\label{emiliapara2}
|H|_{r,p+p'} \le |H|_{r,p}\,.
\end{equation}

\vspace{0.3em}
\noindent
$(4)$ {\bf (Poisson Brackets).}
For $0 <\rho\leq r$  we have
\begin{equation}\label{commXHK}
|\{F,G\}|_{r,p}
\le 
4\Big(1+\frac{r}{\rho}\Big)
|F|_{r+\rho,p}
|G|_{r+\rho,p}\,.
\end{equation}
\end{lemma}
\begin{proof}
Items $(1), (2)$ follow by 
 Lemma $3.1$ in \cite{BMP:CMP}. 
%(see also Lemmata $3.3$, $3.4$ and $A.1$ in \cite{ProStolo}).
Item $(3)$ is a consequence of Proposition $6.3$ in \cite{BMP:CMP} (see also Prop. $2.9$ in \cite{FeoMass:Beam})
Item $(4)$ follows by Proposition $2.1$ in \cite{BMP:CMP}.
%and Lemma $3.5$ in \cite{ProStolo}.
\end{proof}
\noindent
The following Lemma guarantees 
that the flow of a regular Hamiltonian is well-posed on $\th_{p}$. 
Moreover it shows how regular Hamiltonians changes 
under conjugation through flows.

\begin{lemma}{\bf (Hamiltonian flow).}\label{ham flow}
Let $0<\rho< r $,  and $S\in\cH_{r+\rho}(\th_{p})$ with 
\begin{equation}\label{stima generatrice}
\abs{S}_{r+\rho,p} \leq {\delta}:= \frac{\rho}{8 e\pa{r+\rho}}\,, 
\end{equation} 
Then the time $1$-Hamiltonian flow 
$\Phi^1_S: B_r(\th_{p})\to B_{r + \rho}(\th_{p})$  
is well defined, analytic, symplectic with
\begin{equation}\label{pollon}
\sup_{u\in  B_r(\th_{p})} \norm{\Phi^1_S(u)-u}_{\th_{p}}
\le
(r+\rho)  \abs{S}_{r+\rho,p}\leq \frac{\rho}{8 e}\,.
\end{equation}
Moreover, for any $H\in \cH_{r+\rho}(\th_{p})$ we have that
$H\circ\Phi^{1}_S= e^{ L_{S} } H\in\cH_{r}(\th_{p})$ 
and
\begin{align}\label{tizio}
\abs{\es H}_{r,p} & \le 2 \abs{H}_{r+\rho,p}\,,
\end{align}
and for any $h\in\N$ and any sequence  
$(c_k)_{k\in\N}$ with $| c_k|\leq 1/k!$, we have 
\begin{equation}\label{brubeck}
\abs{\sum_{k\geq h} c_k L^k_S\pa{H}}_{r,p} \le 
2 |H|_{r+\rho,p} \big(|S|_{r+\rho,p}/2\delta\big)^h\,.
\end{equation}
%where  $\ad_S\pa{\cdot}:= \set{S,\cdot}$.
\end{lemma}

\begin{proof}
Follows verbatim by Lemma $2.1$ in \cite{BMP:CMP} with $\eta=0$
and ${\rm ad }_{S} \rightsquigarrow L_{S}$.
\end{proof}

\subsubsection{Graded Poisson structure and conjugations}
We start by defining a degree decomposition  which endows $\cH_{r}(\th_{p})$ 
with a graded Poisson algebra structure.
\begin{definition}{\bf (Scaling degree).}\label{def:scalingdegree}
Given $\td \in\N$, let  $\cH^{(\td)}$  be the vector space 
of homogeneous polynomials of degree $\td + 2$, 
that is admissible  Hamiltonians of the form
\[
\sum_{\substack{(\al,\bt)\in\mathcal{M} \\ |\al| + |\bt| = \td + 2}} H_{\al,\bt} u^{\al} \bar{u}^{\bt}\,.
\]
We shall say that a Hamiltonian $H$ has 
\emph{scaling degree $\ge\td=\td(H)$} if 
\[
H\in\cH^{(\ge\td)} = \cH^{(\td)} \oplus_{h>\td} \cH^{(h)}\,.
\]
Accordingly, we shall define projections 
associated with this direct sum decomposition and write
\begin{equation}
\label{proietto}
\Pi^{(\td)} H = \sum_{
\substack{(\al,\bt)\in\mathcal{M}\\
|\al| + |\bt| = \mathtt{d} + 2}} 
H_{\al,\bt}u^{\al} \bar{u}^{\bt}\,,
\quad \quad 
\Pi^{(>\td)} H = \sum_{
\substack{(\al,\bt)\in\mathcal{M} \\ |\al| + |\bt| > \mathtt{d} + 2}} 
H_{\al,\bt}u^{\al} \bar{u}^{\bt}\,.
\end{equation}
We say that $\td(0)=+\infty$.
\end{definition}

\begin{remark}
With this definitions, quadratic Hamiltonians have scaling degree $0$. 
Essentially  $H$ has scaling degree $\td$ if and only if  
it has a zero of order $\td+2$ at zero.
\end{remark}
Definition \ref{def:scalingdegree} produces a graded Poisson algebra structure. 
Moreover one has the following result.

\begin{lemma}\label{gasteropode} 
The projection operators are continuos. In particular, the following hold.

\noindent
$(i)$ If $H\in\cH_r(\th_p)$ with $\td(H)=\td$, then one has 
\begin{equation}\label{bound proiezione}
\abs{\Pi^{(\td)} H}_{r,p} \le \abs{H}_{r,p}\,, \qquad \abs{\Pi^{(>\td)} H}_{r,p} \le \abs{H}_{r,p}\,.
\end{equation}

\noindent
$(ii)$	If $H\in \cH_{r}(\th_p)$ with $\td(H)\geq\td$, then for all $\rf\le r$ one has
\begin{equation*}
\abs{H}^{\wc}_{\rf,p} \le \pa{\frac{\rf}{r}}^{\td} \abs{H}^{\wc}_{r,p}\,.
\end{equation*}
\end{lemma}

\begin{proof} 
See Lemma $2.13$ in \cite{FeoMass:Beam}.
\end{proof}

\begin{remark}\label{rmk:scalaPoi}
If $F$ and $G$ are Hamiltonians in $\cH_{r}(\th_{p})$
with \emph{scaling degree} $\mathtt{d}_1, \mathtt{d}_2$ respectively, then
the Poisson $\{F,G\}$ has scaling degree equal to 
$\mathtt{d}_1+\mathtt{d}_2$. 
In general, if the scaling degrees are $\ge \mathtt{d}_1, \mathtt{d}_2 $, 
then the scaling degree of  $\{F,G\}$ is $\ge \mathtt{d}_1 + \mathtt{d}_2$.
\end{remark}

\subsubsection{Resonant Hamiltonians.} 
We define the \emph{resonant} subset of $\mathcal{M}$ (see \eqref{mass-momindici})  as
\begin{equation}\label{resonant set}
\mathtt{R} = \set{(\al,\bt) \in \mathcal{M}\,:\, 
\al_j = \bt_j \, \vee \al_j = \bt_{-j} \, 
\forall j\in\Z }\,,
\end{equation}
and we denote by $\cK_r(\th_\tw)$  the subset of \emph{resonant} Hamiltonians, i.e.
\begin{equation}\label{ker}
\cK_r(\th_p) = \Big\{H\in\cH_{r}(h_{p})\, : 
\sum_{(\al,\bt) \in\mathtt{R}} H_{\al,\bt} u^{\al}\bar{u}^{\bt}\Big\}\,.
\end{equation}

\begin{remark}\label{rmk:resonant set}
Let $(\al,\bt)\in \mathtt{R}$ and $\ell=\al-\bt$. The condition in \eqref{resonant set}
implies that 
\[
\ell_{j}+\ell_{-j}=\al_j-\bt_j+\al_{-j}-\bt_{-j}\equiv0\,,\qquad \forall\, j\in\Z\,.
\]
\end{remark}

\noindent
The following results regards a fundamental properties of resonant Hamiltonians.
\begin{lemma}{\bf (Flows of Kernel Hamiltonians).}\label{lem:kernel}
Let 
\[
\tf: \th_p \to \th_p\,, \quad u\mapsto \tf(u) =  (\tf_{j} u_j)_{j\in\Z}\,,
\qquad \tf_{j} = \tf_{-j}\,,\qquad \forall j\in\Z\,.
\]
 Then any $H\in\cK_r(\th_p)$ poisson commutes with $\norma{\tf(u)}^2_p$. 
%In particular the norm of Hamiltonians in $\cK_r$ is constants along their flows. 
\end{lemma}

\begin{proof}
See Lemma $2.18$ in \cite{FeoMass:Beam}.
\end{proof}
The following classical Lemma gives \emph{a priori} estimates 
on the time of definition of  flows generated by a wider class of Hamiltonians.
\begin{lemma}\label{tempotempo}
Let $\mathcal{N}\in \mathcal{A}_{r}(\th_p)$ and $R\in \mathcal{H}_{r}(\th_{p})$
(recall Def. \ref{def:admiHam}) for some $r>0$.
Assume that 
\begin{equation}\label{hyp:kernel}
{\rm Re}(X_{\mathcal{N}}(v),v)_{\th_p}=0\,,\;\;\;\forall\, v\in \th_{p}\,. 
\end{equation}
Consider the dynamical system
\[
\dot{v}=X_{\mathcal{N}}(v)+X_{R}(v)\,,\;\;\;\; v(0)=v_0\,,\quad \|v_0\|_{p}\leq \frac{3}{4}r\,.
\]
Then one has
\[
\big|\|v(t)\|_{p}-\|v_0\|_{p}\big|\leq \frac{r}{8}\,,
\quad \forall \, |t|\leq \frac{1}{8|R|_{r,p}}\,.
\]
\end{lemma}
\begin{proof}
See Lemma $5.4$ in \cite{BMP:CMP}.
\end{proof}

\begin{remark}\label{rmk:kernel}
Let $\mathcal{N}\in \mathcal{K}_{r}(\th_{p})$. %Recalling \eqref{poipoisson}, \eqref{scalarW}
%(see also \eqref{symcompform}, \eqref{Poissonbrackets}), w
We have that
\[
2{\rm Re}(X_{\mathcal{N}}(v),v)_{\th_{p}}=\{\mathcal{N}(v), \|v\|^{2}_{{p}}\}=0\,,
\]
by Lemma \ref{lem:kernel}.  Hence we deduce that
the vector field $X_{\mathcal{N}}$ satisfies the condition \eqref{hyp:kernel}
in Lemma \ref{tempotempo}.
\end{remark}

\begin{remark}\label{oddKerham}
Let $H\in \mathcal{K}_r(\th_\tw)$ and assume $H=\Pi^{(\td)}H$
for some $\td\geq1$. 
By using \eqref{resonant set}, \eqref{ker}
one can check that $H\equiv0$ if $\td$ is \emph{odd}.
\end{remark}

\subsection{Small divisors}\label{sec:small}
One can easily note that the  action of $L_{\omega}$ in \eqref{def:adjaction} 
 on homogeneous Hamiltonians in $\cH^{(\td)}$ (see Def. \ref{def:scalingdegree})
is \emph{diagonal}
with eigenvalues given by (recall \eqref{mass-momindici})
\[
{\rm i} \omega\cdot\ell\,,\qquad \ell=\alpha-\beta\,,\quad (\alpha,\beta)\in \mathcal{M}\,,\;\;\;|\al| + |\bt| = \td + 2\,,
\]
where the vector  $\omega\in\mathbb{R}^{\mathbb{Z}}$ 
is defined as (recall \eqref{omegoneWave})
\begin{equation}\label{dispLawWave}
\begin{aligned}
\omega&:=\omega(\mathtt{m}):=(\omega_{j})_{j\in \mathbb{Z}}\in \mathbb{R}^{\mathbb{Z}}\,,
\\
\omega_{j}&:=\omega_{j}(\mathtt{m}):
=\sqrt{|j|^{2}+\mathtt{m}}\,,\qquad j\in \mathbb{Z}\,,\qquad \mathtt{m}>0\,.
\end{aligned}
\end{equation}
Consider the non-resonant set of indexes 
\begin{equation}\label{restrizioni indici}
\Lambda:=\big\{\ell\in \mathbb{Z}^{\mathbb{Z}}\; : \; 
\ell:=\alpha-\beta\,,\; \forall (\alpha,\beta)\in \mathtt{R}^{c}\big\}\,,
\end{equation}
where $\mathtt{R}$ in \eqref{resonant set}.
Given a vector $\ell:=(\ell_i)_{i\in \mathbb{Z}}\in \Lambda$
we define  the set
$\mathcal{A}(\ell):=\{i\in \mathbb{Z} \,:\, \ell_i\neq0\}$
and the map
\begin{equation}\label{polloarrosto}
\ell\mapsto \mathtt{d}:=\mathtt{d}(\ell)\in \mathbb{N}
\end{equation}
where $\mathtt{d}(\ell):=\#\mathcal{A}(\ell)$. We call $\mathtt{d}(\ell)$ the \emph{cardinality} of $\ell$,
i.e. the number of components of $\ell$ which are different form zero.
Finally we give the following Definition.
\begin{definition}\label{n staremmino}
Consider a vector $v=\pa{v_i}_{i\in \Z}$  $v_i\in \N$, $|v|<\infty$. 
We define the vector $m=m(v)$ as the reordering of the elements of the set
\[
\set{j\neq 0 \,,\quad \mbox{repeated}\quad  \abs{u_j} \;\mbox{times}}\,,
\]
where $D<\infty$ is its cardinality, such that
$|m_1|\ge |m_2|\ge \dots\geq |m_D|\ge 1$. 	
\end{definition}

The key property of $\omega$ is that  that, for ``almost all'' choices of the 
parameter $\mathtt{m}$, the frequency vector $\omega$
in \eqref{dispLawWave} satisfies  diophantine-like estimate.

\begin{theorem}\label{thm:mainVERO}
There exists a positive measure set $\mathfrak{M}_{\gamma}\subseteq [1,2]$
such that
for any $\mathtt{m}\in\mathfrak{M}_{\gamma}$, 
 the vector $\omega(\mathtt{m})$ %defined in \eqref{dispLawWave}
satisfies the following: for any 
$\ell=\alpha-\beta$, $(\alpha,\beta)\in \Lambda$
one has 
\begin{equation}\label{goodsmalldivTeorema}
|\omega\cdot\ell|\geq  \gamma^{{\tau^2(\ell)}}
\frac{1}{(\mathtt{C}(|\alpha|+|\beta|))^{6\tau(\ell)}}
\prod_{\substack{n\in\mathbb{Z} \\ n\neq m_1(\ell),m_2(\ell)}} 
\frac{1}{(1+|\ell_{n}|^{2}\langle n\rangle^{2})^{{4\tau^2(\ell)}}}\,,
\end{equation}
where $\tau(\ell):=\td(\ell)(\td(\ell)+4)$, % (see \eqref{polloarrosto}), 
and $\mathtt{C}$ 
is a 
positive pure constant  large enough.
Moreover,  there exists a positive constant $\mathtt{c}$ such that
\[
\meas([1,2]\setminus\mathfrak{M}_{\gamma})\leq \mathtt{c}\gamma\,.
\]
\end{theorem}
\begin{proof}
For the proof we refer the reader to \cite{FeoMass:misurewave}.
\end{proof}
\begin{remark}\label{rmk:ker2}
By Remark \ref{rmk:resonant set}  for any $(\al,\bt)\in \mathtt{R}$
and $\ell=\al-\bt$ one has that $\omega\cdot\ell\equiv0$ is identically zero  for $\mathtt{m} \in[1,2]$.
On the other hand, by Theorem \ref{thm:mainVERO}, for any $\omega$ belonging to the 
``full measure'' set $\mathfrak{M}_{\gamma}$,
 one has $\omega\cdot \ell\neq0$ for any $\ell\in\Lambda$. % (see \eqref{restrizioni indici}).
\end{remark}

\section{The one dimensional case: a normal form approach.}\label{sec:normaldim1}

The aim of this section is to prove Theorems \ref{teonuovo}-\ref{teoBNFwave}.
The key point is to prove an abstract normal form result for Hamiltonian functions
introduced in section \ref{sec:2}, which is adapted to the Klein-Gordon equation 
\eqref{eq:waveComp}.

More precisely we first prove the following.
Let $\bar{r},s_0,p>0$, $0<\gamma<1$, fix a natural number $\tK\ge 1$.
Consider a Hamiltonian $H$ such that $H - D_\omega\in\cH_{\bar{r}}(\th_{p})$ 
that satisfies
\begin{equation}\label{hamIniz}
\begin{aligned}
&H:=D_{\omega}+R_0\,,\qquad R_0=\sum_{\td=1}^{\tK}R_0^{(\td)} + R_0^{(\geq\tK+1)}\,,
\\&
R_0^{(\td)}\in \mathcal{H}^\wc_{\bar{r}}(\th_{p})\cap\mathcal{H}^{(\td)}\,,
\qquad
R_0^{(\geq\tK+1)}\in \mathcal{H}^\wc_{\bar{r}}(\th_{p})\cap\mathcal{H}^{(\geq\tK+1)}\,,
\end{aligned}
\end{equation}
where $D_{\omega} $ is in \eqref{quadraticWave}, 
$\omega\in \mathfrak{M}_{\gamma}\subseteq [1,2]$ (see Theorem \ref{thm:mainVERO}). 
Consider the constant $\tC>0$ provided by Proposition \ref{shulalemma}
and define 
\begin{equation}\label{conditionRstar}
\begin{aligned}
\bar{r}_0 &:= \min \set{
\bar r, (\frac{4^{\tK+3}|R_0|_{\bar{r},p}}{\bar{r}}  J_{\tK}
32e\tK)^{-1}}\,,
\quad 
J_{\tK}:=\gamma^{-\mathtt{C}\tK^4} \exp\big(\mathtt{C} 2^{12} \tK^5\big)\,.
\end{aligned}
\end{equation}

\begin{remark}\label{rescalingHaminiziale}
Without loss of generality we can always assume that $|R_0|_{\bar{r},p}\leq1$.
Indeed if $|R_0|_{\bar{r},p}>1$ one can choose
$\widetilde{r}<\bar{r}$ such that (recall Lemma \ref{gasteropode})
\[
|R_0|_{\widetilde{r},p}\leq \left(\frac{\widetilde{r}\,}{\bar{r}}\right)
|R_0|_{\bar{r},p}\leq 1\,.
\]
\end{remark}
The main result of this section is the following.
\begin{theorem}{\bf (Birkhoff normal form).}\label{mainthm:BNF}
Consider $H$ in \eqref{hamIniz} .
Then, for any $0<r_0\leq \bar{r}_0$,
there exists a symplectic map
\begin{equation}\label{sonno1}
\begin{aligned}
&{\bf \Phi} \; : \; B_{\frac{r_0}{2}}(\th_{p+\zeta})\ \to\
B_{r_{0}}(\th_{p+\zeta})
\\&
\sup_{u\in  B_{\frac{r_0}{2}}(\th_{p+\zeta})} 
\norm{{\Phi}(u)-u}_{\th_{p+\zeta}}
\le
\tC_1 r_0^{2}\leq \frac{r_0}{8}\,,
\qquad \tC_1:= \frac{|R_0|_{\bar{r},p}}{\bar{r}} 
J_{\tK}\,,
\end{aligned}
\end{equation}
where   $\zeta=2^43^4\sum_{i=1}^{\tK}i^4$  
such that the following holds.
The Hamiltonian %(see \eqref{hamIniz})
\begin{equation}\label{sonno2}
\begin{aligned}
&H_{f}:=H\circ{\bf \Phi}:=
D_{\omega}+\mathfrak{Z}+\mathfrak{R}\,,
\\& 
\mathfrak{Z}\in  \cK^\wc_{\frac{r_0}{2}}(\th_{p+\zeta})\cap\mathcal{H}^{(\leq \tK)}\,,
\qquad
\mathfrak{R}\in  \cH^\wc_{\frac{r_0}{2}}(\th_{p+\zeta})\cap\mathcal{H}^{(\geq\tK+1)}
\end{aligned}
\end{equation}
satisfies
\begin{equation}\label{sonno3}
|\mathfrak{Z}|_{\frac{r_0}{2},p+\zeta}\leq \mathtt{C}_2r_0^2\,,
\qquad\;
|\mathfrak{R}|_{\frac{r_0}{2},p+\zeta}\leq\mathtt{C}_3 r_0^{\tK+1}
\end{equation}
with
\[
\mathtt{C}_2:=\frac{16 e\tK  \abs{R_0}^{\wc}_{\bar{r},p}  
4^{\tK+1}}{\bar{r}^2}J_{\tK}\,,
\qquad \;\;\;
\mathtt{C}_3:=
\frac{\abs{R_0}^{\wc}_{\bar{r},p} (16 e\tK   4^{\tK+2})^\tK}{\bar{r}^{\tK+1}}
(J_{\tK})^{\tK}\,.
\]
\end{theorem}

The remaining part of this section is devoted to the proof of the result above.

\subsection{Inverse of the adjoint action}\label{sec:homo}

In this section we provide estimates on the inverse of the adjoint action $L_{\omega}$ defined 
\eqref{def:adjaction} in the case $d=1$. The operator $L_\omega$ must be inverted at each iterative normalization step, so it is the core of the theorem above. 
We will use the estimates provided in section \ref{sec:small}.
In  view of Remark \ref{rmk:ker2} and 
by formula \eqref{def:adjaction} (see also \eqref{resonant set}-\eqref{ker}) we  deduce that
\[
L_{\omega}H=0 \qquad \Leftrightarrow\qquad  H\in \cK_r(\th_p)\,.
\]

\noindent
Hence the operator $L_{\omega}$ is formally invertible
when acting on the subspace  
\begin{equation}\label{range2}
\cR_r(\th_p)=\cK_r(\th_p)^{\perp}:=
\Big\{
H\in\cH_{r}(h_{p})\, : 
\sum_{(\al,\bt) \in\mathtt{R}^{c}} H_{\al,\bt} u^{\al}\bar{u}^{\bt}
\Big\}\,,
 \end{equation}
containing  those Hamiltonians supported on monomials
$u^{\alpha}\bar{u}^{\beta}$ with $(\alpha,\beta)\in \mathtt{R}^{c}$.
We decompose the space of regular Hamiltonians $\cH_r(\th_p)$
as
\[
\cH_r(\th_p) = \cK_r(\th_p)\oplus \cR_r(\th_p)\,,
\]
and we denote by $\Pi_{\cK}$ and $\Pi_{\cR}$
 the  continuous projections	 
 on the subspaces $\cK_r(\th_p)$, $ \cR_r(\th_p)$.
 One can note 
 \begin{equation}\label{fame}
|\Pi_{\cK}H|^\wc_{r,p},
 |\Pi_{\cR}H|^\wc_{r,p} 
 \leq
 |H|^\wc_{r,p}\,.
\end{equation}

\noindent 
Obviously, for {diophantine}  frequency,
$\cR^\wc_{r}(\th_p)$ and	 $\cK^\wc_{r}(\th_p)$
represent the range and kernel 
of  $L_\omega$ respectively.

\begin{proposition}{\bf (Inverse of the adjoint action).}\label{shulalemma}
Fix $\mathtt{N}\in \mathbb{N}$, $r>0$, $p>1$.
Consider a Hamiltonian function 
$f\in \mathcal{R}_r(\th_{p})\cap \mathcal{H}^{(\tN)}$ 
(see Def. \ref{def:scalingdegree} and recall \eqref{range2}). 
For any $\omega\in \mathfrak{M}_{\gamma}$ (see Thm. \ref{thm:mainVERO}) 
the following holds.
Fix $\zeta \geq{(2^43^4 \tN)^4}$. 
There exists an absolute constant $\tC>0$ such that  
\[
|L_{\omega}^{-1} f|_{r,p+\zeta}\leq J_0 |f|_{r,p}\,,
\]
where 
\begin{equation}\label{controlJ0caseSob}
J_0:=J_{0}(\zeta,\tN):= \, \g^{-\mathtt{C}\tN^4} e^{\tC\zeta}\,.
\end{equation}
\end{proposition}

\begin{proof}
By definition of the coefficients in  \eqref{coeffSobo} we have
\begin{equation*}
J_0 := 
\sup_{\substack{j\in\Z,\,  (\al,\bt)\in\Lambda 
\\ \al_j+\bt_j\neq 0 
\\ |\al-\bt|\leq \tN+2}}
\frac{c^{(j)}_{\ri,p + \zeta}(\al,\bt) }
{c^{(j)}_{\ri,p}(\al,\bt) |\omega\cdot (\al-\bt)|}\, 
= 
\sup_{\substack{j\in\Z,\,  (\al,\bt)\in\Lambda 
\\ \al_j+\bt_j\neq 0 
\\ |\al - \bt| \le \tN + 2}} 
\frac{ \lfloor j\rfloor^{2\zeta}}{\abs{\omega\cdot{\pa{\al - \bt}}}}
\prod_{i\in\Z}\lfloor i\rfloor^{-\zeta(\alpha_i+\beta_i)}.
\end{equation*}
First of all we recall that (see Thm. \ref{thm:mainVERO} )
\begin{equation}\label{verga1}
|\alpha|+|\beta|=\tN+2\,,\;\;\;\td(\ell)\leq 4\tN\,,\;\;\; 4\tau^2(\ell)\leq 2^43^4\tN^4\,,\;\;\;
\zeta\geq (2^43^4\tN)^4\,.
\end{equation}
By the  diophantine condition \eqref{goodsmalldivTeorema} we have
\[
J_0 \leq  \g^{-\tau^2} (\mathtt{C}(\tN+2))^{6\tau}
\sup_{\substack{j\in\Z,\,  (\al,\bt)\in\Lambda 
\\ \al_j+\bt_j\neq 0 
\\ |\al - \bt| \le \tN + 2}} 
\Big(\frac{\jjap{j}^2}{\prod_{i\in\Z}\jjap{i}^{\al_i + \bt_i}} \Big)^\zeta 
\prod_{i\in\Z}\big((1+|\al_i-\bt_i|^2)\langle i\rangle^2\big)^{\tau}\,.
\]
By Lemma \ref{lem:constance2SE} we only have 
to consider the case in which \eqref{divisor} holds. So by \eqref{verga1}
we can apply bound   \eqref{seisettedelta} in Lemma \ref{stimaSob:lem}
to obtain
\[
    J_0 \leq  \g^{-\tau^2} (\mathtt{C}(\tN+2))^{6\tau}2^{\zeta-1}6^\zeta\,.
\]
Therefore the estimate \eqref{controlJ0caseSob} 
follows taking $\mathtt{C}>0$ large enough, but independent of $\tN$.
\end{proof}

\begin{proof}[{\bf Proof of Theorem \ref{mainthm:BNF}.} ]
The proof follows immediately from the iterative scheme in  Appendix \ref{sec:birkoff} 
(see the iterative scheme in Lemma \ref{lem:iterazione}). Specifically,
by  \eqref{conditionRstar} and  \eqref{patata2}  we have that condition
\eqref{cond:small} is fulfilled.
So, in view of Lemma \ref{lem:iterazione}
and
\eqref{pollon4}  deduce (recall \eqref{parametri})
\begin{equation*}%\label{sepultura2}
\Phi_k \; : \; B_{r_k}(\th_{{p+\s_{\tK}}})\ \to\ 
B_{r_{k-1}}(\th_{{p+\s_{\tK}}})
\end{equation*}
with 
\begin{equation*}%\label{stimaPhiK}
\sup_{u\in  B_{r_k}(\th_{p+\s_{\tK}})} 	\norm{\Phi_k(u)-u}_{{p+\s_{\tK}}}
\le 
r_{k-1}\mathtt{R}_0\frac{1}{2^k} J_{\tK}^{\star}\e\,.
%({4^k} J_{\tK}\delta_0^{-1})^{k-1} 2^{k-1}
%.....
  % \ri \widetilde{J_0}^{\star}\abs{R^{(\tN)}}_{\ri,\twi}\,.
\end{equation*}

Setting ${\bf \Phi}:=\Phi_1\circ\cdots\circ\Phi_{\tK}$, the estimates above clearly imply
\[
{\bf \Phi} \; : \; B_{\frac{r_0}{2}}(\th_{p+\s_{\tK}})\ \to\
B_{r_{0}}(\th_{p+\s_{\tK}})\,.
\]
In particular, by writing
%which, adding and subtracting $\id$ to each $\Phi_i$ with $i = 1,\ldots \tK$ entails
$$
{\bf \Phi}  - \id = (\Phi_1 - \id)\circ\Phi_2\circ\cdots\circ \Phi_\tK + (\Phi_2 - \id) \circ \Phi_3 \circ\cdots \circ\Phi_\tK + \cdots + \Phi_\tK - \id\,.
$$
we get
\begin{equation*}
\sup_{u\in  B_{\frac{r_0}{2}}(\th_{p+\s_{\tK}})}  \norma{{\bf \Phi}(u) - u}_{p+\s_{\tK}} 
\le 
\mathtt{R}_0 J_{\tK}^{\star}\e \sum_{j=0}^{\tK-1} 
\frac{r_j}{2^{j+1}}\le r_0\mathtt{R}_0 J_{\tK}^{\star}\e \,,
\end{equation*}
which
implies \eqref{sonno1}
by using the definitions of $J_{\tK}$, $\e$ and $\mathtt{R}_0$,
the smallness assumption \eqref{conditionRstar} on $r_0$
and setting $\zeta=\s_{\tK}$.

The new Hamiltonian $H\circ{\bf \Phi}$ is equal to $H_{\tK}$
given in \eqref{hamk} with $k\rightsquigarrow \tK$.
We then set $\mathfrak{Z}:=Z_{\tK}$
and $\mathfrak{R}:=R_{\tK}$.
The \eqref{sonno2} follows.
The bounds \eqref{sonno3} follow by \eqref{smallnormk1}-\eqref{smallnormk3}
recalling \eqref{arancia1}, \eqref{scala0}, \eqref{patata2}.
\end{proof}

\subsection{Conclusions and proof of Theorems \ref{teonuovo}-\ref{teoBNFwave}}

 First of all recall that equation \eqref{eq:wave} can be written as in \eqref{eq:waveComp}.
Moreover (recall \eqref{beam5}), for any initial condition
$(\psi_0,\psi_1)\in H^{p+\frac{1}{2}}\times H^{s,p-\frac{1}{2}}$ satisfying \eqref{main:smallcondSob}
 and setting
 \[
 u_0:=\frac{1}{\sqrt{2}}\Big(\omega^{\frac{1}{2}}\psi_0+{\rm i} \omega^{-\frac{1}{2}}\psi_1\Big)\,,
 \]
 we note that 
 \begin{equation}\label{smallnessu0}
\|u_0\|_{p}\leq \delta\,.
\end{equation}
Assume now that  $u(t)$ is  a solution of \eqref{eq:waveComp}, with initial condition
$u(0)=u_0$,  
satisfying 
\begin{equation}\label{claimclaim}
\sup_{t\in[0,T_0]}\|u(t)\|_{p}\leq 2\delta 
\qquad {\rm for\;some}\qquad T_0>0\,.
\end{equation}
Hence  the solution $(\psi(t), \partial_t\psi(t))$ of \eqref{eq:wave}
with initial conditions $(\psi_0,\psi_1)$  satisfies the
\emph{a priori} bound
\[
\|\psi(t)\|_{p+1/2}+\|\partial_t\psi(t)\|_{p-1/2}\leq 4\|u(t)\|_{p}\leq 8\delta\,,
\qquad \forall\, t\in[0,T_0]\,,
\]
which implies \eqref{boundsol1Sob}.
Our aim is to apply Theorem \ref{mainthm:BNF} to show that actually 
 the claim \eqref{claimclaim}
on solutions of \eqref{eq:waveComp} with initial conditions satisfying \eqref{smallnessu0}
holds true over a time interval $[0,T_0)$ providing  
 suitable lower bounds on the lifespan $T_0>0$.

In order to apply our abstract Birkhoff normal form result we need some preliminary results.
More precisely we shall prove that the 
the Hamiltonian function $H$
in  \eqref{waveHam} (see also \eqref{waveHamFourier}) 
can be written in the form \eqref{hamIniz} with
\begin{equation}\label{hamRR00}
R_0:=\mathtt{H}_{q+1}:=
\int_{\mathbb{T}} F\Big( \frac{\omega^{-1/2}(u+\bar{u})}{\sqrt{2}}\Big)\ {dx}\,,
\end{equation}
where $F$ is the analytic function in \eqref{nonlionearitawave}.
This is the content of the following lemma.
\begin{lemma}\label{lem:applicoSub}
Let $R>0$  and  consider 
the Hamiltonian $R_0$ in \eqref{hamRR00}.
For any $p_0>1$ 
and for any $\bar{r}>0$
satisfying (recall \eqref{costanti algebra})
\[
\mathtt{C}_{{\rm alg}, \mathtt{M}}(p_0) \bar{r} <R\,,%\red{non serve a nulla}
\] 
one has that the Hamiltonian $R_0$ in \eqref{hamRR00}
belongs to the space $\cH_{\bar{r}}(\th_{p_0})$ of regular Hamiltonians
and 
\begin{equation}\label{stimaNEMSob}
|R_0|_{\bar{r},p_0}\leq 
 ( 2 \CalgM \bar{r})^{q}\bar{r}<+\infty\,.
%\frac{\mathtt{C}_{{\rm alg}, \mathtt{M}}(p_0)}{R} R^{2q+2} \bar{r}<+\infty\,.
\end{equation}
\end{lemma}

\begin{proof}
It follows using the analyticity 
of the function $F$ in \eqref{nonlionearitawave} and Lemma \ref{algebra}.
%and reasoning as in Proposition $5.2$ in \cite{BMP:CMP}.

\noindent
More precisely, recalling \eqref{beam5} we have that, setting $d:=q+2$,
\[
\begin{aligned}
R_0=\int_{\mathbb{T}}F\Big( \frac{\omega^{-1/2}(u+\bar{u})}{\sqrt{2}}\Big) dx
&= 
\frac{1}{d}\int_{\mathbb{T}} (\frac{\omega^{-1/2}(u+\bar{u})}{\sqrt{2}})^{d}  dx
\\&= 
\frac{1}{d 2^{d/2}} 
\sum_{k=0}^d 
\binom{d}{k} \int_{\mathbb{T}}(\omega^{-\frac{1}{2}}u)^k(\omega^{-\frac{1}{2}}\bar{u})^{d-k}dx \,,
\end{aligned}
\]
hence, passing to the Fourier basis, we get
\[
 R_0= \frac{1}{d 2^{d/2}}
 \sum_{k=0}^d \binom{d}{k} 
 \sum_{\sum_{i=1}^{k}j_i=\sum_{p=k+1}^{d}j_{p}}
 \!\!\!\!\!
 C^{(d)}(j_1,\ldots,j_{d})
 u_{j_1} \cdots  u_{j_k}\bar{u}_{j_{k+1}}\cdots\bar{u}_{j_{d}}
\]
where the coefficients $C^{(d)}(j_1,\ldots,j_{d})$ are symmetric in $j_i$, $i=1,\ldots,d$, 
and have the form
\begin{equation}\label{cddcdd}
C^{(d)}(j_1,\ldots,j_{d}):=\prod_{i=1}^{d}\omega^{-\frac{1}{2}}(j_{i})
\stackrel{\eqref{omegoneWave}}{=}\prod_{i=1}^{d}\frac{1}{\sqrt{|j_i|^{2}+\mathtt{m}}}\leq1\,.
\end{equation}
In view of \eqref{eq:waveComp} we now compute the first component of the
vector field $X_{R_0}$. We have
\[
\begin{aligned}
X_{R_0}^{(j)}&:=-{\rm i}\partial_{\bar{u}_{j}} R_0
=-{\rm i}
 \frac{1}{d 2^{d/2}}
\sum_{k=0}^d \binom{d}{k} (d-k) \mathtt{a}^{(d,k)}(j)\,,
\\
\mathtt{a}^{(d,k)}(j)&:=
\sum_{\sum_{i=1}^{k}j_i-\sum_{p=k+1}^{d-1}j_{p}=j}
 \!\!\!\!\!
 C^{(d)}(j_1,\ldots,j)
 u_{j_1} \cdots  u_{j_k}\bar{u}_{j_{k+1}}\cdots\bar{u}_{j_{d-1}}\,.
\end{aligned}
\]
where we used the symmetry of the coefficients  $C^{(d)}(j_1,\ldots,j)$.
Notice 
that (recall \eqref{normaHamilto})
\[
\|X_{\underline{R_0}}(u)\|_{p}\leq \|X_{\underline{R_0}}(\underline{u})\|_{p}\leq 
\frac{1}{d 2^{d/2}}
 \sum_{k=0}^d \binom{d}{k} (d-k) \|\mathtt{a}^{(d,k)}\|_{p}
\]
where $\underline{u}=(|u_{j}|)_{j\in\mathbb{Z}}$ and  for any $0\leq k\leq d$, we set
$\mathtt{a}^{(d,k)}:=(\mathtt{a}^{(d,k)}(j))_{j\in\mathbb{Z}}$.
Notice moreover that %(using \eqref{cddcdd})
\[
|\mathtt{a}^{(d,k)}(j)|\leq |(\underbrace{\underline{u}\star \ldots \underline{u}}_{k }\star 
\underbrace{\bar{\underline{u}}\star\ldots \star\bar{\underline{u}}}_{d-k-1})_{j}|\,,\qquad
\forall\, j\in\mathbb{Z}\,.
\]
Therefore, using estimate \eqref{stimacalgcalgMM} in Lemma \ref{algebra}, we obtain
\[
\begin{aligned}
\|X_{\underline{R_0}}(u)\|_{p}&\leq 
 \frac{1}{d 2^{d/2}}
 \sum_{k=0}^d \binom{d}{k} (d-k)
( \CalgM)^{d-2}
 \|u\|_{p}^{d-1}
 \\&
 \leq 
2^{\frac{d}{2}-1} 
 ( \CalgM)^{d-2}
 \|u\|_{p}^{d-1}
 \leq 
 4
 ( 2 \CalgM)^{d-2}
 \|u\|_{p}^{d-1}\,.
 \end{aligned}
\]
Then (recall \eqref{normaHamilto})
\[
|R_0|_{\bar{r},p}\leq 
 ( 2 \CalgM \bar{r})^{d-2}\bar{r} \,,
\]
which implies the bound \eqref{stimaNEMSob}.
\end{proof}
We are now in position to prove our main results.

\subsubsection{Sobolev stability and proof of Theorem \ref{teoBNFwave}}
In order to prove Theorem \ref{teoBNFwave}
we reason as follows. We assume 
that the initial condition  $u_0$ satisfies (recall \eqref{smallnessu0})
\begin{equation}\label{smallnessu0Sob}
\|u_0\|_{p}:=\|u_0\|_{0,p}\leq \delta\,.
\end{equation}
Fix 
\begin{equation}\label{KapponeSob}
\tK:=\tK(p):=\left[\Big(\frac{p-1}{2^4 3^4}\Big)^{\frac{1}{5}}\right]-1\,,\qquad 
p_0:=p-\zeta:=p-2^4 3^4\sum_{i=1}^{\tK}i^4\,.
\end{equation}
Recalling the assumption $p>2^{6}(36)^{2}+1$ in Theorem \ref{teoBNFwave},
one can check that
\begin{equation}\label{boundsKappone}
1\leq \widetilde{c} (p-1)^{\frac{1}{5}}\leq \tK(p)\leq p^{1/5}\,,\qquad 
\widetilde{c}:=\frac{1}{2(36)^{2/5}}\,,\quad p_0>1\,.
\end{equation}
For future convenience we set
\begin{equation}\label{sceltabarrSob}
\bar{r}:=\frac{2}{\mathtt{C}_{{\rm alg}, \mathtt{M}}(p_0)}\,.
\end{equation}
Hence by  Lemma \ref{lem:applicoSub} we have that the Hamiltonian 
$H$ in  \eqref{waveHam} can be written, for any $\tK\geq q+2$, in the form 
\eqref{hamIniz}.
We apply Theorem \ref{mainthm:BNF}
to the Hamiltonian $H=D_{\omega}+R_0$ with $R_0$ in \eqref{hamRR00}.
Recalling the parameters in \eqref{parametri}-\eqref{patata2}
we fix
\begin{equation}\label{cond:r0piccoloSob}
r_0= 2\delta\,.
\end{equation}
We now 
show that
\begin{equation*}%\label{condr0KrSob}
\mathtt{C}_0:= 32 e4^3 |R_0|_{\bar{r},p_0} \tK\frac{4^{\tK}}{\gamma^{\tC K^4}}
\exp\Big( 2^{12}\mathtt{C} \tK^5\Big)
\frac{r_0}{\bar{r}} \leq 1\,,
\end{equation*}
for $K(p)$ defined in \eqref{KapponeSob}.
First, we notice that
\begin{equation}\label{cardano2}
p_0>1\qquad \Rightarrow
\quad
1\leq \mathtt{C}_{{\rm alg},\mathtt{M}}(p_0):=\sqrt{2}\sqrt{2+\frac{2p_0+1}{2p_0-1}}\leq 2^3\,.
\end{equation}
Therefore
\[
\begin{aligned}
\mathtt{C}_0&\leq  |R_0|_{\bar{r},p_0} \exp\Big\{2^{19}\mathtt{C}\ln(1/\gamma) \tK^5\Big\}
\frac{r_0}{\bar{r}}
\\&
\stackrel{\eqref{stimaNEMSob}, \eqref{sceltabarrSob}, \eqref{KapponeSob}}{\leq }
 ( 2 \CalgM \bar{r})^{q}\exp\Big\{2^{21}\mathtt{C}\ln(1/\gamma) p\Big\}r_0\leq 1\,,
\end{aligned}
\]
provided that
\[
r_0\leq \frac{1}{ ( 2 \CalgM \bar{r})^{q}}\exp\Big\{-2^{21}\mathtt{C}\ln(1/\gamma) p\Big\}.
\]
The last inequality follows from \eqref{cond:r0piccoloSob} %,  \eqref{smalldelta0Sob} 
and \eqref{smalldelta1Sob} (see also \eqref{sceltabarrSob}),
taking 
\begin{equation}\label{melone1}
\mathtt{c}\geq2^{21}\mathtt{C}\,.
\end{equation} 
In view of  Theorem \ref{mainthm:BNF} and Lemma \ref{tempotempo} 
we have 
that the solution $u(t)$ of \eqref{eq:waveComp} evolving from initial 
data satisfying \eqref{smallnessu0Sob}
remains in the ball of radius $2\delta$ for time $t\in [0,{T}_0]$ with
(recall the estimate \eqref{sonno3} and $J_{\tK}$ in \eqref{conditionRstar})
 \begin{equation*}
 \begin{aligned}
{T}_0 := 
\frac{\bar{r}^{\tK + 1}}{8 r_0^{\tK+1}\norm{R_0}_{\bar{r},p}} \frac{(16 e\tK 4^{\tK + 2})^{-\tK }}{\pa{\g^{-\tC\tK^4}\exp(2^{12}\mathtt{C}\tK^5)}^\tK}  \,.
\end{aligned}
 \end{equation*}
Observe that
\[
\begin{aligned}
T_0&\geq \frac{1}{8|R_0|_{\bar{r},p}}\frac{\bar{r}}{r_0}\left(\frac{\bar{r}}{r_0}\right)^{\tK}
\frac{1}{[\exp\big(2^{14}\mathtt{C}\ln(1/\gamma)\tK^5
\big)]^\tK}
\\&
\stackrel{\eqref{stimaNEMSob}, \eqref{sceltabarrSob}}{\geq}
\frac{1}{8 (2\mathtt{C}_{\rm alg}(p_0)\bar{r})^{q}}\frac{1}{r_0}
\left(\frac{\bar{r}}{r_0}\right)^{\tK}
\frac{1}{[\exp\big(2^{14}\mathtt{C}\ln(1/\gamma)\tK^5
\big)]^\tK}
\\&
\geq \frac{1}{8}\frac{1}{r_0}\left(\frac{\bar{r}}{r_0}\right)^{\tK}
\frac{1}{2^6 [\exp\big(2^{14}\mathtt{C}\ln(1/\gamma)\tK^5
\big)]^\tK},
\end{aligned}
\]
since $p\geq p_0$ and \eqref{cardano2} holds.
Finally,
by \eqref{boundsKappone}
one has
\[
\begin{aligned}
T_0&\geq 
\frac{1}{8r_0}\left(\frac{1}{\delta}\right)^{\widetilde{c} (p-1)^{1/5}}
 \frac{1}{2^6 
 [\exp\big(p 2^{14}\mathtt{C}\ln(1/\gamma)
\big)]^p}
\\&
\geq
\frac{1}{8r_0}\left(\frac{1}{\delta}\right)^{\widetilde{c} (p-1)^{1/5}}
 \frac{1}{
 [\exp\big(p 2^{15}\mathtt{C}\ln(1/\gamma)
\big)]^p}\,.
\end{aligned}
\]
Setting $\mathtt{c}=\max\{ 2(36)^{2/3}, 2^{21}\mathtt{C}\}$, we get 
\eqref{boundsol1Sob}-\eqref{longtime1Sob}.

\subsubsection{Optimization and proof of Theorem \ref{teonuovo}}\label{sec:ottimizzo}
In order to prove the result, in this section we combine the Birkhoff normal form result 
of Theorem \ref{teoBNFwave} with some \emph{a priori} energy estimates 
on the equation \eqref{eq:waveComp} to bound the 
\emph{high Sobolev norms} of the solution and obtain \eqref{stimaIncredibleBis}.
Such energy estimates are proved in Proposition \ref{thm:energy} of section 
\ref{sec:kleintd} in any space dimension $d\geq1$.

 Let us fix $\delta$ such that
\begin{equation}\label{deltapiccolo}
0<\delta\leq \bar{ \delta}:=\exp\Big\{-\mathtt{b}
\ln(1/\gamma)\Big\}
\end{equation}
where 
\begin{equation}\label{constB}
\mathtt{b}:=\max\big\{24\mathtt{c}^{2}\left(\frac{1}{48\mathtt{c}}\right)^{\frac{10}{9}},
24\mathtt{c}^{2}\big[2^6 (36)^{2}\big]^{5/3}
\big\}\,,
\end{equation}
and let us consider 
\begin{equation}\label{defpdelta}
p=p(\delta)=1+\left(\frac{1}{24\mathtt{c}^2\ln(1/\gamma)} \ln\big(\frac{1}{\delta}\big)\right)^{5/9}
\end{equation}
 where $\mathtt{c}>0$ is the absolute constant given by Theorem \ref{teoBNFwave}.

\noindent
Consider an initial condition satisfying \eqref{main:smallcondSob}, which, passing to complex coordinates 
in \eqref{beam5}
reads as \eqref{smallnessu0}.
%\begin{equation*}%\label{smallnessu0Sob}
%\|u_0\|_{p}\leq \delta\,.
%\end{equation*}
In order to apply Theorem \ref{teoBNFwave} we
shall verify that condition \eqref{smalldelta1Sob} holds for $\delta$ small enough.
First notice that \eqref{defpdelta}
implies 
\begin{equation}\label{defpdelta2}
\delta= \exp\big\{-24\mathtt{c}^2(p-1)^{\frac{9}{5}}\ln(1/\gamma)\big\}\,.
\end{equation}
Then  smallness condition \eqref{smalldelta1Sob} translates in proving that
\[
\exp\{ \mathtt{c}\ln(1/\gamma)\Big[ p-24\mathtt{c}(p-1)^{9/5}\Big]\}\leq1\quad \Leftrightarrow
\quad 
p-24\mathtt{c}(p-1)^{9/5}\leq0\,,
\]
recalling that $0<\gamma<1$, which is true as long as

\[
\begin{aligned}
p&=1+\left(\frac{1}{24\mathtt{c}^2\ln(1/\gamma)} \ln\big(\frac{1}{\delta}\big)\right)^{5/9} 
\geq 1+\big(\frac{1}{240\mathtt{c}}\big)^{\frac{5}{4}}\,,\quad \Leftrightarrow
\\
& \ln(1/\delta)\geq 24\mathtt{c}^2 \left(\frac{1}{240 \mathtt{c}}\right)^{\frac{9}{4}}\ln(1/\gamma)\,.
\end{aligned}
\]
This follows by \eqref{deltapiccolo}-\eqref{constB}.
With similar computations we can check that \eqref{deltapiccolo}-\eqref{constB}, 
together with \eqref{defpdelta}, yield
$p=p(\delta)>1+2^6(36)^{2}$.
Hence,Theorem \ref{teoBNFwave} applies, guaranteeing time of stability of the form
(see \eqref{longtime1Sob})
\[
\begin{aligned}
T_0&\geq\frac{1}{\delta}
\left(\frac{1}{\delta}\right)^{\frac{1}{\mathtt{c}}(p-1)^{1/5}}
\exp\big\{ -p^2\mathtt{c}\ln(1/\gamma)\big\}
\\&
\stackrel{\eqref{defpdelta2}}{=}
\frac{1}{\delta}
\exp\Big\{ 24\mathtt{c}(p-1)^{2}
\ln(1/\gamma)
-p^2\mathtt{c}\ln(1/\gamma)
\Big\}
\\&
\geq \frac{1}{\delta}
\exp\Big\{
\mathtt{c}\ln(1/\gamma) \Big(24(p-1)^{2}-p^2\Big)
\Big\}
\\&
\geq \frac{1}{\delta}\exp\Big\{
\mathtt{c}(p-1)^{2}\ln(1/\gamma) 
\Big\}
\\&
\stackrel{\eqref{defpdelta}}{\geq}
\frac{1}{\delta}\exp\Big\{
\frac{\mathtt{c}(\ln(1/\gamma))^{-1/9})}{(24\mathtt{c}^2)^{10/9}}(\ln(1/\delta) )^{1+\frac{1}{9}}
\Big\}=:T_{good}\,,
\end{aligned}
\]
which implies \eqref{patata1}.
In particular, by \eqref{boundsol1Sob}, we have that the solution $u(t)$ of \eqref{eq:waveComp} evolving form 
$u_0$, satisfies  
\begin{equation}\label{quadro1}
\sup_{t\in[0,T]}\|u(t)\|_{p}\leq 8\delta\,,
\end{equation}
for some $T\geq T_{good}$ and 
for $p=p(\delta)$ in \eqref{defpdelta}.
In order to study the evolution of the high 
norm $\|u(t)\|_{s}$ with $s\geq p+1$ we apply Proposition \ref{thm:energy}
with $s_{1}\rightsquigarrow p=p(\delta)$
and recalling that \eqref{def:japjapModificatoRR2} in Remark \ref{normaR2}.
%$ \|\cdot\|_{s}\equiv\|\cdot\|_{s,R}$ with $R=2$.
By estimate \eqref{buttalapasta} we get
\begin{equation}\label{lauraBeatrice}
\begin{aligned}
\|u(t)\|_{s}&\leq2^{s}
\|u(0)\|_{s}
+
\left(
\frac{\mathtt{M}^{qs }}{2^{q(p-s_0)}}
\int_{0}^{t}
\|u(\s)\|_{s_1}^{q+\frac{1}{s-p}}
d\s\right)^{s-p}
\\&
\leq2^s
\|u(0)\|_{s}+  
 \Big(
 \frac{t\mathtt{M}^{q (s-s_0)}}{2^{2q(p-s_0)}}
 \sup_{t\in[0,T]}\|u(t)\|_{p}^{q}
 \Big)^{s-p}
  \big(
 \sup_{t\in[0,T]}\|u(t)\|_{p}
 \big)
 \\&
 \stackrel{\eqref{quadro1}}{\leq}2^{s}
 \|u(0)\|_{s}+  
 \Big( t\delta^{q}
 \frac{\mathtt{M}^{q (s-s_0)} 8^{2q+1}}{2^{q(p-s_0)}}
 \Big)^{s-p}
 \delta^2\,.
\end{aligned}
\end{equation}
The bounds \eqref{quadro1} and \eqref{lauraBeatrice}, 
together with Lemma \ref{lemmaequivalenza} and \eqref{beam5}, imply
\eqref{stimaIncrediblelowNORM}-\eqref{stimaIncredibleBis}.
This concludes the proof of Theorem \ref{teonuovo}.

\section{Higher dimensional case: a priori estimates}\label{sec:stimetutte}
In this section we prove the long time stability results on $\T^{d}$ of solutions of the equation 
\eqref{eq:wave} given by 
Theorem \ref{thm:mainNOBNF}.
%we provide a priori estimates on the evolution of the 
%Sobolev norms of solutions of the equation 
%\eqref{eq:wave}. 
%To do this we shall use the complex variables introduced in 
%\eqref{beam5}, when $\tm\neq0$, or \eqref{complexNew} when $\tm=0$.
In section \eqref{sec:kleintd} we provide a priori energy estimate for
\eqref{eq:wave} on $\T^{d}$ when $\mathtt{m}\neq0$.
Then, in section \ref{sec:conlcudotd}, we will conclude the proof of the main result.

\subsection{Estimates for the Klein-Gordon equation}\label{sec:kleintd}
In this subsection we consider equation \eqref{eq:wave} with $\mathtt{m}>0$
for any $d\geq1$.
It is convenient to study the system passing to the complex variables 
\eqref{beam5}, i.e. we shall study equation \eqref{eq:waveComp}.
We shall prove the following.
\begin{proposition}{\bf (Basic a priori estimates: $\tm\neq0$).}\label{thm:energyBasic}
Let $d\geq 1$, $s_0>d/2$ and $s_1\geq s_0$.
Then, if $u=u(t,x)$  is a solution of \eqref{eq:waveComp} 
defined for $t\in[0,T]$ for some $T>0$ and  satisfying 
 \begin{equation}\label{hyp:local}
 \sup_{t\in [0,T]}\|u(t,\cdot)\|_{s_0,R}\leq 1\,,\qquad 
\sup_{t\in [0,T]}\|u(t,\cdot)\|_{s_1,R}<+\infty\,,
 \end{equation} 
 one has the a priori bound
 \begin{equation}\label{basicStima}
\|u(t)\|_{s_1,R}\leq \|u_{0}\|_{s_1,R}+\mathtt{M}^{q s_0} 
\Big(\frac{\mathtt{M}}{R}\Big)^{q(s_1-s_0)}\int_{0}^{t}
\|u(\s)\|_{s_1,R}^{q+1}d\s\,,\qquad t\in[0,T]\,,
\end{equation}
for some absolute constant $\mathtt{M}>0$.
\end{proposition}

\begin{proof}
Passing to the Duhamel formulation of the equation \eqref{eq:waveComp} we shall write
\[
u(t)=e^{-\ii\omega t}u(0)+\int_{0}^{t}e^{-\ii\omega (t-\s)}
\omega^{-1/2}f\left(\omega^{-1/2}\left(\frac{u(\s)+\bar{u}(\s)}{\sqrt 2}\right)\right) d\s\,.
\]
Therefore, using that $e^{-\ii\omega t}$  is an isometry of $H^{s}$, 
recalling \eqref{nonlionearitawave} and Remark 
\ref{rmk:collina}, 
we deduce
\[
\begin{aligned}
\|u(t)\|_{s_1,R}&\leq \|u(0)\|_{s_1,R}\,
+\int_{0}^{t}\|\omega^{-1/2}f\left(\omega^{-1/2}\left(\frac{u(\s)+\bar{u}(\s)}{\sqrt 2}\right)\right)\|_{s_1,R}d\s
\\&
\stackrel{\eqref{confort2}}{\leq }
\|u(0)\|_{s_1,R}\,+\int_{0}^{t}
\mathtt{M}^{qs_0}(\mathtt{M}/R)^{q(s_1-s_0)}\|u\|_{s_1,R}^{q+1}
d\s\,.
\end{aligned}
\]
This implies the bound \eqref{basicStima}.
\end{proof}
The second result of this section is the following.
\begin{proposition}{\bf (Improved a priori energy estimates).}\label{thm:energy}
Let $d\geq1$, $s_0>d/2$, $s_1\geq s_0$ and
consider 
$s\geq s_1+1$.
Then, if $u$ is a solution of \eqref{eq:waveComp}
satisfying \eqref{hyp:local} with 
$s_0\rightsquigarrow{s}_1$ and $s_{1}\rightsquigarrow s$,
 there exists an absolute constant $\mathtt{M}>0$ such that
\begin{equation}\label{energyapriori}
\|u(t)\|_{s,R}\leq \|u_0\|_{s,R}
+\mathtt{K}_1
\int_{0}^{t}
\|u(\tau)\|^{q+1-\lambda}_{s_1,R}
\|u(\tau)\|_{s,R}^{\lambda}d\tau\,,
\end{equation}
for every $t\in [0,T]$, and where
\begin{equation}\label{definizioneK1}
\lambda:=1-\frac{1}{s-s_1}\,,\qquad
\mathtt{K}_1:=\mathtt{K}_1(s,s_1,q):=
\frac{ \mathtt{M}^{q s} R^{s}}{R^{q(s_1-s_0)}}\,.
\end{equation}
As a consequence one also has
 \begin{equation}\label{buttalapasta}
\|u(t)\|_{s,R}\leq 2^{s}
\|u(0)\|_{s,R}
+
\left(
\mathtt{K}_1
\int_{0}^{t}
\|u(\s)\|_{s_1,R}^{q+\frac{1}{s-s_1}}
d\s\right)^{s-s_1}
\,,
 \end{equation}
 for any $t \in[0,T]$.
\end{proposition}

\begin{proof}
Reasoning as in the proof of Proposition \ref{thm:energyBasic}, 
one gets
\begin{equation}\label{energyaprioriBIS}
\begin{aligned}
\|u(t)\|_{s,R}&\leq \|u(0)\|_{s,R}\,+
\int_{0}^{t}\|\omega^{-1/2}f\left(\omega^{-1/2}\left(\frac{u(\s)+\bar{u}(\s)}{\sqrt 2}\right)\right)\|_{s,R}d\s
\\&
\stackrel{\eqref{collina1}}{\lesssim}
\|u(0)\|_{s,R}\,+
\sqrt{R}\int_{0}^{t}\|f\left(\omega^{-1/2}\left(\frac{u(\s)+\bar{u}(\s)}{\sqrt 2}\right)\right)\|_{s-\frac{1}{2},R}d\s
\\&
\stackrel{\eqref{confort}}{\lesssim}
\|u(0)\|_{s,R}\,+\sqrt{R}\mathtt{M}^{q(s-\frac{1}{2})}\int_{0}^{t}
\|\omega^{-1/2}u(\tau)\|_{s_0,R}^{q}\|\omega^{-1/2}u(\tau)\|_{s-\frac{1}{2},R}d\s
\\&
\stackrel{\eqref{collina3}, \eqref{collina1}}{\lesssim}
\|u(0)\|_{s,R}\,+\mathtt{M}^{q s}\, R\int_{0}^{t}
\|u(\tau)\|_{s_0,R}^{q}\|u(\tau)\|_{s-1,R}d\s\,.
%\\&
%\leq
%\|u(0)\|_{s,R}\,+\mathtt{M}^{2q(s-1)}\int_{0}^{t}
%\|u(\tau)\|_{s_0,R}^{2q}\|u(\tau)\|_{s-1,R}
%d\s\,.
\end{aligned}
\end{equation}
Notice  that (by assumption)
$s_0\leq s_1\leq s-1$ %\leq s$ (recall that $s\geq s_0+3$) 
and hence
\[
s-1=(1-\lambda) s_1+\lambda s\,\quad \Leftrightarrow\quad  \lambda=1-\frac{1}{s-s_1}\,.
\]
Then, by Lemma \ref{interopolo}-$(ii)$
we get
\[
\|u\|_{s-1,R}\leq 
18 R^{s-1}
\|u\|_{s_1,R}^{1-\lambda}\|u\|_{s,R}^{\lambda}\,,
\] 
for any $ t\in[0,T]$. 
In conclusion, by \eqref{energyaprioriBIS} and \eqref{scaling property},
we obtain
\[
\|u(t)\|_{s,R}\leq \|u_0\|_{s,R}
+\frac{18 \mathtt{M}^{qs} R^{s}}{R^{q(s_1-s_0)}}\int_{0}^{t}
\|u(\s)\|^{q+1-\lambda}_{s_1,R}
\|u(\s)\|_{s,R}^{\lambda}d\s\,,
\]
which is \eqref{energyapriori} (taking $\mathtt{M}$ in \eqref{definizioneK1} large enough).
In order to prove \eqref{buttalapasta} we reason as follows.
By estimate \eqref{energyapriori}
and Lemma \ref{drago}-$(ii)$ 
(applied with $\alpha\rightsquigarrow \lambda$)
one deduce
\begin{equation*}
\|u(t)\|_{s,R}^{1-\lambda}\leq 
\|u(0)\|_{s,R}^{1-\lambda}
+(1-\lambda)
\mathtt{K}_1
\int_{0}^{t}
\|u(\s)\|_{s_1,R}^{q+\frac{1}{s-s_1}}
d\s\,,
\qquad \forall \, t\in[0,T]\,.
\end{equation*}
By a simple computation\footnote{
Using  that $(x+y)^{q}\leq 2^{q-1}(x^{q}+y^{q})$,
$x,y\geq0$ and $q>1$\,.},
recalling that $1-\lambda=(s-s_1)^{-1}$,
we deduce \eqref{buttalapasta}.
\end{proof}

\subsection{Conclusions and  proof of Theorem \ref{thm:mainNOBNF}.}\label{sec:conlcudotd}
In this section we give the the proof of the long time stability result in Theorem 
\ref{thm:mainNOBNF}.
Let us consider initial conditions (see the assumption \eqref{piccolezza dati trisNOBNF}) 
satisfying 
\begin{equation}\label{campagna1}
\|\psi_0\|_{H^{s_1+1/2}}+\|\phi_0\|_{H^{s_1-1/2}}
+
(2\mathtt{M})^{s_1}\big(\|\psi_0\|_{H^{1/2}}+\|\phi_0\|_{H^{-1/2}}\big)\leq \delta\,,
%\|\psi_0\|_{H^{s_1+1/2}}+\|\phi_0\|_{H^{s_1-1/2}}\leq \delta\,,
%\qquad 
%(2\mathtt{M})^{s_1}\big(\|\psi_0\|_{H^{1/2}}+\|\phi_0\|_{H^{-1/2}}\big)\leq \delta\,,
\end{equation}
for some $\mathtt{M}>0$ large to be chosen.

By classical local existence theory 
we have that there exists a time $T_{loc}>0$ (possibly small)  and a unique solution 
of the equation 
\eqref{eq:wave} satisfying 
\[
\psi\in C^{0}([0,T_{loc}];H^{s_1}(\T^{d};\C))\cap C^{1}([0,T_{loc}];H^{s_1-1}(\T^{d};\C))\,,
\]
with $(\psi(0),\partial_{t}\psi(0))=(\psi_0,\phi_0)$ and with bounds
\begin{equation}\label{campagna2}
\sup_{t\in [0,T_{loc}]}\Big(
\|\psi(t)\|_{H^{s_1+1/2}}+\|\partial_{t}\psi(t)\|_{H^{s_1-1/2}}
+(2\mathtt{M})^{s_1}\big(
\|\psi(t)\|_{H^{1/2}}+\|\partial_{t}\psi(t)\|_{H^{-1/2}}
\big)
\Big)\leq 2\delta\,.
\end{equation}
Passing to the
complex coordinates \eqref{beam5}, we have that the function
\begin{equation}\label{venerdipomeriggio}
u=\frac{1}{\sqrt{2}}\big(\omega^{\frac{1}{2}}\psi+\ii \omega^{-\frac{1}{2}}v\big)\,,
\qquad 
v=\partial_{t}\psi\,,
\end{equation}
solves the equation 
\eqref{eq:waveComp}. 
In particular, by  \eqref{campagna1}-\eqref{campagna2} and using 
the equivalence 
\eqref{rmk:equivalenzaTotale1} in Remark
\ref{rmk:equivalenzaTotale}, we deduce that 
there exists a constant $c_{\tm}>0$ such that
\begin{equation}\label{stiminesuu}
\begin{aligned}
\|u(0)\|_{s_1,R}\leq c_{\tm}\delta\,,\qquad 
\sup_{t\in[0,T_{loc}]}\|u(t)\|_{s_1,R}\leq 2c_{\tm}\|u(0)\|_{s_1,R}\,,\qquad R:=2\mathtt{M}\,.
\end{aligned}
\end{equation}
We now show that actually the solution $u(t)$ of \eqref{eq:waveComp} 
exists 
over a time interval $[0,T]$ with $T\geq T_{loc}$ 
satisfying the bound in \eqref{patata1dimd}.

\smallskip
\noindent
Let us define, for $r>0$
\[
B^{s_1}(r):=\{u\in H^{s_1}(\T^{d};\C)\,:\, \|u\|_{s_1,R}\leq r\}\,.
\]
Note that the bounds \eqref{stiminesuu} imply
that the initial condition $u(0)$ belongs to $B^{s_1}(c_{\tm}\delta)$, 
while the solution $u(t)$ is in $B^{s_1}(2 c_{\tm}\delta)$
at least for $t\in[0,T_{loc}]$.
Let us define the time of escape from the ball as 
\begin{equation*}%\label{time escape}
\tau_{e} : = \inf\set{t\in\R\, : u(t)\not\subset B^{s_{1}}(2 c_{\tm}\delta)}\,.
\end{equation*}
Notice that a priori one has 
$\sup_{t\in[0,\tau_e)} \|u(t)\|_{s_1,R}\leq 2c_{\tm}\delta$.
If the set is empty then we have ``perpetual'' stability and the solution stays in the ball 
of radius $2c_{\tm}\delta$ as long as the solution exists. 
If the set is not empty then we shall show that 
\begin{equation}\label{ginogino2}
\tau_{e} \geq  
\frac{2^{q(s_1-s_0)}}{\mathtt{M}^{qs_0} 2^{q+2}c_{\tm}^{q}}\frac{1}{\delta^{q}}
:=T_{good}\,,
\end{equation} 
where we have assumed, without loss of generality, that $\tau_e>0$.
We prove the claim by contradiction assuming $\tau_{e}<T_{good}$.

\smallskip
\noindent
First of all, 
 for $\delta>0$ small enough, conditions \eqref{stiminesuu} imply that  the assumption 
\eqref{hyp:local} in Proposition \ref{thm:energyBasic} is satisfied.
Therefore, by estimate \eqref{basicStima}  (with $T\rightsquigarrow \tau_e$, and 
$R\rightsquigarrow 2\mathtt{M}$)
we deduce
\begin{equation}\label{caffelet10}
\begin{aligned}
2c_{\tm}\delta\leq \|u(\tau_e)\|_{s_1,R}&\leq 
\|u(0)\|_{s_1,R}
+\mathtt{M}^{q s_0} 2^{-q(s_1-s_0)}\int_{0}^{\tau_{e}}
\|u(\s)\|_{s_1,R}^{q+1}d\s\,,
\\&\leq 
c_{\tm}\delta+
\tau_e\mathtt{M}^{q s_0} 2^{-q(s_1-s_0)}2^{q+1}(c_{\tm}\delta)^{q+1}
\stackrel{\eqref{ginogino2}}{\leq} \frac{3}{2}c_{\tm}\delta\,,
\end{aligned}
\end{equation}
which is a contradiction. So one must have $\tau_{e}\geq T_{good}$.
Therefore
the solution $u(t)$ can be extended over a time interval $[0,T]$ with $T_{good}\leq T<\tau_e$ 
 consistently with \eqref{patata1dimd} (with $\mathtt{M}$ given by Proposition \ref{thm:energyBasic}
 and $\mathtt{c}_{\tm}=2^{q+2}c_{\tm}^{q}$), 
and that 
\begin{equation}\label{stimateo11}
\sup_{t\in[0,T]}\|u(t)\|_{s_1,R}\leq 2c_{\tm}\|u(0)\|_{s_1,R}\leq 2c_{\tm}\delta\,.
\end{equation}

\smallskip
\noindent
Let us now consider $s\geq s_1+1$, $s_1\geq s_0>d/2$
and take $(\psi_{0},\phi_0)$ satisfying  (recall \eqref{campagna1})
\[
\begin{aligned}
\|\psi_0\|_{H^{s_1+1/2}}+\|\phi_1\|_{H^{s_1-1/2}}+
(2\mathtt{M})^{s_1}\big(\|\psi_0\|_{H^{1/2}}+\|\phi_0\|_{H^{-1/2}}\big)&\leq \delta\,,
\\
\|\psi_0\|_{H^{s+1/2}}+\|\phi_1\|_{H^{s-1/2}}&<+\infty\,.
\end{aligned}
\]
By classical local existence theory we have that there exists 
a unique solution 
in the space 
\[
C^{0}([0,T_{loc}];H^{s+\frac{1}{2}}(\T^{d};\C))\cap C^{1}([0,T_{loc}];H^{s-\frac{1}{2}}(\T^{d};\C))\,,
\]
for some $T_{loc}>0$ possibly small. 
Passing to complex coordinates \eqref{beam5}, and recalling the equivalence
\eqref{primaEquiv}, we have that the function $u=u(t)$ in 
\eqref{venerdipomeriggio}  is a solution of \eqref{eq:waveComp}
belonging to $C^{0}([0,T_{loc}];H^{s}(\T^{d};\C))\cap C^{1}([0,T_{loc}];H^{s-1}(\T^{d};\C))$.
 Moreover, the  estimate \eqref{stimateo11} proved above,
guarantees that
the \emph{low} norm $\|\cdot\|_{s_1,R}$
stay small for very long time (recall \eqref{ginogino2}):
\begin{equation}\label{staysmall}
\sup_{t\in[0,T]}\|u(t)\|_{s_1,R}\leq2c_{\tm}\|u(0)\|_{s_1,R}\leq 2c_{\tm}\delta\,,
\qquad 
T\geq T_{good}\,,\quad R:=2\mathtt{M}\,.
\end{equation}
We now study the evolution of the high norm
$\|u(t)\|_{s,R}$ along the flow of \eqref{eq:waveComp}.

For any $0\leq t\leq \widehat{T}\leq T$,
by applying Proposition \ref{thm:energy} 
(see \eqref{buttalapasta} and recall $R=2\mathtt{M}$)
 we have the a priori
bound on the \emph{high} Sobolev norm
 \begin{equation*}
 \begin{aligned}
\|u(t)\|_{s,R}&\leq
2^s\|u(0)\|_{s,R}
+
\left(
\mathtt{K}_1
\int_{0}^{t}
\|u(\s)\|_{s_1,R}^{q+\frac{1}{s-s_1}}
d\s\right)^{s-s_1}
\\
 &\leq 
2^{s}\|u(0)\|_{s,R}+  
 \Big( t\mathtt{K}_1
 \sup_{t\in[0,T]}\|u(t)\|_{s_1,R}^{q}
 \Big)^{s-s_1}
  \big(
 \sup_{t\in[0,T]}\|u(t)\|_{s_1,R}
 \big)
 \\&
 \stackrel{\eqref{staysmall}}{\leq }
 2^{s}\|u(0)\|_{s,R}\big(1+t\mathtt{K}_1(2c_{\tm})^{q}\delta^{q}\big)^{s-s_1}
\\&
 \stackrel{\eqref{definizioneK1}}{\leq}
  2^{s}\|u(0)\|_{s,R}\big(1+
  t\frac{\mathtt{M}^{q (s-s_1)+1}2^s(2c_{\tm})^{q}\mathtt{M}^{q s_0}}{2^{q(s-s_1)}}\big)^{s-s_1}
  \\&
  \leq
  2^{s}\|u(0)\|_{s,R}\big(1+
  t\frac{\delta^{q}\mathtt{M}^{qs_0}\mathtt{c}_{\tm}}{2^{q(s_1-s_0)}}
\mathtt{M}^{q(s-s_1)}\big)^{s-s_1}
 \end{aligned}
 \end{equation*}
 where in the last estimates we used that $(2c_{\tm})^{q}\leq \mathtt{c}_{\tm}$ and we took $\mathtt{M}$ larger that the previous one.
 The latter estimate, together with the equivalence \eqref{rmk:equivalenzaTotale1}, implies
 \eqref{stimaIncredibleBisNOBNF} and
 the thesis follows by a standard bootstrap argument. 
 The bound \eqref{stimaIncredible} follows by \eqref{stimaIncredibleBisNOBNF},
where $T_{good}$ is defined in 
 \eqref{patata1dimd}.

%\part{$\mathtt{m}=0$: Energy estimates on $\T^{d}$: the resonant case}
\part{\texorpdfstring{$\mathtt{m}=0$}{m=0}: Energy estimates on 
\texorpdfstring{$\T^{d}$}{Td}: the resonant case}

\section{Time of stability for the wave equation in high space dimension}
Here we study the equation \eqref{eq:wave}  in the case  $d>1$
when the mass parameter $\mathtt{m}=0$.

\noindent
In section \ref{sec:kleintdWAVE} we provide a priori energy estimates for the wave equation.
Then in section \ref{sec:finaleWave} we conclude the proof of the main Theorem \ref{thm:mainNOBNF2}.

\subsection{Estimates for the wave equation}\label{sec:kleintdWAVE}

In this subsection we consider equation \eqref{eq:wave} with $\mathtt{m}=0$
for any $d\geq1$.
It is convenient to study the system passing to the complex variables 
\eqref{complexNew}, i.e. we shall study thew system \eqref{sistemasplittato}.
We have the following.

\begin{proposition}{\bf (Basic a priori bounds: $\tm=0$).}\label{thm:energyBasicBIS}
Let $d\geq 1$, $s_0>d/2$ and $s_1\geq s_0$.
Then, if $
(\psi_0(t),v_0(t),u^{\perp}(t))$ is a solution of \eqref{sistemasplittato}
defined for $t\in[0,T]$ for some $T>0$ and satisfying 
 \begin{equation}\label{hyp:localBIS}
 \begin{aligned}
& \sup_{t\in [0,T]}E_{s_0}(t)\leq 1\,,\qquad 
\sup_{t\in [0,T]}E_{s_1}(t)<+\infty\,,
\\
E_{s}(t)&:=|\psi_0(t)|+|v_0(t)|+\|u^{\perp}(t)\|_{s,R}\,, \quad s\geq s_0\,,
\end{aligned}
 \end{equation} 
 one has the a priori bound
 \begin{equation}\label{basicStimaBIS}
 \begin{aligned}
E_{s_1}(t)&\leq 
|\breve{\psi}|+|\breve{v}|(1+t)+\|\breve{u}^{\perp}\|_{s_1,R}
\\&+\frac{\mathtt{M}^{qs_1}}{R^{q(s_1-s_0)}}\int_{0}^{t}\left[ 
(E_{s_1}(\s))^{q+1}
+\int_{0}^{\s}(E_{s_1}(\tau))^{q+1}d\tau
\right]d\s
\qquad t\in[0,T]\,,
\end{aligned}
\end{equation}
for some  pure constant 
$\mathtt{M}>0$ large and where 
$(\breve{\psi},\breve{v},\breve{u}^{\perp}):=(\psi_0(0),v_0(0),u^{\perp}(0))$.
\end{proposition}

\begin{proof}
We rewrite the first two equation in \eqref{sistemasplittato} as
\[
\begin{aligned}
\psi_{0}(t)&=\breve{\psi}+\int_{0}^{t}v_0(\s)d\s\,,
\\
v_0(t)&=\breve{v}-\int_{0}^{t}\Pi_0\Big( f(\psi_0(\s)+u^{\perp}(\s)+\bar{u}^{\perp}(\s))\Big)d\s\,,
\qquad t\in[0,T]\,.
\end{aligned}
\]
We first note
\[
\begin{aligned}
|v_0(t)|&\leq |\breve{v}|
+\int_0^{t}\big|\Pi_0\Big( f(\psi_0(\s)+u^{\perp}(\s)+\bar{u}^{\perp}(\s))\Big)\big|d\s
\\&
\leq |\breve{v}|
+\int_0^{t}\| f(\psi_0(\s)+u^{\perp}(\s)+\bar{u}^{\perp}(\s))\|_{s_0,R}d\s
\\&
\stackrel{\eqref{nonlionearitawave}, \eqref{confort}}{\leq}
|\breve{v}|+\int_{0}^{t}\mathtt{M}^{qs_0}\| \psi_0(\s)+u^{\perp}(\s)+\bar{u}^{\perp}(\s)\|_{s_0,R}^{q+1}d\s\,.
%\\&
%\stackrel{\eqref{energiaMzero}}{\leq}
%|\breve{v}|+\mathtt{M}^{2q s_0}\int_{0}^{t}(E_{s_0}(\s))^{2q+1}d\s\,.
\end{aligned}
\]
Secondly, we note that, for $s_1\geq s_0$,
\[
\begin{aligned}
\| \psi_0(\s)+u^{\perp}(\s)+\bar{u}^{\perp}(\s)\|_{s_0,R}\stackrel{\eqref{confort3}}{\leq }
R^{-(s_1-s_0)}\|\psi_0(\s)+u^{\perp}(\s)+\bar{u}^{\perp}(\s)\|_{s_1,R}
\stackrel{\eqref{hyp:localBIS}}{\leq}
2 R^{-(s_1-s_0)}E_{s_1}(\s)\,.
\end{aligned}
\]
As a consequence we get
\begin{equation}\label{stimeMedie}
\begin{aligned}
|\psi_0(t)|&\leq |\breve{\psi}|+|\breve{v}| t+ \frac{\mathtt{M}^{qs_0}}{R^{(q+1)(s_1-s_0)}}
\int_{0}^{t}\int_{0}^{\s}(E_{s_1}(\tau))^{q+1}d\tau\,,
\\
|v_0(t)|&\leq |\breve{v}| + \frac{\mathtt{M}^{qs_0}}{R^{(q+1)(s_1-s_0)}}
\int_{0}^{t}(E_{s_1}(\s))^{q+1}d\s\,,
\end{aligned}
\end{equation}
for some $\mathtt{M}\gg1$ large.
Let us now consider the third equation in \eqref{sistemasplittato}, that
we write as
\begin{equation}\label{UUperpperp}
u^{\perp}(t)=e^{-\ii |D|t}\breve{u}^\perp
-\frac{\ii}{2}\int_0^{t}e^{-\ii |D|(t-\s)}|D|^{-1}\Pi_0^{\perp}
\Big(f\big(
\psi_0(\s)+u^{\perp}(\s)+\bar{u}^{\perp}(\s)
\big)\Big)d\s\,.
\end{equation}
Therefore, using that $\|e^{-\ii |D|t}h\|_{s,R}=\|h\|_{s,R}$ for any $h\in H^{s}(\T^d;\C)$, 
recalling Remark \ref{rmk:collina}, one has
\begin{equation}\label{stimeMedieperp}
\begin{aligned}
\|u^{\perp}(t)\|_{s_1,R}&\leq \|\breve{u}^{\perp}\|_{s_1,R}+
\int_{0}^{t}\||D|^{-1}\Pi_0^{\perp}
\Big(f\big(
\psi_0(\s)+u^{\perp}(\s)+\bar{u}^{\perp}(\s)
\big)\Big)\|_{s_1,R}d\s
\\&
\stackrel{\eqref{confort2}}{\leq} \|\breve{u}^{\perp}\|_{s_1,R}+
 \mathtt{M}^{qs_0}(\mathtt{M}/R)^{q(s_1-s_0)}
\int_{0}^{t} (E_{s_1}(\s))^{q+1}d\s\,.
\end{aligned}
\end{equation}
To summarize, \eqref{stimeMedie}-\eqref{stimeMedieperp} imply
\eqref{basicStimaBIS}.
\end{proof}

The second result of this section is the following.

\begin{proposition}{\bf (Improved a priori bounds).}\label{thm:energyBIS}
Let $d\geq1$, $s_0>d/2$, $s_1\geq s_0$ and
consider 
$s\geq s_1+1$.
Then, if $(\psi_0,v_0,u^{\perp})$ is a solution of \eqref{sistemasplittato}
satisfying \eqref{hyp:localBIS} with 
$s_0\rightsquigarrow{s}_1$ and $s_{1}\rightsquigarrow s$,
 there exists an absolute constant $\mathtt{M}>0$ such that
\begin{equation}\label{energyaprioriTRIS}
\begin{aligned}
E_{s}(t)&\leq r(t)+20 \mathtt{K}_2\big(h(t)\big)^{1-\lambda}
\Big(\int_{0}^{T}\big[r(t)+(T\mathtt{K}_2)^{\frac{1}{1-\lambda}}h(t)\big]dt\Big)^{\lambda}\,,
\end{aligned}
\end{equation}
for every $t\in [0,T]$, and where
\begin{equation}\label{costBrutte}
\begin{aligned}
\lambda&:=1-\frac{1}{(s-s_1)}\,,\qquad 
\mathtt{K}_2:=\frac{\mathtt{M}^{q(s-s_1-s_0)}}{2^{q(s_1-s_0)}}\,,
\end{aligned}
\end{equation}
and
\begin{equation}\label{costBrutte2}
\begin{aligned}
r(t)&:=E_{s}(0)+|\breve{v}|t+ \mathtt{K}_2
\int_{0}^{t}\int_{0}^{\s}(E_{s_1}(\tau))^{q+1}d\tau d\s\,,
\\
h(t)&:=
\int_0^{t}(E_{s_1}(\s))^{\frac{q+1-\lambda}{1-\lambda}} d\s\,,
\end{aligned}
\end{equation}
%As a consequence one also has
% \begin{equation}\label{buttalapastaBIS}
%s
% \end{equation}
% for any $t \in[0,T]$.
\end{proposition}

\begin{proof}
In the following we shall not keep track of absolute constants, which may be 
taken conveniently bigger from one inequality to the other and, we shall indicate them as $\mathtt{M}$.

\noindent
First of all, 
from the proof of Proposition \ref{thm:energyBasicBIS}  (see bounds \eqref{stimeMedie})
we recall that
 \begin{equation}\label{pastaFredda}
 \begin{aligned}
|\psi_{0}(t)|+|v_0(t)|&\leq 
|\breve{\psi}|+|\breve{v}|(1+t)
\\&+\frac{\mathtt{M}^{qs_0}}{R^{(q+1)(s_1-s_0)}}\int_{0}^{t}\left[ 
(E_{s_1}(\s))^{q+1}
+\int_{0}^{\s}(E_{s_1}(\tau))^{q+1}d\tau
\right]d\s
\qquad t\in[0,T]\,.
\end{aligned}
\end{equation}
We now provide an estimate on $\|u^{\perp}(t)\|_{s,R}$ for $s\geq s_1+1$.
Reasoning as in \eqref{UUperpperp}-\eqref{stimeMedieperp}, 
%and using that 
%$\||D|^{-1}\Pi_0^{\perp}h\|_{s,R}\lesssim \|\Pi_0^{\perp}h\|_{s-1,R}$, 
one gets 
\begin{equation}\label{stimeMedieperpBIS}
\begin{aligned}
\|u^{\perp}(t)\|_{s,R}&\leq \|\breve{u}^{\perp}\|_{s,R}+\frac{1}{2}
\int_{0}^{t}\||D|^{-1}\Pi_0^{\perp}
\Big(f\big(
\psi_0(\s)+u^{\perp}(\s)+\bar{u}^{\perp}(\s)
\big)\Big)\|_{s,R}d\s
\\&
\stackrel{\eqref{collina2}}{\leq}\|\breve{u}^\perp\|_{s,R}+\mathtt{M}R
\int_{0}^{t}\|
\Big(f\big(
\psi_0(\s)+u^{\perp}(\s)+\bar{u}^{\perp}(\s)
\big)\Big)\|_{s-1,R}d\s
\\
&\stackrel{\eqref{nonlionearitawave}, \eqref{confort}}{\leq}
\|\breve{u}^\perp\|_{s,R}
+ \mathtt{M}^{q s} R \int_0^t 
\|\psi_0(\s)+u^{\perp}(\s)\|_{s_0,R}^{q}
\|\psi_0(\s)+u^{\perp}(\s)\|_{s-1,R}
d\sigma
\\
&\stackrel{\eqref{confort3}}{\leq} 
\|\breve{u}^\perp\|_{s,R}
+ \mathtt{M}^{qs} R R^{-q(s_1 - s_0)} 
\int_0^t E_{s_1}^{q}(\sigma)E_{s-1}(\sigma)\, d\sigma\,.
\end{aligned}
\end{equation}
Moreover, by Lemma \ref{interopolo}-$(ii)$ 
for any $s-1$ in the interval $0\leq s_1\leq s-1 \leq s$  we obtain
\[
\|u^{\perp}\|_{s-1,R}\leq {18 R^{s-1}}\|u^{\perp}\|_{s_1,R}^{1-\lambda}\|u^{\perp}\|_{s,R}^{\lambda}\,,
%\leq 
%{18 R^{s-1}} \|u^{\perp}\|_{s_1,R}^{\frac{1}{s-s_1}}\|u^{\perp}\|_{s,R}^{1-\frac{1}{s-s_1}}\,,
\qquad 
\lambda:=1-\frac{1}{s-s_1}\,,
\] 
and therefore 
\begin{align*}
E_{s-1} (t) = |\psi_0(t)| + |v_0(t)| + \|{u^\perp(t)}\|_{s-1,R} & \leq  |\psi_0(t)| + |v_0(t)|  + 18R^{s-1}  
\|u^{\perp}\|_{s_1,R}^{1-\lambda}\|u^{\perp}\|_{s,R}^{\lambda}
\\
& \leq  |\psi_0(t)| + |v_0(t)| + 18R^{s-1}E_{s_1}^{1-\lambda}E_s^\lambda\,.
\end{align*}
By monotonicity of $E_{s}$, one has that $|\psi_0(t)| + |v_0(t)| \le E_{s_1}
\leq E_{s_1}^{1-\lambda}E_s^\lambda $. 
Then, from inequality \eqref{stimeMedieperpBIS}  we deduce that

\begin{equation}\label{pastaFredda2}
\begin{aligned}
	\|u^\perp\|_{s,R} &\leq  \|\breve{u}^\perp\|_{s,R} 
	+ \frac{\mathtt{M}^{q s}R^{s}}{R^{q(s_1 - s_0)} }
	\int_0^t E_{s_1}^{q+1-\lambda}(\sigma)E_s^\lambda(\sigma))\, d\sigma\,,
\end{aligned}
\end{equation}
for some $\mathtt{M}>0$ possibly larger than the one in \eqref{stimeMedieperpBIS}.
Hence, collecting together \eqref{pastaFredda}-\eqref{pastaFredda2},
%the above estimates 
we get
\begin{equation}\label{suits1}
\begin{aligned}
E_{s}(t)&\leq E_{s}(0)+|\breve{v}| t+
\frac{\mathtt{M}^{qs_0}}{R^{q(s_1-s_0)}}\int_{0}^{t}\left[ 
(E_{s_1}(\s))^{q+1}
+\int_{0}^{\s}(E_{s_1}(\tau))^{q+1}d\tau
\right]d\s
\\&
+\frac{\mathtt{M}^{qs}R^{s}}{R^{q(s_1-s_0)}}\int_{0}^{t}  (E_{s_1}(\s))^{q+1-\lambda}
(E_{s}(\s))^{\lambda}d\s\,.
\end{aligned}
\end{equation}
Using again $E_{s_1}\leq(E_{s_1})^{1-\lambda}(E_{s})^{\lambda}$ and setting $R:=2\mathtt{M}$ for some 
$\mathtt{M}$ large (possibly larger than the one in \eqref{suits1}) we deduce
\begin{equation}\label{suits111}
\begin{aligned}
E_{s}(t)&\leq E_{s}(0)+|\breve{v}| t+
K
\int_{0}^{t}
\int_{0}^{\s}(E_{s_1}(\tau))^{q+1}d\tau d\s
%\\&
+K\int_{0}^{t}  (E_{s_1}(\s))^{q+1-\lambda}
(E_{s}(\s))^{\lambda}d\s\,,
\end{aligned}
\end{equation}
with\footnote{Here we are using that
\[
A^{q s}B^{s}\leq C^{q s}\,,
\]
taking $C\gg A\cdot B$ and using that $q\geq 1$.
}
\[
K:=\frac{\mathtt{M}^{q(s-s_1+s_0)}}{2^{q(s_1-s_0)}}\,.
\]
By setting, for $t\in[0,T]$,
\begin{equation}\label{def:lemmadrago}
\begin{aligned}
f(t)&:= 
E_{s}(0)+|\breve{v}| t+
K
\int_{0}^{t}
\int_{0}^{\s}(E_{s_1}(\tau))^{q+1}d\tau d\s
\,,
\\
g(\s)&:= K
(E_{s_1}(\s))^{q+1-\lambda}\,,
\qquad
x(t):=E_{s}(t)\,,
\end{aligned}
\end{equation}
we rewrite \eqref{suits111} as
\[
x(t)\leq f(t)+\int_{0}^{t}g(\s) (x(\s))^{\lambda}d\s\,,
\]
so that we shall apply Lemma \ref{drago2}. Hence, by \eqref{stimaTesi}, we get
\begin{align*}
&E_s(t) 
 \leq
f(t)
\\
& + 20^\lambda \left\{ \frac{1}{1-\lambda}\int_0^T 
\frac{f(t)}{1-\lambda}
dt 
%\\& 
+ \left(\int_0^T \left(   \int_0^t\left( 
g(\s) \right)^{\frac{1}{1-\lambda}}\,d\s\right)^{1-\lambda}\,dt 
\right)^{\frac{1}{1-\lambda}} 
\right\}^{\lambda}\times
 \left(   \int_0^t\left( 
g(\s)\right)^{\frac{1}{1-\lambda}}\,d\s\right)^{1-\lambda}\,.
\end{align*}
Note that, using \eqref{def:lemmadrago} and \eqref{costBrutte2}, we have
\[
\begin{aligned}
\left(\int_0^T \left(   \int_0^t\left( 
g(\s) \right)^{\frac{1}{1-\lambda}}\,d\s\right)^{1-\lambda}\,dt 
\right)^{\frac{1}{1-\lambda}} 
&\stackrel{\eqref{def:lemmadrago}}{=}
\left(\int_0^T \left(   \int_0^t\left( 
K
(E_{s_1}(\s))^{q+1-\lambda} \right)^{\frac{1}{1-\lambda}}\,d\s\right)^{1-\lambda}\,dt 
\right)^{\frac{1}{1-\lambda}} 
\\&
\stackrel{\eqref{costBrutte2}}{=}
\left(\int_0^T K \left(  h(t)
\right)^{1-\lambda}\,dt 
\right)^{\frac{1}{1-\lambda}} 
\leq 
(TK)^{\frac{1}{1-\lambda}}
\int_0^T    h(t)
\,dt \,,
\end{aligned}
\]
where the last inequality we used Jensen's inequality.
%
%\begin{align*}
%E_s(t) 
%& \le 
%E_{s}(0)+|\breve{v}| t+
%K
%\int_{0}^{t}
%\int_{0}^{\s}(E_{s_1}(\tau))^{q+1}d\tau d\s +
%\\
%& + 20^\lambda \left\{ \frac{1}{1-\lambda}\int_0^T 
%\left[
%E_{s}(0)+|\breve{v}| t+
%K
%\int_{0}^{t}
%\int_{0}^{\s}(E_{s_1}(\tau))^{q+1}d\tau d\s
%\right]
%dt + \right.
%\\
%& + \left.\left(\int_0^T \left(   \int_0^t\left( 
%K
%(E_{s_1}(\s))^{q+1-\lambda} \right)^{\frac{1}{1-\lambda}}\,d\s\right)^{1-\lambda}\,dt 
%\right)^{\frac{1}{1-\lambda}} 
%\right\}^{\lambda}\times
%\\
%& \times \left(   \int_0^t\left( 
%K
%(E_{s_1}(\s))^{q+1-\lambda} \right)^{\frac{1}{1-\lambda}}\,d\s\right)^{1-\lambda}\,.
%\end{align*}
%Note that, by Jensen inequality we get
%\[
%\begin{aligned}
%&\left(\int_0^T \left(   \int_0^t\left( 
%K
%(E_{s_1}(\s))^{q+1-\lambda} \right)^{\frac{1}{1-\lambda}}\,d\s\right)^{1-\lambda}\,dt 
%\right)^{\frac{1}{1-\lambda}} 
%\\&
%\qquad \qquad \qquad\qquad \qquad\leq
%(TK)^{\frac{1}{1-\lambda}}
%\int_0^T    \int_0^t\left(
%(E_{s_1}(\s))^{q+1-\lambda} \right)^{\frac{1}{1-\lambda}}\,d\s\,dt \,.
%\end{aligned}
%\]
Therefore, we have obtained 
\begin{equation}\label{fogliodicarta}
\begin{aligned}
E_{s}(t)&\leq f(t)+ 20K\left\{
\int_{0}^{T}\left[ \frac{f(t)}{1-\lambda}+
(TK)^{\frac{1}{1-\lambda}}
h(t)\right]\,dt
\right\}^{\lambda}\times
 \left(   h(t)
\right)^{1-\lambda}\,.
\end{aligned}
\end{equation}
%\begin{equation}\label{fogliodicarta}
%\begin{aligned}
%E_{s}(t)&\leq f(t)+ 20K\left\{
%\int_{0}^{T}\left[ \frac{f(t)}{1-\lambda}+(TK)^{\frac{1}{1-\lambda}}
%   \int_0^t\left(
%(E_{s_1}(\s))^{q+1-\lambda} \right)^{\frac{1}{1-\lambda}}\,d\s\right]\,dt
%\right\}^{\lambda}\times
%\\&\qquad\qquad
%\times
% \left(   \int_0^t\left(
%(E_{s_1}(\s))^{q+1-\lambda} \right)^{\frac{1}{1-\lambda}}\,d\s\right)^{1-\lambda}\,.
%\end{aligned}
%\end{equation}
Finally, recalling $\lambda$  in \eqref{costBrutte}, we have
\[
\frac{1}{(1-\lambda)^{\lambda}}=(s-s_1)^{1-\frac{1}{s-s_1}}\leq s-s_1\leq \mathtt{C}^{q(s-s_1-s_0)}\,,
\]
for some absolute constant $\mathtt{C}>0 $
 large enough. The latter inequality,
together with \eqref{fogliodicarta}, \eqref{def:lemmadrago} and \eqref{costBrutte}-\eqref{costBrutte2}, 
implies the bound \eqref{energyaprioriTRIS} taking the constant $\mathtt{M}$
in \eqref{costBrutte} large enough.
\end{proof}

\subsection{Conclusions and proof of Theorem \ref{thm:mainNOBNF2}.}\label{sec:finaleWave}
In this section we give the the proof of the long time stability result in Theorem 
\ref{thm:mainNOBNF2}.
Let us consider initial conditions (see the assumption \eqref{piccolezza dati trisNOBNF1}) 
satisfying 
\begin{equation}\label{campagna1BIS}
\|\breve{\psi}\|_{H^{s_1}}+\|\breve{v}\|_{H^{s_1-1}}\leq \delta\,,
\qquad 
(2\mathtt{M})^{s_1}\big(\|\breve{\psi}\|_{L^2}+\|\breve{v}\|_{H^{-1}}\big)\leq \delta\,,
\end{equation}
for some $\mathtt{M}>0$ large to be chosen.
By classical local existence theory 
we have that there exists a time $T_{loc}>0$ (possibly small)  and a unique solution 
of the equation 
\eqref{eq:wave} satisfying 
\[
\psi\in C^{0}([0,T_{loc}];H^{s_1}(\T^{d};\C))\cap C^{1}([0,T_{loc}];H^{s_1-1}(\T^{d};\C))\,,
\]
with $(\psi(0),\partial_{t}\psi(0))=(\breve{\psi},\breve{v})$
and satisfying the bounds
\begin{equation}\label{campagna2BIS}
\begin{aligned}
\sup_{t\in [0,T_{loc}]}\Big(
\|\psi(t)\|_{H^{s_1}}&+\|\partial_{t}\psi(t)\|_{H^{s_1-1}}
+(2\mathtt{M})^{s_1}\big(
\|\psi(t)\|_{L^{2}}+\|\partial_{t}\psi(t)\|_{H^{-1}}
\big)
\Big)
\\
&\leq 2\big(\|\breve{\psi}\|_{H^{s_1}}+\|\breve{v}\|_{H^{s_1-1}}+
(2\mathtt{M})^{s_1}\big(\|\breve{\psi}\|_{L^2}+\|\breve{v}\|_{H^{-1}}\big)\big)
\leq 2\delta
\end{aligned} 
\end{equation}
Passing to the
complex coordinates \eqref{splittoperp}-\eqref{complexNew}, we have that 
$(\psi_0,v_0,u^{\perp})$ with 
\begin{equation}\label{venerdipomeriggio2}
u^{\perp}=\frac{1}{\sqrt{2}}\big(\psi^{\perp}+\ii |D|^{-1}v^{\perp}\big)\,,\qquad 
v_0=\partial_{t}\psi_0\,,
\qquad
v^{\perp}=\partial_{t}\psi^{\perp}\,,
\end{equation} 
solves the system \eqref{sistemasplittato} with initial conditions
\[
(\psi_0(0),v_0(0),u^{\perp}(0))=\big(\breve{\psi}_0,\breve{v}_0, \breve{u}^{\perp}\big)\,,
\]
where
\[
 \breve{\psi}_{0}:=\frac{1}{(2\pi)^d}\int_{\T^{d}}\breve{\psi}(x)dx\,,
 \qquad
  \breve{v}_{0}:=\frac{1}{(2\pi)^d}\int_{\T^{d}}\breve{v}(x)dx\,,
  \qquad
  \breve{u}^{\perp}:=\frac{1}{2}(\breve{\psi}^{\perp}
+\ii |D|^{-1}\breve{v}^{\perp})\,.
\]
We also define
\[
E_{s}(t):=|\psi_0(t)|+|v_0(t)|+\|u^{\perp}(t)\|_{s,R}\,.
\]
Moreover, by  \eqref{campagna1BIS}-\eqref{campagna2BIS} and using 
the equivalence 
\eqref{rmk:equivalenzaTotale2} in Remark
\ref{rmk:equivalenzaTotale}, we deduce that 
there exists a constant $c\geq 1$ such that
\begin{equation}\label{stiminesuuBIS}
\begin{aligned}
\\
E_{s_1}(0)&\leq c\delta\,,
\qquad 
|\breve{\psi}_0|+|\breve{v}_0|
\leq E_{s_1}(0) (2\mathtt{M})^{-s_1}
\leq \delta(2\mathtt{M})^{-s_1}\,,
\\
\sup_{t\in[0,T_{loc}]}E_{s_1}(t)&\leq 2E_{s_1}(0)\leq 2c\delta\,,\qquad R:=2\mathtt{M}\,.
\end{aligned}
\end{equation}
We now show that actually the solution of \eqref{sistemasplittato} 
exists 
over a time interval $[0,T]$ with $T\geq T_{loc}$ 
satisfying the bound in \eqref{patata1dimdMASS}.
We follow the strategy of the proof of Theorem \ref{thm:mainNOBNF}.
Let us define  
  \begin{equation}\label{ginogino2BIS}  
\frac{2^{\frac{1}{2}(s_1-s_0)}}{\mathtt{M}^{qs_0} 2^{q+3}c^{q}}
:=T_{good}\,,
\end{equation} 
and assume that $E_{s_1}(t)$ is well-defined on $[0, \hat{T}]$ for some $\hat{T}\leq T_{good}$ and that
\begin{equation}\label{caffelet1}
\sup_{t\in[0,\hat{T}]} E_{s_1}(t)\leq 2E_{s_1}(0)\leq 2c\delta\,.
\end{equation}
Note that this assumption is not empty in view of the local existence theory.
We now aim to show that assumption \eqref{caffelet1} leads to a contradiction  for $\delta$ small enough.
We reason as in \eqref{caffelet10}.

\noindent
First of all, note that,
 for $\delta>0$ small enough, conditions \eqref{stiminesuuBIS} imply that  the assumption 
\eqref{hyp:localBIS} in Proposition \ref{thm:energyBasicBIS} is satisfied.
Therefore, by estimate \eqref{basicStimaBIS}  in Proposition \ref{thm:energyBasicBIS} 
(with $T\rightsquigarrow \hat{T}$, and 
$R\rightsquigarrow 2\mathtt{M}$)
we deduce
\[
\begin{aligned}
E_{s_1}(t)&\leq E_{s_1}(0)+|\breve{v}_0| \hat{T}
\\&+\frac{\mathtt{M}^{qs_1}}{R^{q(s_1-s_0)}}\int_{0}^{t}\left[ 
(E_{s_1}(\s))^{q+1}
+\int_{0}^{\s}(E_{s_1}(\tau))^{q+1}d\tau
\right]d\s
\\&
\stackrel{\eqref{stiminesuuBIS}, \eqref{caffelet1}}{\leq}
%c\delta
E_{s_1}(0)
\left(1+\hat{T}\frac{(2\mathtt{M})^{-s_1}}{c}+\frac{\mathtt{M}^{qs_0}(2c\delta)^{q} 2}{2^{q(s_1-s_0)}}
\Big(\hat{T} 
+\hat{T}^{2} \Big)\right)\leq \frac{7}{4}E_{s_1}(0)<2 E_{s_1}(0)\,,
%c\delta<2c\delta\,,
\end{aligned}
\]
provided that
\[
\hat{T}\leq\min\Big\{ \frac{c(2 \mathtt{M})^{s_1}}{4}\,,\;
\frac{2^{q(s_1-s_0)}}{\mathtt{M}^{qs_0} 8 (2c\delta)^{q}}\,,
\; \frac{2^{\frac{q}{2}(s_1-s_0)}}{\mathtt{M}^{\frac{q}{2}s_0} \sqrt{8} (2c\delta)^{\frac{q}{2}}}
\Big\}\,.
\]
The latter conditions are verified taking $\hat{T}\leq T_{good}$ with $T_{good}$ as in \eqref{ginogino2BIS}
(recalling that $q\geq 1$, $\delta<1$, $c\geq1$).
By the reasoning above, we deduce  that the supremum over the times $\hat{T}$ such that 
\eqref{caffelet1} holds cannot be smaller than $T_{good}$  in \eqref{ginogino2BIS}.
Hence, a classical bootstrap argument allows to extend the solution of \eqref{sistemasplittato}
in such a way
\begin{equation}\label{ariaaria1}
\sup_{t\in [0,T]} E_{s_1}(t)\leq 2E_{s_1}(0)\,,
%2c \delta\,,
\qquad T\geq T_{good}\,.
\end{equation}

 \medskip
 \noindent
 We now study the evolution of the high norm
$E_{s}(t)$ with $s\geq s_1+1$ along the flow of \eqref{sistemasplittato}.
We shall apply Proposition \ref{thm:energyBIS} with $T\geq T_{good}$ as above.
By \eqref{energyaprioriTRIS}-\eqref{costBrutte2} one gets
\begin{equation}\label{ariaaria2}
E_{s}(t)\leq r(t)+20 \mathtt{K}_2\big(h(t)\big)^{1-\lambda}
\Big(\int_{0}^{T}\big[r(t)+(T\mathtt{K}_2)^{\frac{1}{1-\lambda}}h(t)\big]dt\Big)^{\lambda}\,,
\end{equation}
where $\mathtt{K}_{2},\lambda$ are in \eqref{costBrutte} and 
\[
\begin{aligned}
r(t)&:=E_{s}(0)+|\breve{v}_0|t+ \mathtt{K}_2
\int_{0}^{t}\int_{0}^{\s}(E_{s_1}(\tau))^{q+1}d\tau d\s\,,
\\
h(t)&:=
\int_0^{t}(E_{s_1}(\s))^{\frac{q+1-\lambda}{1-\lambda}} d\s\,.
\end{aligned}
\]
Fist of all, in view of \eqref{stiminesuuBIS} and \eqref{ariaaria1}, 
we have the estimates
\begin{equation*}
\begin{aligned}
\sup_{t\in[0,T]}|r(t)|&\leq E_{s}(0)\big(1+(2\mathtt{M})^{-s_1}T+2\mathtt{K}_2(2c\delta)^{q}T^{2}\big)\,,
\\
\sup_{t\in[0,T]}|h(t)|&\leq E_{s}(0) T(2c\delta)^{\frac{q}{1-\lambda}}2
\end{aligned}
\end{equation*}
Therefore, by \eqref{ariaaria2}, we deduce\footnote{
Here we use  that $(x+y)^{\lambda}\leq (x^{\lambda}+y^{\lambda})$,
$x,y\geq0$ and $0<\lambda<1$\,.},
\begin{equation*}
\begin{aligned}
E_{s}(t)&\leq E_{s}(0)\big(1+(2\mathtt{M})^{-s_1}T+2\mathtt{K}_2(2c\delta)^{q}T^{2}\big)
\\&
+2\mathtt{K}_2T^{1-\lambda}(E_s(0))^{1-\lambda}(2c\delta)^{q}2^{1-\lambda}
 T^{\lambda} (E_s(0))^{\lambda} \big(1+(2\mathtt{M})^{-s_1}T+2\mathtt{K}_2(2c\delta)^{q}T^{2}\big)^{\lambda}
 \\&
+2\mathtt{K}_2T^{1-\lambda}(E_s(0))^{1-\lambda}(2c\delta)^{q}2^{1-\lambda} 
T^{\lambda} (T\mathtt{K}_2)^{\frac{\lambda}{1-\lambda}}
(E_{s}(0))^{\lambda}T^{\lambda}(2c\delta)^{\frac{q\lambda}{1-\lambda}}2^{\lambda}
\\&
\leq E_{s}(0)\big(1+4\mathtt{K}_2(2c\delta)^{q}T\big)\big(1+(2\mathtt{M})^{-s_1}T+2\mathtt{K}_2(2c\delta)^{q}T^{2}\big)
\\&
+4 E_{s}(0)(2c\delta)^{\frac{q}{1-\lambda}}\mathtt{K}_2^{\frac{1}{1-\lambda}} T^{\frac{1}{1-\lambda}+\lambda+1}\,.
\end{aligned}
\end{equation*}
Using the explicit  expression of $\mathtt{K}_2$ in \eqref{costBrutte} and recalling 
\eqref{ginogino2BIS} we deduce that 
\[
\begin{aligned}
&2\mathtt{K}_2(2c\delta)^{q}\leq 4\mathtt{K}_2(2c\delta)^{q}
=\frac{2^{q+2}c^{q}\delta^{q}M^{q(s-s_1-s_0)}}{2^{q(s_1-s_0)}}\leq \frac{\mathtt{a}}{T_{good}}\,,
\quad
(2\mathtt{M})^{-s_1}\leq \frac{\mathtt{a}}{T_{good}}\,,\quad \mathtt{a}:=\mathtt{M}^{q(s-s_1-s_0)}\,.
\end{aligned}
\]
Moreover we note that
\[
T/T_{good}\geq 1\,,\qquad T_{good}\leq M^{q(s-s_1-s_0)}\,.
\]
Therefore, we get
\[
\begin{aligned}
E_{s}(t)&\leq E_{s}(0)\Big(1+2\mathtt{a} \frac{T}{T_{good}}+\mathtt{a}\left(\frac{T}{T_{good}}\right)^2+
\mathtt{a}^2\left(\frac{T}{T_{good}}\right)^2
+\mathtt{a}^2 T_{good} \left(\frac{T}{T_{good}}\right)^3\Big)
\\&
+\mathtt{a}^{s-s_1}  T_{good}^{2-\frac{1}{s-s_1}}
\left(\frac{T}{T_{good}}\right)^{s-s_1+2-\frac{1}{s-s_1}}\,.
\end{aligned}
\]
Taking $\mathtt{M}\gg1$ large enough, we obtain
\[
\begin{aligned}
E_{s}(t)&\leq E_{s}(0)\Big(1+\mathtt{M}^{3q(s-s_1-s_0)(s-s_1)}\left(\frac{T}{T_{good}}\right)^{s-s_1+3-\frac{1}{s-s_1}}\Big)
\\&
\leq
E_{s}(0)\Big(1+ \mathtt{M}^{3qs_0(s-s_1-s_0)^2} \left(\frac{T}{T_{good}}\right)^{s-s_1+3}\Big)
\end{aligned}
\]
which, together with the equivalence \eqref{rmk:equivalenzaTotale2}, 
implies \eqref{stimaIncredibleBisNOBNFMASS}.

\vspace{1em}
\noindent {\bf Acknowledgements.} 
The authors are supported 
by ``GNAMPA - INdAM'', CUP E53C25002010001,
``GNAMPA - INdAM'', CUP E5324001950001.
J.E. Massetti acknowledges also the support of the Department of Excellence
grant MatMod@TOV (2023-27), awarded to the Department of Mathematics at University of Rome
Tor Vergata.

\vspace{1em}
\noindent
\gr{Declarations}. Data sharing not applicable to this article as no datasets were generated or analyzed during the current study.

\noindent
Conflicts of interest: The authors have no conflicts of interest to declare.

\appendix
\section{On some Gr\"onwall type inequalities}

We collect here some classical results about Gr\"owall type inequalities.
%We refer for instance to \cite{MPF1991}.
%
%\begin{lemma}\label{gronclassical}
%Consider   $a,b\in \R$ with $a<b$,  continuous
%functions $f:\R\to\R_{+}:=[0,+\infty)$ and assume that
%$x : [a,b]\to\R_{+}$  is a differentiable function
%satisfying
%\begin{equation*}%\label{drago1}
%x(t)\leq  M+  \int_{a}^{t}f(\s)x(\s) d\s\,,\quad \forall \, t\in[a,b]\,,
%\end{equation*}
%where $M\geq0$.
%Then one has
%\begin{equation*}%\label{drago2}
%x(t)\leq M\exp\big\{\int_{a}^{t}f(\s)d\s\big\}\,,\quad \forall\, t\in[a,b]\,.
%\end{equation*}
%\end{lemma}

\begin{lemma}\label{drago}
Consider  parameters $M\geq0$, $\alpha\in(0,1)$, $a,b\in \R$ with $a<b$.
Let 
%a continuous
%function 
$f:\R\to\R_{+}:=[0,+\infty)$ and 
$x : [a,b]\to\R_{+}$  be continuous functions.

\noindent
$(i)$ If the function $x(t)$ satisfies 
\begin{equation*}
x(t)\leq  M+  \int_{a}^{t}f(\s)x(\s) d\s\,,\quad \forall \, t\in[a,b]\,,
\end{equation*}
then one has
\[
 x(t)\leq M\exp\big\{\int_{a}^{t}f(\s)d\s\big\}\,, \quad \forall \, t\in[a,b]\,.
\]

\noindent
$(ii)$ If the function $x(t)$ satisfies 
\begin{equation*}%\label{drago1}
x(t)\leq  M+  \frac{1}{1-\alpha}\int_{a}^{t}f(\s)(x(\s))^{\alpha}d\s\,,\quad \forall \, t\in[a,b]\,,
\end{equation*}
then one has
\begin{equation*}%\label{drago2}
\big(x(t)\big)^{1-\alpha}\leq M^{1-\alpha}+\int_{a}^{t}f(\s)d\s\,,\quad \forall\, t\in[a,b]\,.
\end{equation*}
\end{lemma}

We also need the following Lemma. For similar estimates
we refer the reader to \cite{MITRINOVIC}  (see also \cite{GAMIDOV}).
\begin{lemma}\label{drago2}
Let $1/2<\lambda<1$ and let $x(t), f(t),g(t)\geq 0$ be continuous functions on $[0,T]$, $T>0$.
If 
\begin{equation}\label{stimaIpo}
x(t)\leq f(t)+\int_{0}^{t}g(\s) (x(\s))^{\lambda}d\s\,,\qquad t\in[0,T]\,,
\end{equation}
then one has 
\begin{equation}\label{stimaTesi}
x(t)\leq f(t)+
20^\lambda\Big(\frac{a}{1-\lambda}+b^{\frac{1}{1-\lambda}}\Big)^{\lambda}
\left(\int_{0}^{t}(g(\s))^{\frac{1}{1-\lambda}}d\s\right)^{1-\lambda}
\end{equation}
where
\begin{equation}\label{def:AeB}
a:=\int_{0}^{T}f(t)dt\,,\qquad b:=\int_{0}^{T}
\left(\int_0^{t}(g(\s))^{\frac{1}{1-\lambda}}
d\s\right)^{1-\lambda} dt\,.
\end{equation}
\end{lemma}
\begin{proof}
By using H\"older inequality and \eqref{stimaIpo}, we note that
\begin{equation}\label{suits3}
\begin{aligned}
x(t)&\leq f(t)
+
\int_{0}^{t}g(\s)(x(\s))^{\lambda}d\s
\\&\leq 
f(t)+\left(\int_0^{t}(g(\s))^{\frac{1}{1-\lambda}}
d\s\right)^{1-\lambda} \left(
\int_{0}^{t}x(\s)d\s
\right)^{\lambda}
\\&\leq 
f(t)+\left(\int_0^{t}(g(\s))^{\frac{1}{1-\lambda}}
d\s\right)^{1-\lambda} \left(
\int_{0}^{T}x(\s)d\s
\right)^{\lambda}\,.
\end{aligned}
\end{equation}
By integrating the inequality above on $[0,T]$, we have
\begin{equation}\label{suits2}
\int_{0}^{T}x(t)dt\leq \int_{0}^{T}f(t)dt+
\left[\int_{0}^{T}
\left(\int_0^{t}(g(\s))^{\frac{1}{1-\lambda}}
d\s\right)^{1-\lambda} dt\right] \left(
\int_{0}^{T}x(t)dt
\right)^{\lambda}\,.
\end{equation}
We introduce the notation
\[
\begin{aligned}
%a&:=\int_{0}^{T}f(t)dt\,,\qquad b:=\int_{0}^{T}
%\left(\int_0^{t}(g(\s))^{\frac{1}{1-\lambda}}
%d\s\right)^{1-\lambda} dt\,,
%\\
\xi&:=\xi(t):=\int_{0}^{t}x(\s)d\s\,,\;\;\;t\in[0,T]\,.
\end{aligned}
\]
Therefore, recalling \eqref{def:AeB}, the inequality 
\eqref{suits2} (which comes from hypothesis \eqref{stimaIpo}) reads
\begin{equation}\label{suits4}
\x(T)\leq a+b(\x(T))^{\lambda}\,.
\end{equation}
Note that the function $[0,T]\ni t\to \x(t)$ is nondecreasing. Moreover, by 
\eqref{suits3}, we deduce
that
\begin{equation}\label{suits444}
x(t)\leq f(t)+\left(\int_{0}^{t}(g(\s))^{\frac{1}{1-\lambda}}d\s\right)^{1-\lambda}\cdot(\x(t))^{\lambda}\,,
\;\;\;t\in[0,T]\,.
\end{equation}
In order to get \eqref{stimaTesi} we shall prove an upper bound 
on $\x(t)$ using the hypothesis \eqref{suits4}.
To do this we introduce the function
\begin{equation}\label{suits44}
z(\x):=\x-a-b\x^{\lambda}\,,\;\;\;\x>0\,,
\end{equation}
and we show that $z(\x)\leq 0$ (recall \eqref{suits4})
for $\x\in[0,\x_{0}]$ for some $\x_0>0$ which we will compute.

Firstly, we note that
\begin{equation}\label{suits5}
\begin{aligned}
&z(0)=z(\x_1)=-a\,,\qquad \x_1:=b^{\frac{1}{1-\lambda}}>0\,,
\\
&z''(\x)=-\lambda(\lambda-1)b \x^{\lambda-2}>0\;\;\; \Rightarrow\;\;\; z\; {\rm is\; convex}\,.
\end{aligned}
\end{equation}
Secondly, we observe that the function $z(\x)$ admits a unique minimum. 
Indeed one has 
$z'(\x)=1-\lambda b\x^{\lambda-1}$, and hence
\begin{equation}\label{suits6}
\bar{\x}:=(b\lambda)^{\frac{1}{1-\lambda}}\,,
\end{equation}
is the unique solution of the equation $z'(\x)=0$. Since $z(\x)$ is convex,  the point $\bar{\x}$
is a minimum.
It is easy to note
that
\begin{equation}\label{suits7}
\begin{aligned}
&\bar{\x}=(b\lambda)^{\frac{1}{1-\lambda}}< b^{\frac{1}{1-\lambda}}=\x_1\,,
\\
z(\bar{\x})&=(b\lambda)^{\frac{1}{1-\lambda}}-a-b(b\lambda)^{\frac{\lambda}{1-\lambda}}
\\&
=-a-(b\lambda)^{\frac{1}{1-\lambda}}\Big(\frac{1}{\lambda}-1\Big)<-a=z(\x_1)
\end{aligned}
\end{equation}
Since $z(\x)\to+\infty$ as $\x\to +\infty$ and it is convex,  there exist a unique point
$\x_0>\x_1$ such that 
$z(\x_0)=0$. Now, since $z$ is convex, $z(0)=-a<0$, and $z(\x_0)=0$, 
it follows that $z(\x)\leq 0$ for any $\x\in[0,\x_0]$.
Moreover, by \eqref{suits4} we have $z(\x(T))\leq 0$. Since $\x_0$ is the unique positive zero of $z$
and $z(\x)>0$ for any $\x>\x_0$, we must have $\x(T)\leq \x_0$.
Recalling that $\x(t)$ is nondecreasing, we conclude that $\x(t)\leq \x(T)\leq \x_0$.
Therefore,
%In view of \eqref{suits4}-\eqref{suits44}, the convexity of $z(t)$, \eqref{suits7} one deduces
%$z(\x)\leq 0$ for $\x\in[0,\x_0]$, so that 
from 
\eqref{suits444} we deduce
\begin{equation}\label{suits55}
x(t)\leq f(t)+\x_0^{\lambda}\left(\int_{0}^{t}(g(\s))^{\frac{1}{1-\lambda}}d\s\right)^{1-\lambda}\,,
\;\;\;\;\;t\in[0,T]\,.
\end{equation}

We shall apply the secant method to provide an upper bound on the zero $\x_0$.
We consider the equation of the secant line passing through
the points $A=(\bar{\x},z(\bar{\x}))$ and $C=(\x_1,z(\x_1))$ has the form
\[
r(\x)=z(\bar{\x})+\frac{z(\x_1)-z(\bar{\x})}{\x_1-\bar{\x}}(\x-\bar{\x})\,,
\]
and intersect the $x$-axis in the point $B=(\x_2,0)$ with
\[
\x_2:=\bar{\x}-\frac{z(\bar{\x})(\x_1-\bar{\x})}{z(\x_1)-z(\bar{\x})}\,.
\] 
Then one must have that $\x_1<\x_0<\x_2$. Recalling \eqref{suits5}, \eqref{suits6} and \eqref{suits7}
we have
\[
\begin{aligned}
\x_1-\bar{\x}&=b^{\frac{1}{1-\lambda}}(1-\lambda^{\frac{1}{1-\lambda}})\,,
\\
z(\x_1)-z(\bar{\x})&=b^{\frac{1}{1-\lambda}}\lambda^{\frac{1}{1-\lambda}}\Big(\frac{1}{\lambda}-1\Big)\,,
\end{aligned}
\]
so that 
\begin{equation}\label{def:xi2}
\begin{aligned}
\x_2&=(b\lambda)^{\frac{1}{1-\lambda}}-\Big[
-a-(b\lambda)^{\frac{1}{1-\lambda}}\Big(\frac{1}{\lambda}-1\Big)
\Big]
\frac{(1-\lambda^{\frac{1}{1-\lambda}})}{\lambda^{\frac{1}{1-\lambda}}\Big(\frac{1}{\lambda}-1\Big)}
\\&=
(b\lambda)^{\frac{1}{1-\lambda}}+\Big[
a+b^{\frac{1}{1-\lambda}}\lambda^{\frac{\lambda}{1-\lambda}}(1-\lambda)
\Big]
\frac{(1-\lambda^{\frac{1}{1-\lambda}})}{\lambda^{\frac{\lambda}{1-\lambda}}(1-\lambda)}\,.
\end{aligned}
\end{equation}
The true zero $\x_0$ of $z(\x_0)=0$ can be found iteratively by setting
\[
\x_0=\lim_{n\to+\infty}\tau_{n}\,,\qquad \tau_{1}=\x_1=b^{\frac{1}{1-\lambda}}\,,
\qquad
\tau_{n+1}=\tau_{n}-\frac{z(\tau_{n})}{z(\x_2)-z(\tau_{n})}(\x_2-\tau_{n})\,.
\]
To provide an upper bound on $\x_0$ we choose to use the very rough estimate
$\x_0<\x_2$.
%\red{We estimate $\x_2$ under the assumption}
%\[
%\frac{1}{2}<\lambda<1\,.
%\]
%\textcolor{blue}{questo si verifica nel senso che 
%$\lambda=1-1/(s-s_1)$, quindi per $s\gg s_1$ va bene e a noi interessa questo.
%diciamo che per sbarazzarsi di un po' di constanti che dipendono da $\lambda$ \`e utile questa ipotesi, perch\`e si hanno problemi per $\lambda \to0^+$ nelle stime, ma tanto non e' quello il regime su cui lavoriamo noi}
%In this case, 
%Recalling that $1/2<\lambda<1$, one can easily check that
%\[
%\begin{aligned}
%\red{e^{-1}<\lambda^{\frac{\lambda}{1-\lambda}}}&<1\,,\qquad 
%\red{\lambda^{\frac{1}{1-\lambda}}<e^{-1}<1\,,vera}
%\qquad 
%1-\lambda<2\,,
%\qquad \quad 
%%\frac{1}{1-\lambda}
%(1-\lambda)^{-1}<2\,.
%\end{aligned}
%\]
One can easily check that
\[
e^{-1}<\lambda^{\frac{\lambda}{1-\lambda}}<1\,,
\qquad 
\lambda^{\frac{1}{1-\lambda}}<e^{-1}<1\,,
\qquad 
\forall\, \lambda\in(1/2,1)\,,
\]
and hence one gets
\[
\frac{(1-\lambda^{\frac{1}{1-\lambda}})}{\lambda^{\frac{\lambda}{1-\lambda}}(1-\lambda)}\leq \frac{2e}{1-\lambda}\,.
\]
Using these bounds in \eqref{def:xi2} one gets
\[
\begin{aligned}
\x_0<\x_2&\leq b^{\frac{1}{1-\lambda}}+
a\frac{2e}{(1-\lambda)}+2eb^{\frac{1}{1-\lambda}}
<20\Big(\frac{a}{1-\lambda}+b^{\frac{1}{1-\lambda}}\Big)\,,
\end{aligned}
\]
which together with \eqref{suits55} implies  the bound \eqref{stimaTesi}.
\end{proof}

\section{Small divisors estimates}
We collect here some useful results which allow us to 
 control the small divisors appearing in the solution of the Homological equation
 in section \ref{sec:homo}.

 \begin{lemma}\label{luchino}
Let $x_1\geq x_2\geq \ldots\geq x_N\geq 2.$ Then
\[
\frac{\sum_{1\leq\ell\leq N} x_\ell}{\prod_{1\leq\ell\leq N} \sqrt{x_\ell}}
\leq 
\sqrt{x_1}+\frac{4}{ \sqrt{x_1}}\,.
\]
\end{lemma}

\begin{proof}
See Lemma $A.4$ in \cite{FeoMass:Beam}.
\end{proof}
In the following it will be convenient to use the following way of reordering of the indexes
$j\in \Z$ appearing in the Hamiltonian \eqref{HamPower}.

\begin{definition}\label{n star}
Consider a vector $v=\pa{v_i}_{i\in \Z}$  $v_i\in \N$, $|v|<\infty$. 
	
\noindent
$(i)$ We denote by $\na=\na(v)$ the vector $\pa{\na_l}_{l\in I}$ 
(where $I\subset \N$ is finite)  
which is the decreasing rearrangement of
\[
\{\N\ni h> 1\;\; \mbox{ repeated}\; v_h + v_{-h}\; \mbox{times} \} 
\cup 
\set{ 1\;\; \mbox{ repeated}\; v_1 + v_{-1} + v_0\; \mbox{times}  }
\]
%\noindent
%$(ii)$ Define the vector $m=m(v)$ as the reordering of the elements of the set
%\[
%\set{j\neq 0 \,,\quad \mbox{repeated}\quad  \abs{u_j} \;\mbox{times}}\,,
%\]
%where $D<\infty$ is its cardinality, such that
%$|m_1|\ge |m_2|\ge \dots\geq |m_D|\ge 1$. 	
\end{definition}

\begin{remark}\label{rmk:emmino}
(i) We  observe that the number $N:=|\al|+|\bt|$ is the cardinality of \,$\na$
and that, 
by momentum conservation, there 
exists a choice of $\s_i = \pm1, 0$ such that 
\begin{equation}\label{pi e cappucci}
\sum_l \sigma_l\na_l=0\,,\qquad 
\end{equation}
with $\sigma_l \neq 0$  if  $\na_l \neq 1$.
Hence, 
\begin{equation}\label{eleganza}
\na_1\le  \sum_{l\ge 2}\na_l\,,
\end{equation}
Indeed, if $\sigma_1 = \pm 1$, 
the inequality follows directly from \eqref{pi e cappucci}; 
if $\sigma_1 = 0$, then $\na_1=1$, hence $\na_l = 1$  $\forall l$. 

\noindent
(ii) Recall Definition \ref{n staremmino}.
 Given $\al\neq\bt\in\N^\Z,$ with $|\al|+|\bt|<\infty$
 we consider $m=m(\al-\bt)$ and $\na=\na(\al+\bt)$.	
If we denote by $D$ the cardinality of $m$ and $N$ the one of $\na$ we have 
\begin{align}
D+\al_0+\bt_0&\le N\,, \label{cappella}
\\
(|m_1|,\dots,|m_D|,\underbrace{1,\;\dots \;,1}_{N-D\;\rm{times}} )\, 
&\leq\,
\pa{\na_1,\dots \na_N}\,.\label{abbacchio}
\end{align}
\end{remark}

We have the following.

\begin{lemma}\label{lem:constance2SE}
For all $(\al,\bt)\in \mathcal{M}$ (see \eqref{mass-momindici}) the following holds.

\noindent
$(i)$ If 
\begin{equation}\label{divisor}
\sum_i (\al_i-\bt_i)|i| \le 10 \sum_i |\al_i-\bt_i| \,,
\end{equation}
then we have (recall Def. \ref{n staremmino})
\begin{equation}\label{constanceSOB}
\prod_{i\neq m_1(\ell),m_2(\ell)}(1+\abs{\al_i-\bt_i}{\jap{i}}) 
\le \prod_{l=3}^N\na_l\,.
%e^{27}N^6\prod_{l=3}^N\na_l^{\tau_0}\,.
\end{equation}
where $\ell=\alpha-\beta$ and 
$N=|\al|+|\bt|$. 

\noindent
$(ii)$ If on the contrary \eqref{divisor} does not hold then
\begin{equation}\label{pool1}
|\omega\cdot(\alpha-\beta)|\geq1\,,
\end{equation}
where $\omega$ is given in \eqref{dispLawWave}.
\end{lemma}

\begin{proof}
\emph{Item (i).} The bound \eqref{constanceSOB} follows trivially recalling 
Definitions \ref{n staremmino}, \ref{n star}
 and Remark \ref{rmk:emmino}-(ii).

\smallskip
\noindent
\emph{Item (ii).} Assume that \eqref{divisor} does not hold. Recalling \eqref{dispLawWave}
 we note that, for $|j|\neq0$
 \begin{equation*}%\label{eq:asy}
 \omega_{j}=\sqrt{|j|^2+\tm}=|j|\sqrt{1+\frac{\tm}{|j|^{2}}}=|j|+\mathtt{r}_{j}(\tm)\,,\qquad |\tr_{j}(\tm)|\leq \frac{\tm}{2|j|}\leq 1\,.
 \end{equation*}
Therefore, using also the triangular inequality, we deduce
\[
\begin{aligned}
\left|\sum_{i\in \mathbb{Z}}(\al_i-\bt_i) \omega_{j}\right|
&\geq 
\left|\sum_{i\in \mathbb{Z}}(\al_i-\bt_i)|i|\right|-
\left|\sum_{i\in \mathbb{Z}}(\al_i-\bt_i) \mathtt{r}_{i}(\tm)\right|
\\&\geq
10 \sum_i |\al_i-\bt_i|- \sum_i |\al_i-\bt_i|\geq1\,.
\end{aligned}
\]
This concludes the proof.
\end{proof}

\begin{lemma}\label{stimaSob:lem}
Fix $\tN\geq1$, $\delta\geq(2^43^4 \tN)^4$, $\tau\leq 36 \tN^2$ and $\td\geq 4\tN$.
Then one has
\begin{equation}\label{seisettedelta}
\mathtt{J}:=\sup_{ j\in\Z,\, (\al,\bt)\in\mathcal{A}} 
\Big(\frac{\jjap{j}^2}{\prod_{i\in\Z}\jjap{i}^{\al_i + \bt_i}} \Big)^\delta 
\prod_{\substack{i\neq m_1(\alpha-\beta)\\ i\neq m_2(\alpha-\beta)}}
\big(1+|\al_i-\bt_i|^2\langle i\rangle^2\big)^{4\tau^2}\leq 
2^{{\delta}-1}6^{\delta}\,,
\end{equation}
where $\mathcal{A}\subseteq \Lambda$
is the set of indexes $(\al,\bt)$ such that \eqref{divisor} holds and
\[
|\al| + |\bt| = \tN + 2\,,\qquad 
\al_j+\bt_j\neq 0 \,,\quad |\al - \bt| \le \tN + 2\,.
\]

\end{lemma}

\begin{proof}
By \eqref{constanceSOB}   
 toghether with $|\al| + |\bt| = \tN + 2$  
and $\td \le 4\tN$ we get
\begin{equation*}
\mathtt{J} \leq 
\sup_{\substack{ %j\in\Z,\,  (\al,\bt)\in\Lambda 
\\ \al_j+\bt_j\neq 0 
\\ |\al - \bt| \le \tN + 2}} 
\Big(\frac{\jjap{j}^2}{\prod_{i\in\Z}\jjap{i}^{\al_i + \bt_i}} \Big)^\delta 
 \pa{\prod_{\ell = 3}^{\tN + 2} \na_\ell^{4\tau^2}}^2,
\end{equation*}
with $\na=\na(\al+\bt)$.
%We claim that
%\begin{equation}\label{filomena}
%\tN + 2 \leq 4 \prod_{l= 3}^{\tN + 2}\jml{\na_l}^{\frac{1}{4\ln 2}}\,.
%\end{equation}
%Indeed if $\tN=0$, the inequality is trivial.
%The case 
%$\tN\geq 1$ follows by 
%Lemma \ref{luchino}.
Now, recalling Def.  \ref{n star} we have
\begin{equation}\label{fiorentina2}
\prod_i\jml{i}^{\al_i+\bt_i}= \prod_{l\ge 1}\jml{\na_l}\,.
\end{equation}
Then
\begin{equation*}
\sup_{\substack {j,\al,\bt
\\ \al_j+\bt_j\geq 1 }} 
\frac{\jml{j}^2}{\prod_i\jml{i}^{\al_i+\bt_i}}
\leq 
\frac{\jml{\hat n_1}^2}{\prod_{l\ge 1}\jml{\na_l}}
=
\frac{\jml{\hat n_1}}{\prod_{l\ge 2}\jml{\na_l}}
\leq
\frac{\sum_{l\ge 2}\jml{\na_l}}{\prod_{l\ge 2}\jml{\na_l}} 
=
\frac{1}{\prod_{l\ge 3}\jml{\na_l}}
+
\frac{\sum_{l\ge 3}\jml{\na_l}}{\prod_{l\ge 2}\jml{\na_l}} \,,
\end{equation*}
where the last inequality holds by momentum conservation.
Recall that $\tau \le 36\tN^2$, $4\tau^2\leq 2^43^4 \tN^4 $ and
 $\delta \ge (2^43^4 \tN)^4 $.  
Then, by the fact that
$(a+b)^{\delta}\leq 2^{{\delta}-1}(a^{\delta}+b^{\delta})$
for $a,b\geq 0,$ ${\delta}\geq 1$, one has
\begin{equation*}
\begin{aligned}
\mathtt{J}&\leq 
2^{{\delta}-1}
\left(
\frac{1}{\prod_{l\ge 3}\jml{\na_l}^{\delta}}
+\frac{(\sum_{l\ge 3}\jml{\na_l})^{\delta}}{\prod_{l\ge 2}\jml{\na_l}^{\delta}} 
\right)
\prod_{l\ge 3}\jml{\na_l}^{{\delta/2}}
\\&\leq  
2^{{\delta}-1}
\left(
1+\frac{(\sum_{l\ge 3}\jml{\na_l})^{\delta}}{
\jml{\na_2}^{\delta}\prod_{l\ge 3}\jml{\na_l}^{{\delta}/2}} 
\right)
\leq  
2^{{\delta}-1}
\left(
1+\frac{(\jml{\na_3}^{1/2}+4)^{\delta}}{\jml{\na_2}^{\delta}} 
\right)\,,
\end{aligned}
\end{equation*}
where we used Lemma \ref{luchino}. Then the thesis follows.
\end{proof}

\section{Birkhoff normal form scheme}\label{sec:birkoff}
In this section we provide the proof of Theorem \ref{mainthm:BNF}.

%\subsection{A normal form step}

\vspace{0.5em}
\noindent
{\bf A normal form step.}
Let $r>r'>0$,  $\zeta,p>0$, 
$\tK\gg1$ and $1\leq \tN\leq \tK-1$.
We consider an Hamiltonian function of the form 
\begin{equation}\label{Hstep}
H = D_\omega + \sum_{\td=1}^{\tN-1}Z^{(\td)}
+  \sum_{\td=\tN}^{\tK}R^{(\td)}+ R^{(\geq \tK+1)}\,,
\end{equation}
where $D_\omega$ is in \eqref{quadraticWave}
and
\[
\begin{aligned}
&Z^{(\td)}\in \cK^\wc_{\ri}(\th_{p})\cap\mathcal{H}^{(\td)}\,,
\qquad 1\leq \td\leq \tN-1\,,
\\&
R^{(\td)}\in \mathcal{H}^\wc_{\ri}(\th_{p})\cap\mathcal{H}^{( \td)}\,,
\qquad\tN\leq \td\leq \tK\,,
\\&
R^{(\geq\tK+1)}\in \mathcal{H}^\wc_{\ri}(\th_{p})\cap\mathcal{H}^{(\geq\tK+1)}\,.
\end{aligned}
\]
In the case $\tN=1$ we assume $Z^{(\td)}\equiv0$.
We set
\begin{equation}\label{def:epsiepsi}
\epsilon_{\td}:=|R^{(\td)}|_{r,p}\,,\;\;\tN\leq \td\leq \tK\,,
\qquad
\epsilon_{\tK+1}:=|R^{(\geq \tK+1)}|_{r,p}\,.
\end{equation}
\begin{lemma}{\bf (Birkhoff normal form step).}\label{dolcenera}
Consider the Hamiltonian $H$ in \eqref{Hstep}
and fix $\omega\in \mathtt{D}_{\gamma}$.
	Assume 
that
\begin{equation}\label{viadelcampo}
 J_0\Big(\sum_{\td=\tN}^{\tK}\epsilon_{\td}+\epsilon_{\tK+1}\Big)
\leq\delta 
\qquad
\text{with}
\quad
\delta:=\frac{\ri-\rf}{16e\ri}
\,,
\end{equation}
where  $J_0=J_0(\zeta,\tN)$ is  in \eqref{controlJ0caseSob}).

\noindent
Then there exists a change of variables
\begin{eqnarray}
	\label{pollon3}
	& \Phi\ :\ B_\rf(\th_{p+\zeta})\ \to\ 
B_{\ri}(\th_{p+\zeta})\,,
	\end{eqnarray}
such that
\begin{equation}\label{HstepPIUUNO}
H\circ\Phi= D_\omega + \sum_{\td=1}^{\tN}Z_{+}^{(\td)}+  \sum_{\td=\tN+1}^{\tK}R_{+}^{(\td)}+ R_{+}^{(\geq \tK+1)}\,,
\end{equation}
where 
\[
\begin{aligned}
&Z_{+}^{(\td)}\in \cK^\wc_{\ri'}(\th_{p+\zeta})\cap\mathcal{H}^{(\td)}\,,
\qquad 1\leq \td\leq\tN\,,
\\&
R_{+}^{(\td)}\in \mathcal{H}^\wc_{\ri'}(\th_{p+\zeta})\cap\mathcal{H}^{( \td)}\,,
\quad \tN+1\leq \td\leq \tK\,,
\\&
R_{+}^{(\geq\tK+1)}\in \mathcal{H}^\wc_{\ri'}(\th_{p+\zeta})\cap\mathcal{H}^{(\geq\tK+1)}\,.
\end{aligned}
\]
%\[
%Z'\in \cK^\wc_{\ri}(\th_{\twi})\cap\mathcal{H}^{(\leq \tN)}\,,\qquad
%R'\in \mathcal{H}^\wc_{\ri}(\th_{\twi})\cap\mathcal{H}^{(> \tN)}\,.
%\]
	Moreover the following estimates hold
	%\footnote{$C$ is defined in \eqref{zucchina}.}
\begin{align}
&Z_{+}^{(\td)}:=Z^{(\td)}\,,\quad 1\leq \td\leq\tN-1\,,\qquad 
|Z_{+}^{(\tN)}|_{\rf,p+\zeta}
\leq
\epsilon_{\tN}
\,,
 %\delta^{-1} J_0^{\star}|R|_{\ri,\twi} (|R|_{\ri, \twi}+ |Z|_{\ri,\twi})\,,
\label{signorinabis}
\\
|R_{+}^{(p)}|_{\rf,p+\zeta}
&\leq 
\epsilon_{p}+
\sum_{\substack{j\geq2  \\ (j-1)\tN+\tN=p}}
\frac{\epsilon_{\tN}}{j!} \left(\frac{\epsilon_{\tN}J_0}{2\delta}\right)^{j-1}
+
\sum_{\substack{1\leq j,\td\leq \tK \\ j\tN+\td=p}}
\frac{1}{j!}\left(\frac{\epsilon_{\tN}J_0}{2\delta}\right)^{j} |Z^{(\td)}|_{r,p}	
\label{patata}
\\&	
\qquad 
+\sum_{\substack{1\leq j\leq \tK \\ \tN\leq\td\leq \tK \\j\tN+\td=p}}
\frac{\epsilon_{\td}}{j!} \left(\frac{\epsilon_{\tN}J_0}{2\delta}\right)^{j}\,,
 \qquad \tN+1\leq \td\leq \tK\,,\nonumber
 \\
 |R_{+}^{(\geq\tK+1)}|_{\rf,p+\zeta}
&\leq 
\epsilon_{\tK+1}+2\sum_{\td=\tN}^{\tK} 
\left(\frac{\epsilon_{\tN}J_0}{2\delta}\right)^{\left[\frac{\tK+1-\td}{\tN}\right]}\epsilon_{\td}
+2\left(\frac{\epsilon_{\tN}J_0}{2\delta}\right)^{\tK+1}(|Z|_{r,p}+\epsilon_{\tN})\,.
\label{signorina}
\end{align}
Finally, for any  $\zeta^{\sharp}\geq 0$, 
assume the further  conditions 
\begin{equation}\label{viadelcampo'}
 \widetilde{J_0}
 \Big(\sum_{\td=\tN}^{\tK}\epsilon_{\td}+\epsilon_{\tK+1}\Big)
\leq\delta\,,
\end{equation}
where  $\widetilde{J_0}=J_0(\zeta+\zeta^{\sharp},\tN)$  in \eqref{controlJ0caseSob}.
Then 
\begin{equation}\label{pollon4}
\begin{aligned}
& \Phi_{\big|B_\rf(\th_{p+\zeta+\zeta^{\sharp}})}\ :\ B_\rf(\th_{p+\zeta+\zeta^{\sharp}})\ \to\ 
B_{\ri}(\th_{p+\zeta+\zeta^{\sharp}})\,,
\\
\sup_{u\in  B_\rf(\th_{p+\zeta+\zeta^{\sharp}})} &\norm{\Phi(u)-u}_{\th_{p+\zeta+\zeta^{\sharp}}}
\le
\ri \widetilde{J_0} |R^{(\tN)}|_{\ri,p}\,.
\end{aligned}
\end{equation}
\end{lemma}

\begin{proof}
Recalling \eqref{Hstep}
we define the new normal form
\begin{equation}\label{omoeq2}
\begin{aligned}
&Z_+:=\sum_{\td=1}^{\tN}Z_{+}^{(\td)}\,,\quad Z_{+}^{(\td)}:=Z^{(\td)}\,,\quad
1\leq\td\leq\tN-1\,,
%\\&
\qquad
Z_{+}^{(\tN)}:=\Pi_{\mathcal{K}}R^{(\tN)}\,.
\end{aligned}
\end{equation}
It is easy to check that $Z_{+}\in \cK^\wc_{\ri}(\th_{p})\cap\mathcal{H}^{(\leq \tN)}$
and satisfies the bound \eqref{signorinabis} by using 
 \eqref{omoeq2}, \eqref{def:epsiepsi}, 
\eqref{fame} and \eqref{bound proiezione}.

Using formula \eqref{def:adjaction} the equation $L_{\omega}S=\{S,D_{\omega}\}=\Pi_{\mathcal{R}}R^{(\tN)}$
admits the unique solution 
\begin{equation}\label{omoeq1}
S:=L_{\omega}^{-1}(\Pi_{\mathcal{R}}R^{(\tN)})\,.
\end{equation}
By Proposition \ref{shulalemma} we have that 
$S\in \mathcal{R}_{r}(\th_{p+\zeta})\cap \mathcal{H}^{(\tN)}$
%(see \eqref{parampara})
and satisfies the estimate
\begin{equation}\label{cavolfiore}
|S|_{r,p+\zeta}\leq J_0 |R^{(\tN)}|_{r,p}
\stackrel{\eqref{def:epsiepsi}}{\leq}J_0\epsilon_{\tN}\,,
\end{equation}
where $J_0=J_0(\zeta,\tN)$ 
  in \eqref{controlJ0caseSob}.
  Note that \eqref{viadelcampo} and \eqref{cavolfiore}
  imply \eqref{stima generatrice} with $\rho=r-\rf$. Hence Lemma 
 \ref{ham flow} applies and we set
$\Phi:=\Phi_S^1$. Recalling Remark \ref{rmk:change} we note that 
the conjugated Hamiltonian has the form
\begin{align*}
%H':=
H\circ\Phi
&= D_\omega + \sum_{\td=1}^{\tN-1}Z^{(\td)} 
+  \Pi_{\mathcal{K}}R^{(\tN)}+\{D_{\omega},S\}+\Pi_{\mathcal{R}}R^{(\tN)}
%\\&
+(e^{L_{S}}-\id -\{\cdot,S\}) D_\omega  
\\&+ (e^{L_{S}}-\id)(\sum_{\td=1}^{\tN-1}Z^{(\td)} +R^{(\tN)})
%\\&
+e^{L_{S}}\big(\sum_{\td=\tN+1}^{\tK}R^{(\td)}+R^{(\geq \tK+1)}\big)
\\&
\stackrel{{\eqref{omoeq1},\eqref{omoeq2}}}{=}
 D_\omega + Z_{+}+R_{+}\,,
\end{align*}
where
\[
\begin{aligned}
R_{+}&:=  - \sum_{j=2}^\infty\frac{\pa{L_{S}}^{j-1}}{j!} \Pi_{\mathcal{R}}R^{(\tN)}
\\&
+ (e^{L_{S}}-\id)(\sum_{\td=1}^{\tN-1}Z^{(\td)}
+\sum_{\td=\tN}^{\tK}R^{(\td)}+R^{(\geq \tK+1)})
+\sum_{\td=\tN+1}^{\tK}R^{(\td)}+R^{(\geq \tK+1)}\,.
\end{aligned}
\]
%Recall that $Z=\sum_{\td=1}^{\tK}Z^{(\td)}+Z^{(\geq\tK+1)}$
%with $Z^{(\td)}\in \mathcal{H}^{(\td)}$, $Z^{(\geq\tK+1)}\in \mathcal{H}^{(\geq\tK+1)}$
%and $|Z^{(\td)}|_{r,\tw}, |Z^{(\geq\tK+1)}|_{r,\tw}\leq |Z|_{r,\tw}$.
Therefore we have
\begin{equation}\label{Rplus}
\begin{aligned}
R_{+}&:=\sum_{p=\tN+1}^{\tK}R_{+}^{(p)}+R_{+}^{(\geq\tK+1)}\,,
\\
R_{+}^{(p)}&:=R^{(p)}+\sum_{\substack{j\geq2  \\ j\tN+\tN=p}}
\frac{\pa{L_{S}}^{j-1}}{j!} \Pi_{\mathcal{R}}R^{(\tN)}
+
\sum_{\substack{1\leq j,\td\leq \tK \\j\tN+\td=p}}
\frac{\pa{L_{S}}^{j}}{j!} Z^{(\td)}
+
\sum_{\substack{1\leq j\leq \tK \\\tN\leq\td\leq \tK\\j\tN+\td=p}}
\frac{\pa{L_{S}}^{j}}{j!} R^{(\td)}
\end{aligned}
\end{equation}
and $R_{+}^{(\geq\tK+1)}$ defined by difference.
Moreover by, the explicit formul\ae\,\eqref{Rplus}, Lemma \ref{ham flow}, bounds
\eqref{cavolfiore}, \eqref{bound proiezione}, 
the smallness assumption
\eqref{viadelcampo}, Remark \ref{rmk:scalaPoi} 
and the monotonicity property (see Lemma \ref{norme proprieta}) 
we get $R_{+}\in \mathcal{R}^\wc_{\ri'}(\th_{\twi'})\cap\mathcal{H}^{(> \tN)}$
which satisfies \eqref{patata}-\eqref{signorina}.

Finally, let us assume \eqref{viadelcampo'}.
 By Proposition \ref{shulalemma}
 let $S^\sharp= L_\omega ^{-1} \Pi_{\mathcal{R}}R^{(\tN)}$ in 
 $\cR_{r}(\th_{p+\zeta+\zeta^\sharp})$ be the 
  solution of the homological equation 
 $L_\omega S^\sharp = \Pi_{\mathcal{R}}R^{(\tN)}$ on 
  $B_{r}(\th_{p+\zeta+\zeta^\sharp})\subseteq B_{r}(\th_{p+\zeta})$ for any $\zeta^\sharp \geq 0$.
 Since $S$ and $S^\sharp$ solve the same linear
 equation on 
 $B_{r}(\th_{p+\zeta+\zeta^\sharp})$, we have that
 $$
 S^\sharp=S_{\big| B_{r}(\th_{p+\zeta+\zeta^\sharp})}\,.
 $$ 
By Proposition \ref{shulalemma} we get
\begin{equation}\label{cavolfiore'}
\abs{S}_{r,p+\zeta+\zeta^\sharp}
\leq 
\widetilde{J_0} | R^{(\tN)}|_{r,p}\,.	
\end{equation} 
We now apply Lemma \ref{ham flow} with
$(r,p)\rightsquigarrow(r,p+\zeta+\zeta^\sharp)$ 
and $\rho:=r-\rf.$
Note that \eqref{viadelcampo'} and \eqref{cavolfiore'}
imply \eqref{stima generatrice}.
Then \eqref{pollon4}
follows by \eqref{pollon} and \eqref{cavolfiore'}.
\end{proof}

%\subsection{The iterative scheme}
\vspace{0.5em}
\noindent
{\bf The iterative scheme.}
The proof of Theorem \ref{mainthm:BNF} is based on the following 
 iterative scheme.

\vspace{0.3em}
\noindent
\textbf{Setting of parameters.} For any $0\le k \le \tK$, let us recursively define: 
\begin{equation}\label{parametri}
\begin{aligned}
&r_k= r_0(1- \frac{k}{2\tK})\,,
\qquad  
\delta_k = \frac{r_{k} - r_{k+1}}{16 e r_{k}}\,,
\qquad 
\zeta_k:= (2^43^4 k)^4\,,
\qquad \s_{k}:=\sum_{i=1}^k\zeta_{i}\,.
\end{aligned}
\end{equation}

Moreover, let us define 
\begin{equation}\label{patata2}
\eps:= \pa{\frac{r_0}{\bar{r}}}\,,
\qquad
\mathtt{R}_0:=|R_0|_{\bar{r},p}\,.
\end{equation}
By Lemma \ref{gasteropode}
and Remark \ref{rescalingHaminiziale} we have that
\begin{equation}\label{scala0}
\abs{R_0^{(\td)}}^{\wc}_{r_0,p}\leq \e^{\td} \mathtt{R}_0\,,
\qquad  1\leq \td\leq \tK\,,
\qquad
\abs{R_0^{(\geq\tK+1)}}^{\wc}_{r_0,p}\leq  \e^{\tK+1} \mathtt{R}_0\,.
\end{equation}
Let us also introduce %(recall \eqref{controlJ0}, \eqref{controlJ0caseSob})
\begin{equation}\label{arancia1}
 J_k:=J_0(p+\s_k, k))\qquad 0\le k \le \tK\,,
\end{equation}
where  $J_0$ is given  in  \eqref{controlJ0caseSob}.
In the following we assume taht
\begin{equation}\label{cond:small}
\mathtt{R}_0{4^{\tK+3}} J_{\tK} \,\eps \leq \delta_0\,.
\end{equation}
We have the following result whose proof is nowadays classical, so we omit it and we refer the reader for instance to
Lemma $6.3$ in
\cite{FeoMass:Beam}.

\begin{lemma}{\bf (Iteration lemma).}\label{lem:iterazione}
The following holds true for any $0\leq k\leq \tK$:

\vspace{0.3em}
\noindent
${\bf (S1)}_k$ there are Hamiltonians $H_{k}$ of the form
\begin{equation}\label{hamk}
\begin{aligned}
&H_k= D_\omega +Z_{k} +R_{k}\,,
\\&
Z_{k}:= \sum_{\substack{1\leq\td\leq k\\ \td\; {\rm even} }}Z_{k}^{(\td)}\,,\qquad
 R_{k}:= \sum_{\td= k+1 }^{\tK}R_{k}^{(\td)}+
R_{k}^{(\geq\tK+1)}\,, 
\end{aligned}
\end{equation}
where $Z_0\equiv 0$, and 
\begin{equation}\label{hamk2}
\begin{aligned}
 Z^{(\td)}_{k}&\in \cK^\wc_{\ri_k}(\th_{p+\s_{k}})\cap\mathcal{H}^{(\td)}\,,
\qquad 1\leq \td\leq k
\\
R_k^{(\td)}&\in \mathcal{H}^\wc_{\ri_k}(\th_{p+\s_{k}})\cap\mathcal{H}^{( \td)}\,,
\qquad k+1\leq \td\leq \tK\,,
\\
R_k^{(\geq\tK+1)}&\in \mathcal{H}^\wc_{\ri_k}(\th_{p+\s_{k}})\cap\mathcal{H}^{( \geq\tK+1)}\,;
\end{aligned}
\end{equation}

\noindent
${\bf (S2)}_k$ one has, for $1\leq k\leq \tK$,
\begin{equation}\label{smallcondk}
J_k
\left(
\sum_{\td=k}^{\tK}|R_{k-1}^{(\td)}|_{r_{k-1},p+\s_{k-1}}+
|R_{k-1}^{(\geq\tK+1)}|_{r_{k-1},p+\s_{k-1}}
\right)
\le \delta_{k-1}\,;
\end{equation}

\noindent
${\bf (S3)}_k$ one has, for $1\leq k\leq \tK$,
\begin{align}
\norm{Z_k^{(\td)}}_{r_k,p+\s_{k}} &\leq \e^{\td}
\mathtt{R}_0(4^k J_{\tK}\delta_0^{-1})^{\td-1}
2^{\td-1}\,,
\qquad 1\leq \td\leq k\,,
\label{smallnormk1}
%\tag{$S_1$-$k$}
\\
\norm{R_{k}^{(\td)}}_{r_k,p+\s_{k}} &\le \e^{\td}
\mathtt{R}_0({4^k} J_{\tK}\delta_0^{-1})^{\td-1}
{2^{k-1}}\,,
\quad
k+1\leq \td\leq \tK\,,
\label{smallnormk2}
%\tag{$S_{2}$-$k$}
\\
\norm{R_{k}^{(\geq\tK+1)}}_{r_k,p+\s_{k}} &\le \e^{\tK+1}\mathtt{R}_0
(4^k J_{\tK}\delta_0^{-1})^{\tK}
2^{k}\,;
\label{smallnormk3}
%\tag{$S_{3}$-$k$}
\end{align}

\noindent
${\bf (S4)}_k$  for any $1\leq k\leq \tK$ there are maps
$\Phi_k \; : \; B_{r_k}(\th_{\tw_k})\ \to\ 
B_{r_{k-1}}(\th_{\tw_k})$
%\begin{equation}\label{sepultura}
%\Phi_k \; : \; B_{r_k}(\th_{\tw_k})\ \to\ 
%B_{r_{k-1}}(\th_{\tw_k})
%\end{equation}
such that
\begin{equation}\label{hamkkk}
H_k=H_{k-1}\circ \Phi_k\,.
\end{equation}
Moreover, for any $k\leq n\leq \tK$,  
 one has
\begin{equation}\label{sepultura2}
\Phi_k \; : \; B_{r_k}(\th_{p+\s_{n}})\ \to\ 
B_{r_{k-1}}(\th_{p+\s_{n}})
\end{equation}
with 
\begin{equation}\label{stimaPhiK}
\sup_{u\in  B_{r_k}(\th_{p+\s_{n}})} \norm{\Phi_k(u)-u}_{{p+\s_{n}}}
\le 
r_{k-1}\mathtt{R}_0\frac{1}{2^k} J_{\mathtt{K}}\e\,.
%({4^k} J_{\tK}\delta_0^{-1})^{k-1} 2^{k-1}
%.....
  % \ri \widetilde{J_0}^{\star}\abs{R^{(\tN)}}_{\ri,\twi}\,.
\end{equation}
\end{lemma}

%%%%%%%%%%%%%%%%%%%%%%%%%%%%%
%%%%%%%%%%%%%%%%%%%%%%%%%%%%%
%%%%%%%BIBLIOGRAPHY%%%%%%%%%%%%%%
%%%%%%%%%%%%%%%%%%%%%%%%%%%%%
%%%%%%%%%%%%%%%%%%%%%%%%%%%%%
%%%%%%%%%%%%%%%%%%%%%%%%%%%%%

\end{document}